\documentclass[11pt,reqno]{amsart}

\pdfoutput=1

\usepackage{geometry}
\usepackage{xcolor}
\usepackage{amsthm}
\usepackage{amssymb}
\usepackage{hyperref}
\usepackage{arydshln}
\usepackage{bbm}
\usepackage{mathrsfs}
\usepackage{graphicx}

\usepackage[T1]{fontenc} 
\usepackage[utf8]{inputenc}
\usepackage[english]{babel}

\usepackage{float}
\usepackage{multirow}
\usepackage{caption}

\makeatletter
\numberwithin{equation}{section}
\numberwithin{figure}{section}

\newcommand{\C}{\mathbb{C}}
\newcommand{\R}{\mathbb{R}}
\newcommand{\N}{\mathbb{N}}

\newcommand{\F}{\mathcal{F}}

\renewcommand{\P}{\mathbb{P}}
\newcommand{\E}{\mathbb{E}}
\newcommand{\e}{\varepsilon}

\newcommand{\1}{\mathbbm{1}}

\newcommand{\bfrac}[2]{\genfrac{}{}{0pt}{}{#1}{#2}}
\newcommand{\Lu}{ \overset{\Theta}{\Longrightarrow}}

\newtheorem{Theorem}{Theorem}[section]
\newtheorem{Proposition}[Theorem]{Proposition}\newtheorem{Corollary}[Theorem]{Corollary}\newtheorem{Lemma}[Theorem]{Lemma}\newtheorem{Remark}[Theorem]{Remark}\newtheorem{Definition}[Theorem]{Definition}\newtheorem{Example}[Theorem]{Example}
\newtheorem{assumption}[Theorem]{Assumption}

\numberwithin{equation}{section}

\makeatother

\begin{document}
\title[MLE in the ergodic Volterra Ornstein-Uhlenbeck process]{Maximum Likelihood estimation in the ergodic Volterra Ornstein-Uhlenbeck process}
\author{Mohamed Ben Alaya}
\author{Martin Friesen}
\author{Jonas Kremer$^1$}
\thanks{$^1$The views, opinions, positions or strategies expressed in this article are those of the authors and do not necessarily
represent the views, opinions, positions or strategies of, and should not be attributed to E.ON Energy Markets.}

\address[Mohamed Ben Alaya]{Laboratoire de Mathématiques Raphaël Salem (LMRS)
\\ Université de Rouen Normandie}
\email{mohamed.ben-alaya@univ-rouen.fr}

\address[Martin Friesen]{School of Mathematical Sciences\\
Dublin City University\\ Glasnevin, Dublin 9, Ireland}
\email{martin.friesen@dcu.ie}

\address[Jonas Kremer]{modelling \& Quant. Analytics \\ E.ON Energy Markets \\ 45131 Essen, Germany}
\email{jonas.kremer@eon.com}

\date{\today}



\begin{abstract}
We investigate maximum likelihood estimation for the drift parameters of stochastic Volterra processes in the ergodic regime. In our first result, we establish the equivalence of laws under general changes of drift and provide the corresponding Radon-Nikodym derivative. This allows us to develop a rigorous maximum likelihood estimation framework. As an application, we study the Volterra Ornstein–Uhlenbeck process in the ergodic regime, considering both continuous-time and high-frequency discrete-time observations. In both regimes, we prove the consistency and asymptotic normality of the maximum likelihood estimators. A key intermediate result, which may be of independent interest, is a uniform Birkhoff-type theorem under an asymptotic independence condition. This theorem yields a locally uniform Law of Large Numbers over the parameter space.
\end{abstract}

\maketitle

\allowdisplaybreaks

\section{Introduction}

\subsection{Overview}
Stochastic Volterra processes have gained increased attention, e.g. due to their ability to capture the rough behaviour of sample paths (for an overview of the literature see e.g. \cite{math11194201}), and their flexibility to describe short and long-range dependencies \cite{MR3561100}. In the absence of jumps, the general convolution-type stochastic Volterra equation on $\R^d$ takes the form
\begin{align}\label{eq: VSDE}
 X_t = x_0 + \int_0^t K(t-s)b(X_s)ds + \int_0^t K(t-s)\sigma(X_s)dB_s
\end{align}
where $b: \R^d \longrightarrow \R^d$ denotes the drift, $\sigma: \R^d \longrightarrow \R^{m \times d}$ the diffusion matrix, $B$ a standard Brownian motion on $\R^m$, and $K \in L_{loc}^2(\R_+; \R^{d \times d})$ the Volterra kernel. The fractional Riemann-Liouville kernel 
\begin{align}\label{eq: fractional kernel}
 K_{\alpha}(t) = \frac{t^{\alpha - 1}}{\Gamma(\alpha)}, \qquad \alpha \in (1/2, 1]
\end{align}
constitutes the most prominent example of Volterra kernels that allow for flexible incorporation of rough sample path behaviour. Other examples of Volterra kernels covered by this work are, e.g., the exponentially damped fractional kernel $K(t) = \frac{t^{\alpha-1}}{\Gamma(\alpha)}e^{-\lambda t}$, $K(t) = \sum_{j=1}^N c_j e^{-\lambda_j t}$ used for Markovian approximations \cite{MR3934104, MR4521278}, and the fractional $\log$-kernel $K(t) = \log(1+t^{-\alpha})$ with $\alpha \in (0,1]$. For multi-dimensional models, one often considers a diagonal matrix of Volterra kernels given as above, i.e. $K(t) = \mathrm{diag}(K_1(t), \dots, K_d(t))$, where each $K_j$ is given by one of the previous examples.

Suppose that the drift depends on an unknown parameter $\vartheta \in \Theta$, i.e., $b = b_\vartheta$. In applications, e.g. risk management in financial markets, it is often essential to estimate the parameter $\vartheta$ from historical data. While the calibration of stochastic Volterra models to option prices has been studied extensively in the literature \cite{MR3494612, MR4412586, MR3905737, MR4091168}, statistical inference based on time series data has received far less attention and constitutes the main focus of this work. For Markovian diffusion processes\footnote{that is $K(t) \equiv 1$}, maximum-likelihood estimation (MLE) allows us to estimate the drift parameters $\vartheta$ under continuous observations, see e.g. \cite{MR2144185}. However, when dealing with Volterra processes of the form \eqref{eq: VSDE}, it is often the case that $X$ is not a semimartingale and hence we cannot directly apply the Girsanov transformation to prove the equivalence of measures, construct the likelihood ratio, and study asymptotic properties of the MLE. Furthermore, due to the absence of the Markov property, which typically applies to solutions of \eqref{eq: VSDE}, we cannot rely on existing methods to obtain long-term limit theorems, which are crucial for investigating the convergence of parameter estimators in the ergodic regime.

Statistical inference for stochastic equations driven by fractional Brownian motion (fBm) is currently an active area of research, see \cite{MR4023508, MR4068875}, and also \cite{MR2638974, MR3918739, MR4265053} for the specific case of the fractional Ornstein--Uhlenbeck process. Although models based on fBm and those based on Volterra equations may appear structurally similar, key differences prevent a direct application of results from the fBm literature. Most notably, Volterra processes are defined via convolutions against standard Itô integrals, in contrast to fBm models, which require a more delicate integration theory. Finally, in our setting, the Volterra kernel appears in both the drift and the noise terms, leading to qualitatively different ergodic properties, see \cite{BBF23}. Thus, the methods of this work complement the literature towards statistical inference for stochastic Volterra equations \eqref{eq: VSDE} applicable to a general class of completely monotone Volterra kernels $K$.

\subsection{Equivalence of laws for stochastic Volterra equations}

Let $b, \widetilde{b}: \R^d \longrightarrow \R^d$ denote the drift functions, $\sigma: \R^d \longrightarrow \R^{m \times k}$ the diffusion matrix. Denote by $X$ and $\widetilde{X}$ the processes given by \eqref{eq: VSDE} with drift $b$ and $\widetilde{b}$, respectively. Let $\P^X, \P^{\widetilde{X}}$ denote their law on the canonical path space $C(\R_+; \R^{d})$. In Theorem \ref{thm: equivalence of laws} and Corollary \ref{cor: equivalence laws lipschitz}, we provide sufficient conditions for the equivalence of the laws $\P^X|_{\F_T}, \P^{\widetilde{X}}|_{\F_T}$ with $T > 0$ fixed, and compute the Radon-Nikodym derivative on the path space given by
\begin{align}\label{eq: 4}
    \log\ \frac{d\P^X}{d\P^{\widetilde{X}}}\bigg|_{\F_T} = \ell_T(\widetilde{X})
\end{align}
where $\ell_T$ is a path-dependent functional that resembles the same characteristics as the classical Girsanov transformation. Below, we briefly outline how the absence of the semimartingale and Markov property for stochastic Volterra processes is circumvented and \eqref{eq: 4} obtained.

Firstly, let us suppose that the Volterra kernel admits a resolvent of the first kind denoted by $L$, i.e., a locally finite measure such that $L \ast K = K \ast L = 1$. Remark that, if $K$ is completely monotone such that $K(t_0) > 0$ for some $t_0 > 0$, then it admits a resolvent of the first kind. For additional details we refer to Section 2.1. Define the following functional on the path space of continuous functions 
\begin{align}\label{eq: Z process}
 Z_t(x) := \int_{[0,t]}L(ds)(x_{t-s} - x_0), \qquad x \in C(\R_+; \R^d).
\end{align}
Inserting $x = X$, in Section 2 we show that $(Z_t(X))_{t \geq 0}$ is a semimartingale with differential $dZ_t(X) = b(X_t)dt + \sigma(X_t)dB_t$. Remark that transformation \eqref{eq: Z process} already appears in \cite{MR4019885} to study, e.g., the convergence of martingale problems, and was first used in \cite{Z21} to obtain consistency of the MLE for the special case of the Volterra Ornstein-Uhlenbeck process in the non-ergodic regime $(b = 0$ and $\beta > 0$ below). In this work, we show that the inverse transformation \eqref{eq: Z process} and its inverse play a crucial role in deriving \eqref{eq: 4} for a general class of stochastic Volterra processes.

Let $\P_{Z}, \P_{\widetilde{Z}}$ be the law of $Z = Z(X)$, and $\widetilde{Z} = Z(\widetilde{X})$ respectively. Using the Girsanov transformation for semimartingales, we may study the equivalence of laws $\P^Z|_{\F_T} \sim \P^{\widetilde{Z}}|_{\F_T}$, provided that the Novikov condition is satisfied. Unfortunately, the latter is often difficult to verify. To obtain the desired equivalence under weaker assumptions, we identify $Z(X), Z(\widetilde{X})$ as solutions to path-dependent diffusion equations for which the Novikov condition may be relaxed, see \cite{MR1800857}. For this purpose, under some natural regularity conditions on $K$, in Lemma \ref{lemma: fractional differentiation}, we prove that the path-dependent transformation 
\begin{align}\label{eq: Gamma}
     \Gamma_t(z) = K(t)z_t + \int_0^t K'(t-s)(z_s - z_t)ds
\end{align}
is the inverse for \eqref{eq: Z process} in the sense that $\Gamma(Z(x)) = x - x_0$ and $Z(\Gamma(z)) = z$ holds for suitable paths $x,z$. For singular kernels $K$ the inverse transformation $\Gamma$ cannot be defined on all continuous paths, but only on the smaller space $C_0^{\eta}([0, T]; \R^d)$ of $\eta$-H\"older continuous paths that vanish at zero with $\eta + \alpha > 1$ where $\alpha$ is determined by the order of the singularity of $K$ and its derivative $K'$, see \eqref{eq: K asymptotics} for the precise condition. Note that $\Gamma$ describes an abstract differentiation operator while $Z$ is the corresponding abstract integration operator. For the case of fractional kernels \eqref{eq: fractional kernel}, the latter coincide with the well-known formulas for fractional differentiation and integration.

Using this inverse transformation, we verify that $Z_t = Z_t(X)$ is a solution of the path-dependent stochastic equation
\begin{align}\label{eq: path dependent}
    dZ_t = b_{\vartheta}\left( x_0 + \Gamma_t(Z)\right)dt + \sigma\left( x_0 + \Gamma_t(Z)\right)dB_t.
\end{align}
Hence, the non-markovian nature of $X$ transforms into path-dependence in \eqref{eq: path dependent}. For such path-dependent equations, we prove $\P^Z|_{\F_T} \sim \P^{\widetilde{Z}}|_{\F_T}$ and compute the Radon-Nikodym derivative on $\F_T$, see Theorem \ref{thm: equivalence of laws}. Finally, by another application of the transformation $x \longmapsto Z(x)$, we deduce in Corollary \ref{cor: equivalence laws lipschitz} property \eqref{eq: 4} for stochastic Volterra processes.

Let us conclude that the results developed in Section \ref{sec:2_equivalence_of_measures} are based on spaces of H\"older continuous paths merely for the sake of convenience. Indeed, it should be possible to modify Lemma \ref{lemma: fractional differentiation} to cover also cases where $X$ is a stochastic Volterra process with jumps. Theorem \ref{thm: equivalence of laws} and Corollary \ref{cor: equivalence laws lipschitz} would then remain valid as long as we may apply a Girsanov transformation for path-dependent stochastic equations beyond diffusions as given, e.g., by the Esscher transformation. Details of such a modification, however, go beyond the scope of this work and are therefore left for future research.

\subsection{Application to Maximum Likelihood Estimation}

As an application, we may now study MLE of the drift parameters. Namely, suppose that the drift is given in parametric form $b = b_{\vartheta}$ with $\vartheta \in \Theta$. Assuming that $\sigma, K$ are known, we aim to estimate the parameters $\vartheta$ from observations of $X$. Let us denote by $\P_{\vartheta}$ the law of $X$ with parameter $\vartheta$, and by $\P_{\vartheta_0}$ the law of $\widetilde{X}$ with some auxiliary parameter $\vartheta_0$. Following the general theory on MLE, the corresponding estimator $\widehat{\vartheta}_T$ for the parameters $\vartheta \in \Theta$ is defined by the maximum of the $\log$-likelihood ratio $\ell_T(\vartheta, \vartheta_0; X)$, i.e.
\begin{align}\label{eq: minimum MLE}
    \begin{cases} \widehat{\vartheta}_T = \mathrm{argmax}_{\vartheta \in \Theta} \ \ell_T(\vartheta, \vartheta_0; \widetilde{X}) 
    \\ \ell_T(\vartheta, \vartheta_0; \widetilde{X}) = \log \ \frac{d\P_{\vartheta}}{d\P_{\vartheta_0} }\bigg|_{\F_T}(\widetilde{X}).
    \end{cases}
\end{align}
Remark that $\vartheta_0$ is an auxiliary value that can be chosen arbitrarily. In case of several minima, one may pick any of such. For specific models, instead of solving the minimisation problem \eqref{eq: minimum MLE}, it is often feasible to study the first-order conditions $\nabla_{\vartheta} \ell_T(\vartheta, \vartheta_0; X) = 0$, which may be solved explicitly.

As an application, we carry out all computations for the Volterra Ornstein-Uhlenbeck process (VOU process) on $\R^d$ given as the unique strong solution of the stochastic Volterra equation
\begin{align}\label{eq: VOU}
    X_t = x_0 + \int_0^t K(t-s)\left( b + \beta X_s \right)ds + \int_0^t K(t-s) \sigma dB_s.
\end{align}
Here $B$ denotes a standard Brownian motion on $\R^m$, $\sigma \in \R^{m \times d}$, $b \in \R^d$, $\beta \in \R^{d \times d}$, and $K \in L_{loc}^2(\R_+; \R^{d \times d})$. Since the coefficients are globally Lipschitz continuous, the existence and uniqueness of a strong solution with sample paths in $L^2_{loc}(\R_+; \R^d)$ can be obtained from a standard fixed-point procedure. Under mild additional regularity conditions, it can be shown that $X$ has a modification with (H\"older) continuous sample paths, see \cite[Lemma 2.4, Lemma 3.1]{MR4019885}. Such processes, including their Markovian counterparts with $K \equiv 1$, are widely used in Physics (see, e.g., \cite{MR3834854} for the case of the Langevin equation and models with a spatial structure), appear in interest rate modelling as the Vasicek model, arise as a model for rough volatility \cite{MR3805308}, but also may serve as building blocks for electricity spot-price models \cite{BENNEDSEN2017301}.

In Theorem \ref{lemma: VOU Girsanov}, we derive an explicit form of the maximum likelihood estimator (MLE). Due to the specific structure of $\ell_T$ given in \eqref{eq: 4} (see Corollary \ref{cor: equivalence laws lipschitz}), this MLE takes the same form as the classical one studied in \cite{MR2471289}, with the distinction that stochastic integrals of the form $\int_0^T Y_s dX_s$ are replaced by $\int_0^T Y_s dZ_s(X)$. To analyze its asymptotic properties, we focus on the ergodic regime, which allows us to exploit the Law of Large Numbers and martingale limit theorems.

Accordingly, in Section 3, we first present in Theorem \ref{lemma: Birkhoff ergodic theorem} a general version of the Law of Large Numbers for stochastic processes (not necessarily Markovian) that satisfy the asymptotic independence condition
\begin{align}\label{eq: asymptotic independence}
 (X_t, X_s) \Longrightarrow \pi \otimes \pi \text{ as } s,t \to \infty \text{ such that } |t-s| \to \infty,
\end{align}
where $\pi$ denotes the limiting distribution of $X$ in the sense that $X_t \Longrightarrow \pi$. Subsequently, we verify in Lemma \ref{lemma: VOU asymptotic independence} the asymptotic independence property, and prove in Theorem \ref{thm: law of large numbers VOU} a Law of Large Numbers uniformly in the model parameters. As a by-product, we show that the stationary VOU process is mixing. As a further application, we prove that the method of moments is consistent under continuous, high-frequency, and low-frequency observations (see Corollary \ref{cor: VOU method of moments}). Finally, in Section 5, we turn to the convergence of the MLE for continuous observations (see Theorem \ref{thm: continuous MLE for VOU}) and, through numerical discretisation, also for discrete high-frequency observations (see Theorem \ref{thm: discrete MLE VOU}). In both settings, we prove that the derived estimators are consistent and asymptotically normal, locally uniformly in the model parameters. Remark that, while the H\"older continuity of sample paths for $X$ is convenient to apply the results of Section 2, subsequent results do not \textit{explicitly} use the latter until Section 5.2, where convergence of discrete high-frequency observations is established.

\subsection{Structure of the work}

In Section 2 we state and prove our main results on the construction of the inverse mapping $\Gamma$ given by \eqref{eq: Gamma}, and a proof for the equivalence of laws $\P^X|_{\F_T} \sim \P^{\widetilde{X}}|_{\F_T}$. We conclude Section 2 with the derivation of the MLE for the Volterra Ornstein-Uhlenbeck process. In Section 3, we prove a general result on the Law of Large Numbers under the asymptotic independence assumption, while Section 4 provides the application thereof to the VOU process. Finally, asymptotic properties of the MLE in continuous and discrete time are studied in Section 5, while particular examples of Volterra kernels are considered in Section 6. In Section 7, we provide a numerical illustration of our results, while some technical results on weak convergence uniformly in a family of parameters are collected in the appendix.

\section{Equivalence of Measures}\label{sec:2_equivalence_of_measures}

\subsection{Completely montone Volterra kernels}

Below, we collect some notation and results on completely monotone Volterra kernels and their resolvents of the first kind. Firstly, we say that $A \in \R^{d \times d}$ satisfies $A \geq 0$ iff $v^{\top}Av \geq 0$ holds for all $v \in \C^d$. Likewise, $A > 0$ iff $v^{\top}A v > 0$ for all $v \in \C^d \backslash \{0\}$. A Volterra kernel $K \in L_{loc}^1(\R_+; \R^{d \times d})$ is called completely monotone, if it is $C^{\infty}$ on $(0,\infty)$ and satisfies
\[
    (-1)^n K^{(n)}(t) \geq 0, \qquad \forall t > 0, \ n \in \N_0.
\]
For a given matrix-valued measure $F_0$ on $\R_+$ of locally bounded variation, and another matrix-valued function $F_1$ of locally bounded variation, we define the convolutions $F_1 \ast F_0$ and $F_0 \ast F_1$ by
\[
    (F_1 \ast F_0)(t) = \int_{[0,t]} F_1(t-s)F_0(ds)\ \text{ and } \ (F_0 \ast F_1)(t) = \int_{[0,t]}F_0(ds)F_1(t-s)
\]
when $t > 0$, provided that the dimensions match and the corresponding integral is absolutely convergent. These definitions are extended to $t = 0$ by right-continuity, provided that the limit exists. We refer to \cite{MR1050319} for additional details on completely monotone functions and convolutions.

A resolvent of the first kind for $K$ is the unique measure $L$ on $\R_+$ with values in $\R^{d \times d}$ of locally bounded variation satisfying $(K \ast L)(t) = (L \ast K)(t) = 1$ for all $t > 0$. Remark that, if $K$ is completely monotone such that $K(t_0) > 0$ for some $t_0 > 0$, the existence of such a resolvent is guaranteed by \cite[Chapter 5, Theorem 5.4]{MR1050319}. In particular, this resolvent is of the form
\begin{align}\label{eq: resolvent of the first kind}
    L(ds) = A\delta_0(ds) + L_0(s)ds
\end{align}
where $A \in \R^{d \times d}$ and $L_0 \in L_{loc}^1(\R_+; \R^{d \times d})$ is completely monotone. Moreover, $A$ is invertible iff $\limsup_{t \to 0}|K(t)| = + \infty$, and $A = 0$ iff $\lim_{t \to 0}v^{\top}K(t)v = +\infty$ for all $v \in \C^d \backslash \{0\}$. Finally, note that in dimension $d = 1$, one has $A = K(0^+)^{-1}$ with the convention $1/+\infty = 0$.

\subsection{Generalized fractional differentiation}

Let $C([0,T]; \R^d)$ denote the space of continuous functions on $[0,T]$, and $C^{\eta}([0,T]; \R^d)$ denote its subspace of $\eta$-H\"older continuous functions with $\eta \in [0,1]$, where $C^{0}([0,T]; \R^d) = C([0,T]; \R^d)$. Finally, let us denote by $C_0([0,T]; \R^d)$ the subspace of continuous functions with $x_0 = 0$, and $C_0^{\eta}([0,T]; \R^d)$ denotes the subspace of $\eta$-H\"older continuous functions that vanish in zero. 

Suppose that $K \in L_{loc}^1(\R_+; \R^{d \times d})$ is completely monotone such that $K(t_0) > 0$ for some $t_0 > 0$. Let $L$ be its resolvent of the first kind, i.e. $K \ast L = L \ast K = 1$. Fix $T > 0$. Below we study the operation $Z$ given by \eqref{eq: Z process}, and construct, in particular, its inverse transformation $\Gamma$. Firstly, recall that $Z$ is defined by 
\[
    Z_t(x) = \int_{[0,t]}L(ds)(x_{t-s} - x_0), \qquad x \in C([0,T]; \R^d),
\]
and let 
\[
    \Gamma_t(z) = K(t)z_t + \int_0^t K'(t-s)( z_s - z_t ) ds, \qquad z \in D(\Gamma)
\]
where its domain $D(\Gamma)$ consist of functions $z \in C([0,T]; \R^d)$ such that  
\[
    \int_0^T |K(t)z_t|dt + \int_0^T\int_0^t |K'(t-s)| |z_s - z_t| ds dt < \infty.
\]
Remark that $\Gamma(z)$ is well-defined as an element in $L^1([0,T]; \R^d)$. Below we study the mapping properties of $Z$ and $\Gamma$ in the scale of H\"older spaces and, in particular, show that they are inverse to each other. The latter allows us to show that $Z(X)$ solves a path-dependent diffusion equation.

\begin{Lemma}\label{lemma: fractional differentiation}
    Fix $T > 0$. Suppose that $K \in L^1([0,T]; \R^{d \times d})$ is completely monotone such that $K(t_0) > 0$ for some $t_0 > 0$. Then the following assertions hold: 
    \begin{enumerate}
        \item[(a)] The mapping $Z$ defines for each $\eta \in [0,1]$ a continuous linear operator
        \[
            Z: C^{\eta}([0,T]; \R^d) \longrightarrow C_0^{\eta}([0,T]; \R^d).
        \]
        \item[(b)] Suppose that $K \in C^1(\R_+; \R^{d \times d})$. Then $C([0,T]; \R^d) \subset D(\Gamma)$, 
        \[
            \Gamma: C([0,T]; \R^{d}) \longrightarrow C([0,T]; \R^d)
        \]
        is a continuous linear operator, and there exists a constant $C > 0$ independent of $t \in [0,T]$ and $z \in C([0,T]; \R^d)$ such that
        \[
            |\Gamma_t(z)|^2 \leq C|K(0)|^2 |z_t|^2 + C \left( \int_0^T |K'(r)|^2 dr \right)\int_0^t |z_s|^2 ds.
        \]
        Moreover, for each $x \in C([0,T]; \R^d)$ and each $z \in C([0,T]; \R^d)$ it holds that
        \begin{align}\label{eq: ZGamma inverse}
            \Gamma(Z(x)) = x - x_0 \ \text{ and } \ Z(\Gamma(z)) = z - L([0,t])z_0.
        \end{align}
        \item[(c)] Suppose there exists $\alpha \in (1/2,1]$ with 
        \begin{align}\label{eq: K asymptotics}
            \sup_{t \in (0,T]}\left(t^{1-\alpha}|K(t)| + t^{2 - \alpha}|K'(t)| \right) < \infty.
        \end{align}
        Let $\eta \in (1-\alpha,1]$. Then $C_0^{\eta}([0,T]; \R^d) \subset D(\Gamma)$ and 
        \[
            \Gamma: C_0^{\eta}([0,T]; \R^d) \longrightarrow C_0([0,T]; \R^d)
        \]
        is a continuous linear operator. Moreover, for each $x \in C([0,T]; \R^d)$ satisfying $Z(x) \in C_0^{\eta}([0,T]; \R^d)$, and each $z \in C_0^{\eta}([0,T]; \R^d)$, the following relations hold
        \[
            \Gamma(Z(x)) = x - x_0 \ \text{ and } \ Z(\Gamma(z)) = z.
        \] 
        \item[(d)] Finally, if $K$ satisfies in addition to \eqref{eq: K asymptotics} also $\sup_{t \in [0,T]}t^{3-\alpha}|K''(t)| < \infty$, then 
        \[
            \Gamma: C_0^{\eta}([0,T]; \R^d) \longrightarrow C_0^{\eta + \alpha - 1}([0,T]; \R^d)
        \]
        is a continuous linear operator.
    \end{enumerate}
\end{Lemma} 
\begin{proof}
    \textit{(a)} Let $x \in C([0,T]; \R^d)$ and $t \in [0,T]$. Then $|x_t - x_0| \leq 2 \|x\|_{\infty}$ and hence 
    \begin{align}\label{eq: 1}
        |Z_t(x)| \leq \int_{[0,t]}|x_{t-s} - x_0| |L|(ds) \leq 2 \|x\|_{\infty}|L|([0,T])
    \end{align}
    where $|L|$ denotes the variation measure of $L$. Hence, $Z$ is well-defined. It is clear that $Z_0(x) = 0$. To prove continuity, let $0 \leq s < t \leq T$. Then using $L(dr) = A\delta_0(dr) + L_0(r)dr$ with $L_0 \in L^1([0,T]; \R^{d \times d})$ completely monotone (see \eqref{eq: resolvent of the first kind}), we obtain
    \begin{align*}
        |Z_t(x) - Z_s(x)| \leq |A||x_{t} - x_s| + \int_{(0,s]} |x_{t-r} - x_{s-r}| |L_0(r)|dr
        + \int_{(s,t]}|x_{t-r} - x_0| |L_0(r)|dr.
    \end{align*}
    Since $x$ is continuous, all terms converge to zero by dominated convergence. This proves $Z(x) \in C_0([0,T]; \R^d)$. By \eqref{eq: 1}, it is clear that $Z$ is a continuous linear operator on $C([0,T]; \R^d)$. Finally, let $x \in C^{\eta}([0,T]; \R^d)$ and $0 \leq s < t$. Then
    \begin{align*}
        |Z_t(x) - Z_s(x)|
        &\leq \int_{[0,s]}|x_{t-r} - x_{s-r}| |L|(dr)
        + \int_{(s,t]}|x_{t-r} - x_0| |L|(dr)
        \\ &\leq \|x\|_{C^{\eta}}\left(|L|([0,s])(t-s)^{\eta} + \int_{(s,t]} (t-r)^{\eta}|L|(dr) \right)
        \\ &\leq \|x\|_{C^{\eta}}|L|([0,T])(t-s)^{\eta},
    \end{align*}
    and hence $Z(x) \in C_0^{\eta}([0,T]; \R^d)$.

    \textit{(b)} Firstly, it is easy to see that $C([0,T]; \R^d) \subset D(\Gamma)$. Let $z \in C([0,T]; \R^d)$. Then we obtain by direct computation
    \begin{align}\label{eq: 3}
        \Gamma_t(z) = \frac{d}{dt}\int_0^t K(t-s)z_s ds = K(0)z_t + \int_0^t K'(t-s)z_sds.
    \end{align}
    Hence, an application of the Cauchy-Schwarz inequality gives  
    \[
        |\Gamma_t(z)|^2 \lesssim |K(0)|^2 |z_t|^2 + \left(\int_0^T |K'(r)|^2 dr \right) \int_0^t |z_s|^2 ds
    \]
    which is finite due to $K \in C^1(\R_+; \R^{d \times d})$. In particular, $\Gamma: C([0,T]; \R^d) \longrightarrow C([0,T]; \R^d)$ is a continuous linear mapping. 
    
    Finally, it remains to show that $Z, \Gamma$ are inverse to each other. Integrating both sides of \eqref{eq: 3}, we obtain for each $z \in C([0,T]; \R^d)$
    \begin{align}\label{eq: Gamma relation}
     \int_0^t K(t-s)z_sds = \int_0^t \Gamma_s(z)ds.
    \end{align}
    This identity uniquely determines the relationship between $Z$ and $\Gamma$. Indeed, let $x \in C([0,T]; \R^d)$. Then \eqref{eq: Gamma relation} applied to $z = Z(x)$ gives 
    \[
     \int_0^t (x_s - x_0)ds = \int_0^t K(t-s)Z_s(x)ds = \int_0^t \Gamma_s(Z(x))ds
    \]
    and hence $x_t = x_0 + \Gamma_t(Z(x))$. Conversely, let $z \in C([0,T]; \R^d)$. Convolving \eqref{eq: Gamma relation} with $L$ and using $L \ast K = 1$ gives $1 \ast z = L \ast 1 \ast \Gamma(z) = 1 \ast (L \ast \Gamma(z))$ and hence
    \begin{align*}
     z_t = \int_{[0,t]}L(ds)\Gamma_{t-s}(z) &= \int_{[0,t]}L(ds)(\Gamma_{t-s}(z) - \Gamma_0(z)) + L([0,t])\Gamma_0(z)
     \\ &= Z_t(\Gamma(z)) + L([0,t])z_0.
    \end{align*}
    This proves all assertions.

    \textit{(c)} Let $z \in C_0^{\eta}([0,T]; \R^d)$ where $\eta + \alpha > 1$. Then
    \[
        |K(t)z_t| = |K(t)(z_t - z_0)| \lesssim t^{\alpha - 1}t^{\eta}\|z\|_{C^{\eta}}
    \]
    and similarly
    \[
        \int_0^T \int_0^t |K'(t-s)| |z_s - z_t| ds dt \lesssim \int_0^T \int_0^t |t-s|^{\alpha - 2}|t-s|^{\eta}ds dt \|z\|_{C^{\eta}} < \infty.
    \]
    Thus, since $\eta + \alpha > 1$, we obtain $C_0^{\eta}([0,T]; \R^d) \subset D(\Gamma)$ and  
    \begin{align*}
        |\Gamma_t(z)| &\lesssim \|z\|_{C^{\eta}}\left( t^{\eta + \alpha - 1} + \int_0^t (t-s)^{\eta + \alpha - 2}ds \right)
        = \|z\|_{C^{\eta}}\frac{\alpha + \eta}{\alpha + \eta - 1}t^{\alpha + \eta - 1}.
    \end{align*}
    This shows that $\Gamma(z)$ is well-defined and has the desired H\"older continuity at $t = 0$. Now let $0 < s < t$ be arbitrary. A short computation yields the decomposition
    \begin{align} \notag
        \Gamma_t(z) - \Gamma_s(z) &= (K(t) - K(s))z_s + K(t-s)(z_t - z_s)
        \\ &\qquad + \int_s^t K'(t-r)(z_r - z_t)dr \notag
        \\ &\qquad + \int_0^s (K'(t-s + r) - K'(r))(z_{s-r} - z_s)dr \notag
        \\ &= I_1 + I_2 + I_3 + I_4. \label{eq: 9}
    \end{align}
    Using \eqref{eq: K asymptotics} combined with dominated convergence, the right-hand side tends to zero as $t-s \to 0$. Hence $\Gamma(z) \in C_0([0,T]; \R^d)$.

    It remains to show that $\Gamma(Z(x)) = x - x_0$ and $Z(\Gamma(z)) = z$. Let $\e \in (0,1)$, define $K_{\e}(t) = K(t + \e)$, and note that $K^{\varepsilon} \in C^1(\R_+; \R^{d \times d})$ is also completely monotone with $K_{\varepsilon}(t_0 - \varepsilon) > 0$. Let
    \[
     \Gamma^{\e}_t(z) = K_{\e}(t)z_t + \int_0^t K_{\e}'(t-s)(z_s - z_t)ds.
    \]
    Then according to (b), $\Gamma^{\varepsilon}(z)$ is well-defined, satisfies $\Gamma^{\e}_t(z) = \frac{d}{dt}(K_{\e} \ast z)_t$, and hence after integration also
    \[
        \int_0^t \Gamma^{\e}_s(z)ds = \int_0^t K_{\e}(t-s)z_s ds.
    \]
    Using \eqref{eq: K asymptotics}, we find $|K(\varepsilon + t)| \lesssim (\varepsilon + t)^{\alpha - 1} \leq t^{\alpha - 1}$ and $|K'( \varepsilon + t)| \lesssim t^{\alpha - 2}$. Hence, by dominated convergence, we may pass on both sides to the limit $\varepsilon \searrow 0$, which gives 
    \[
        \int_0^t \Gamma_s(z)ds = \int_0^t K(t-s)z_s ds, \qquad z \in C_0^{\eta}([0,T]; \R^d).
    \]
    The assertion can now be deduced exactly as in part (b). 

    \textit{(d)} It remains to show that $\Gamma(z) \in C^{\eta + \alpha - 1}([0,T]; \R^d)$. Using decomposition \eqref{eq: 9}, it suffices to estimate $I_1,\dots, I_4$ separately. For $I_1$ we use $z_0 = 0$ and hence the H\"older continuity of $z$ in $0$ combined with \eqref{eq: K asymptotics} to find that
    \begin{align*}
        |I_1| \lesssim \| z\|_{C^{\eta}} s^{\eta} \int_s^t r^{\alpha - 2}dr 
        &\leq \| z\|_{C^{\eta}}\int_s^{t} r^{\alpha + \eta - 2}dr 
        \\ &\leq \| z\|_{C^{\eta}}\frac{t^{\alpha + \eta - 1} - s^{\alpha + \eta - 1}}{\alpha + \eta - 1} \lesssim \| z\|_{C^{\eta}} (t-s)^{\eta + \alpha - 1}.
    \end{align*}
    For $I_2$ and $I_3$ we obtain from \eqref{eq: K asymptotics} and the H\"older continuity of $z$
    \[
     |I_2| + |I_3| \lesssim \| z\|_{C^{\eta}} (t-s)^{\eta}|K(t-s)| + \| z\|_{C^{\eta}}\int_s^t (t-r)^{\alpha - 2 + \eta}dr \lesssim \| z\|_{C^{\eta}}(t-s)^{\eta + \alpha - 1}.
    \]
    To bound the last term, let us first note that again by \eqref{eq: K asymptotics}
    \begin{align*}
     |K'(t-s+r) - K'(r)| &\lesssim \int_r^{t-s+r} u^{\alpha - 3}du
     \\ &\lesssim r^{\alpha - 1}\int_r^{t-s+r}u^{-2}du \lesssim r^{\alpha - 1}\frac{t-s}{r(t-s + r)}.
    \end{align*}
    Hence we obtain by substitution $u = r/(t-s)$
    \begin{align*}
        |I_4| &\lesssim \| z\|_{C^{\eta}}\int_0^s r^{\eta +\alpha - 1} \frac{t-s}{r(t-s + r)} dr
        \\ &= \| z\|_{C^{\eta}}(t-s)^2 \int_0^{\frac{s}{t-s}} (t-s)^{\eta + \alpha - 1} u^{\eta + \alpha - 1} \frac{1}{(t-s)u(t-s + (t-s)u)}du
        \\ &\lesssim \| z\|_{C^{\eta}} (t-s)^{\eta + \alpha - 1}\int_0^{\infty} u^{\eta + \alpha - 2}(u+1)^{-1}du < \infty.
    \end{align*}
    Thus, we have shown that $\|\Gamma(z)\|_{C^{\eta + \alpha - 1}} \lesssim \| z\|_{C^{\eta}}$, which proves the first assertion.
\end{proof}

Condition \eqref{eq: K asymptotics} asserts that $K$ and its derivatives are bounded above by the fractional kernel \eqref{eq: fractional kernel}. In particular, for the fractional kernel \eqref{eq: fractional kernel}, we get $L(ds) = \frac{t^{-\alpha}}{\Gamma(1-\alpha)}ds$, and hence $Z$ corresponds to fractional integration, while its inverse transformation $\Gamma$ becomes fractional differentiation of order $\alpha$. From this perspective, this lemma provides an abstract analogue for such operations.
 
Condition \eqref{eq: K asymptotics} is not satisfied for $K(t) = e^{t^{-1}}$, for example. However, this condition is not essential. Analogous results can be obtained as long as we may pass to the limit $\varepsilon \to 0$ as done in part (c). Finally, remark that condition \eqref{eq: K asymptotics} may be slightly weakened at the expense of working with not necessarily H\"older continuous functions. We did not pursue the direction further to keep this work at a reasonable length. Moreover, if $K = \mathrm{diag}(K_1,\dots, K_d)$, then also $L = \mathrm{diag}(L_1,\dots, L_d)$, and this lemma can be applied component-wise.

Finally, let us remark that here we focus on Volterra kernels that admit a resolvent of the first kind which allows us to define transformation $Z$. For regular Volterra kernels that do not admit a resolvent of the first kind, e.g. $K(t) = t^{\alpha - 1}$ with $\alpha > 3/2$, the stochastic Volterra process is a semimartingale. In such a case, an integration by parts allows us to express $X$ in terms of its path-dependent semimartingale characteristics. The latter is sufficient for the methods presented in the subsequent sections. Details are beyond the scope of this work and left to the interested reader.

\subsection{Equivalence of measures}

In this section, we prove the equivalence of laws \eqref{eq: 4} under the assumption that the Volterra kernel $K$ satisfies the following condition.
\begin{assumption}\label{A}
The Volterra kernel $K \in L_{loc}^2(\R_+; \R^{d \times d})$ is completely monotone such that $K(t_0) > 0$ holds for some $t_0 > 0$, and either $K \in C^1(\R_+; \R^{d \times d})$ or there exists $\alpha \in (1/2,1]$ such that
    \[
        \sup_{t \in (0,T]}\left( t^{1-\alpha}|K(t)| + t^{2 - \alpha}|K'(t)|\right) < \infty, \qquad T > 0.
    \]    
\end{assumption}
Remark that under this assumption we may apply Lemma \ref{lemma: fractional differentiation}. Let $b, \widetilde{b}: \R^d \longrightarrow \R^d$ denote the drift functions, $\sigma: \R^d \longrightarrow \R^{m \times k}$ the diffusion matrix. Under standard Lipschitz assumptions and assumption \ref{A}, there exists a unique strong solution $X$ of the stochastic Volterra equation \eqref{eq: VSDE}. Similarly, let $\widetilde{X}$ be the unique strong solution with $b$ replaced by $\widetilde{b}$, i.e. 
\[
    \widetilde{X}_t = x_0 + \int_0^t K(t-s)b(\widetilde{X}_s)ds + \int_0^t K(t-s)\sigma( \widetilde{X}_s)dB_s,
\]
where $B$ denotes the same standard Brownian motion on $\R^m$ as in \eqref{eq: VSDE} used for $X$. It follows from \cite[Lemma 3.1]{MR4019885} that $\sup_{t \in [0,T]}\E[|X_t|^p + |\widetilde{X}_t|^p] < \infty$ holds for each $T, p > 0$. Hence, an application of \cite[Lemma 2.4]{MR4019885} shows that, for each $\varepsilon > 0$, $X, \widetilde{X}$ have a modification with $\alpha - 1/2 - \varepsilon$ H\"older continuous sample paths. 

Denote by $L$ the resolvent of the first kind of $K$, and let $Z_t = Z_t(X)$ and $\widetilde{Z}_t = Z_t(\widetilde{X}_t)$ be given by \eqref{eq: Z process} with pointwise evaluation at $x = X$ and $x = \widetilde{X}$, respectively. Finally, denote by $\P^Z$ the law of $Z(X)$ on the path space $C_0([0,T]; \R^d)$, and similarly $\P^{\widetilde{Z}}$. If $\P^Z|_{\F_T} \ll \P^{\widetilde{Z}}|_{\F_T}$, we denote by
\[
    \frac{d\P^Z}{d\P^{\widetilde{Z}}}\bigg|_{\F_T}: C_0([0,T]; \R^d) \longrightarrow \R_+
\]
its Radon-Nikodym derivative on $\F_T$. By $\frac{d\P^Z}{d\P^{\widetilde{Z}}}|_{\F_T}(\widetilde{Z})$ we denote the random variable obtained by pointwise evaluation of this density at $z = \widetilde{Z}$.

\begin{Theorem}\label{thm: equivalence of laws}
    Suppose that assumption \ref{A} is satisfied, that the drift and diffusion coefficients $b, \widetilde{b}, \sigma$ are Lipschitz continuous, $\sigma \sigma^{\top}$ is invertible, 
    \begin{align}\label{eq: uniform nondegeneracy}
        H(x) := b(x)^{\top} \left( \sigma(x) \sigma(x)^{\top} \right)^{-1} b(x) + \widetilde{b}(x) \left(\sigma(x) \sigma(x)^{\top} \right)^{-1} \widetilde{b}(x) 
    \end{align}
    is locally bounded, and 
    \begin{align}\label{eq: integrability}
        \P\left[ \int_0^T H(\widetilde{X}_t) dt < \infty \right] = 1.
    \end{align}
    Then $\P_{Z}|_{\F_T} \ll \P_{\widetilde{Z}}|_{\F_T}$ on $\F_T$ with Radon-Nikodym derivative given by
    \begin{align*}
        \log\ \frac{d\P_Z}{d\P_{\widetilde{Z}}}\bigg|_{\F_T}(\widetilde{Z}) = \ell_T(\widetilde{Z})
    \end{align*}
        with 
    \begin{align}\label{eq: ellT general}
        \ell_T(z) = \int_0^T B_t^-(z)^{\top} \Sigma_t(z) dz_s - \frac{1}{2} \int_0^T B_t^-(z)^{\top} \Sigma_t(z) B_t^+(z) dt 
    \end{align} 
    and the functions 
    \begin{align*}
        B_t^{\pm}(z) &= b(x_0 + \Gamma_t(z)) \pm \widetilde{b}(x_0 + \Gamma_t(z)),
        \\ \Sigma_t(z) &= \left(\sigma(x_0 + \Gamma_t(z)) \sigma(x_0 + \Gamma_t(z))^{\top}\right)^{-1}.
    \end{align*}
\end{Theorem}
\begin{proof}
     \textit{Step 1.} Let us first show that $Z := Z(X)$ solves a path-dependent diffusion equation. Firstly, using the associativity of the convolution (see \cite[Lemma 2.1]{MR4019885}) combined with $L \ast K = K \ast L = 1$, we obtain $Z = L \ast (X - x_0) = L \ast (K \ast b(X)) + L \ast (K \ast \sigma(X)dB) = 1 \ast b(X) + 1 \ast \sigma(X)dB$. Hence, $Z$ is a semimartingale with representation 
    \[
        dZ_t = b(X_t)dt + \sigma(X_t)dB_t.
    \]
    Since $\sup_{t \in [0,T]}\E[|X_t|^p] < \infty$ for each $T, p > 0$, an application of the Burkholder-Davies-Gundy inequality gives
    \begin{align*}
        \| Z_t - Z_s\|_{L^p(\Omega)}
        &\lesssim \int_s^t \|b(X_r)\|_{L^p(\Omega)}dr 
        + \left( \E\left[ \left(\int_s^t |\sigma(X_r)|^2 dr \right)^{p/2} \right]\right)^{1/p}
        \\ &\lesssim (t-s) + (t-s)^{1/2}.
    \end{align*}
    By an application of the Kolmogorov-Chentsov continuity criterion, we conclude that for each $\eta > 0$, $Z$ has a modification with $1/2 - \eta$-H\"older continuous sample paths, i.e. $\P^Z[ z \in C_0^{1/2 - \eta}(\R_+; \R^d)] = 1$. In particular, either if $K \in C^1([0,T]; \R^{d \times d})$ or $K$ satisfies \eqref{eq: K asymptotics} and hence letting $\eta$ be small enough such that $\frac{1}{2} - \eta + \alpha > 1$, we may apply Lemma \ref{lemma: fractional differentiation} to find $X_t - x_0 = \Gamma_t(Z)$. Hence, $Z$ solves the path-dependent diffusion equation
    \[
        dZ_t = b(x_0 + \Gamma_t(Z))dt + \sigma(x_0 + \Gamma_t(Z))dB_t.
    \]
    A similar argument applied to $\widetilde{Z} := Z(\widetilde{X})$ also gives the semimartingale representation $d\widetilde{Z}_t = \widetilde{b}(\widetilde{X}_t)dt + \sigma(\widetilde{X}_t)dB_t$, and hence $\widetilde{Z}$ solves the path-dependent diffusion equation $d\widetilde{Z}_t = \widetilde{b}(x_0 + \Gamma_t(\widetilde{Z}))dt + \sigma(x_0 + \Gamma_t(\widetilde{Z}))dB_t$. 
    
    \textit{Step 2.} Let us now suppose that $K \in C^1(\R_+; \R^{d \times d})$. Then, using Lemma \ref{lemma: fractional differentiation}.(b), combined with the Lipschitz continuity of $b, \widetilde{b}, \sigma$, we find 
    \begin{align*}
        |b(x_0 + \Gamma_t(z))|^2 + |\widetilde{b}(x_0 + \Gamma_t(z))|^2 + |\sigma(x_0 + \Gamma_t(z))|^2 
        &\lesssim 1 + |\Gamma_t(z)|^2
        \\ &\lesssim ( 1 + |z_t|^2) + \int_0^t (1 + |z_s|^2) ds.
    \end{align*}
    Similarly, we obtain from the Lipschitz continuity combined with the linearity of $\Gamma$
    \begin{align*}
         |b(x_0 + \Gamma_t(z)) &- b(x_0 + \Gamma_t(\widetilde{z}))|^2 + |\widetilde{b}(x_0 + \Gamma_t(z)) - \widetilde{b}(x_0 + \Gamma_t(\widetilde{z}))|^2 
         \\ &+ |\sigma(x_0 + \Gamma_t(z)) - \sigma(x_0 + \Gamma_t(\widetilde{z}))|^2 
        \\ &\qquad \lesssim |\Gamma_t(z - \widetilde{z})|^2 
        \lesssim |z_t - \widetilde{z}_t|^2 + \int_0^t |z_s - \widetilde{z}_s|^2 ds.
    \end{align*}
    Hence, the multi-dimensional analogue of the path-dependent Lipschitz and linear growth condition \cite[condition (I)]{MR1800857} is satisfied. Define 
    \begin{align}\label{eq: u}
        u(x) = \sigma(x)^{\top} ( \sigma(x) \sigma(x)^{\top})^{-1} (b(x) - \widetilde{b}(x)).
    \end{align}
    Then we obtain $\sigma(x) u(x) = b(x) - \widetilde{b}(x)$ for $x \in \R^d$. Using the notation $x_t = x_0 + \Gamma_t(z)$, we verify that $u_t(z) := u(x_0 + \Gamma_t(z))$, satisfies with $z \in C([0,T]; \R^d)$
    \begin{align*}
        \sigma(x_0 + \Gamma_t(z)) u_t(z) 
        = \sigma(x_t) u(x_t) 
        &= b(x_t) - \widetilde{b}(x_t)
        \\ &= b(x_0 + \Gamma_t(z)) - \widetilde{b}(x_0 + \Gamma_t(z))
    \end{align*}
    which corresponds to the multi-dimensional analogue of \cite[condition (II), p. 291]{MR1800857}. Finally, using $x_0 + \Gamma_t(Z) = X_t$ due to Lemma \ref{lemma: fractional differentiation}, we find  
    \begin{align*}
        &\ \P^{\widetilde{Z}}\left[ \int_0^T \left( b(x_0 + \Gamma_t(z))^{\top}\Sigma_t(z) b(x_0 + \Gamma_t(z)) + \widetilde{b}(x_0 + \Gamma_t(z))^{\top}\Sigma_t(z) \widetilde{b}(x_0 + \Gamma_t(z)) \right)dt < \infty \right]
        \\ &\qquad \qquad \qquad = \P\left[ \int_0^T H(\widetilde{X}_t) dt < \infty \right] = 1
    \end{align*}
    by assumption \eqref{eq: uniform nondegeneracy}. Hence, according to \cite[p. 296]{MR1800857}, $\P^Z|_{\F_T} \ll \P^{\widetilde{Z}}|_{\F_T}$ with the desired Radon-Nikodym derivative. 

    \textit{Step 3.} Next, we consider the general case where $K$ is not necessarily $C^1(\R_+; \R^{d \times d})$, but satisfies \eqref{eq: K asymptotics}. Let $\varepsilon \in (0,t_0)$, and set $K^{\varepsilon}(t) = K(t+\varepsilon)$. Let $X^{\varepsilon}$ be the unique solution of 
    \[
        X^{\varepsilon}_t = x_0 + \int_0^t K^{\varepsilon}(t-s)b(X^{\varepsilon}_s)ds + \int_0^t K^{\varepsilon}(t-s)\sigma(X^{\varepsilon}_s)dB_s,
    \]
    and let $\widetilde{X}^{\varepsilon}$ be defined analogously with $b$ replaced by $\widetilde{b}$. Since $K^{\varepsilon} \in C^1(\R_+; \R^{d \times d})$ is also completely monotone with $K_{\varepsilon}(t_0 - \varepsilon) > 0$, it admits a unique resolvent of the first kind denoted by $L^{\varepsilon}$. Let $Z^{\varepsilon} := Z^{\varepsilon}(X^{\varepsilon}) = L^{\varepsilon} \ast (X^{\varepsilon} - x_0)$, and similarly $\widetilde{Z}^{\varepsilon} = Z^{\varepsilon}(\widetilde{X}^{\varepsilon})$. By \eqref{eq: K asymptotics} an application of dominated convergence gives $K^{\varepsilon} \longrightarrow K$ in $L^2([0,T]; \R^{d \times d})$, from which it is easy to see that 
    \begin{align}\label{eq: 6}
        \lim_{\varepsilon \to 0}\sup_{t \in [0,T]}\E\left[ |X^{\varepsilon}_t - X_t|^p + |\widetilde{X}^{\varepsilon}_t - \widetilde{X}_t|^p \right] = 0, \qquad p \geq 2,
    \end{align}
    e.g. one may modify the arguments given in \cite{MR3934104} to the multi-dimensional case. Hence, using $Z^{\varepsilon} = \int_0^{\cdot} b(X^{\varepsilon}_s)ds + \int_0^{\cdot} \sigma(X_s^{\varepsilon})dB_s$ and $Z = \int_0^{\cdot} b(X_s)ds + \int_0^{\cdot}\sigma(X_s)dB_s$, and analogous expressions for $\widetilde{Z}^{\varepsilon}, \widetilde{Z}$, an application of the Burkholder-Davis-Gundy inequality combined with \eqref{eq: 6} gives 
    \begin{align}\label{eq: 7}
        \lim_{\varepsilon \to 0}\E\left[ \sup_{t \in [0,T]} \left(|Z_t^{\varepsilon} - Z_t|^p + |\widetilde{Z}_t^{\varepsilon} - \widetilde{Z}_t|^p \right) \right] = 0, \qquad p \geq 2.
    \end{align}
    Define $\Gamma^{\varepsilon}$ as $\Gamma$ with $K$ replaced by $K^{\varepsilon}$. According to step 2, we obtain $\P^{Z^{\varepsilon}}|_{\F_T} \ll \P^{\widetilde{Z}^{\varepsilon}}|_{\F_T}$ with Radon-Nikodym derivative given by 
    \[
        \log\ \frac{d\P^{Z^{\varepsilon}}}{d\P^{\widetilde{Z}^{\varepsilon}}}\bigg|_{\F_T}(\widetilde{Z}^{\varepsilon}) = \ell_T(\widetilde{Z}^{\varepsilon})
    \]
    where $\ell_T$ is given by \eqref{eq: ellT general}, and is, in particular, independent of $\varepsilon$.
    
    Let $F \in C_b( C([0,T];\R^d))$, then we obtain
    \begin{align}\label{eq: 2}
        \E\left[ F(Z^{\varepsilon}) \right]
        = \E_{\P^{Z^{\varepsilon}}}\left[ F(z) \right]
        = \E_{\P^{\widetilde{Z}^{\varepsilon}} } \left[ e^{\ell_T(z)} F(z) \right]
        = \E\left[ e^{\ell_T(\widetilde{Z}^{\varepsilon})} F(\widetilde{Z}^{\varepsilon}) \right].
    \end{align}
    Using the particular form of $\ell_T(z)$ combined with \eqref{eq: 6}, \eqref{eq: 7}, and the Lipschitz continuity of the coefficients, we find 
    \begin{align*}
        \ell_T( \widetilde{Z}^{\varepsilon}) &= \int_0^T B_t^-(\widetilde{Z}^{\varepsilon})^{\top} \Sigma_t( \widetilde{Z}^{\varepsilon}) \widetilde{b}( \widetilde{X}_s^{\varepsilon})ds 
        \\ &\qquad + \int_0^T B_t^-( \widetilde{Z}^{\varepsilon})^{\top} \Sigma_t( \widetilde{Z}^{\varepsilon}) \sigma( \widetilde{X}_s^{\varepsilon})dB_s 
        - \frac{1}{2} \int_0^T B_t^-( \widetilde{Z}^{\varepsilon})^{\top}\Sigma_t( \widetilde{Z}^{\varepsilon})B_t^+( \widetilde{Z}^{\varepsilon})dt
        \\ &\longrightarrow \int_0^T B_t^-(\widetilde{Z})^{\top} \Sigma_t( \widetilde{Z}) \widetilde{b}( \widetilde{X}_s)ds 
        \\ &\qquad + \int_0^T B_t^-( \widetilde{Z})^{\top} \Sigma_t( \widetilde{Z}) \sigma( \widetilde{X}_s)dB_s 
        - \frac{1}{2} \int_0^T B_t^-( \widetilde{Z})^{\top}\Sigma_t( \widetilde{Z})B_t^+( \widetilde{Z})dt
        \\ &= \ell_T( \widetilde{Z})
    \end{align*}
    in $L^2(\Omega)$ as $\varepsilon \to 0$, and hence $e^{\ell_T(\widetilde{Z}^{\varepsilon})} \longrightarrow e^{\ell_T(\widetilde{Z})}$ in probability. Let us first assume that 
    \begin{align}\label{eq: 8}
        \E\left[ e^{\ell_T(\widetilde{Z})} \right] = 1.
    \end{align}
    Using \eqref{eq: 7} and $e^{\ell_T(\widetilde{Z}^{\varepsilon})} \longrightarrow e^{\ell_T(\widetilde{Z})}$ in probability, we get $e^{\ell_T(\widetilde{Z}^{\varepsilon})} F(\widetilde{Z}^{\varepsilon}) \longrightarrow e^{\ell_T(\widetilde{Z})} F(\widetilde{Z})$ in probability. Noting that $|e^{\ell_T(\widetilde{Z}^{\varepsilon})} F(\widetilde{Z}^{\varepsilon})| \lesssim e^{\ell_T(\widetilde{Z}^{\varepsilon})}$ since $F$ is bounded, and that $\E\left[ e^{\ell_T(\widetilde{Z})} \right] = 1 = \lim_{\varepsilon \to 0}\E[e^{\ell_T(\widetilde{Z}^{\varepsilon})}]$ by \eqref{eq: 8} and $\E\left[ e^{\ell_T( \widetilde{Z}^{\varepsilon})}\right] = 1$ which follows from $\P^{Z^{\varepsilon}}|_{\F_T} \ll \P^{\widetilde{Z}^{\varepsilon}}|_{\F_T}$ as shown in step 2. Hence, we may apply the generalised dominated convergence theorem to take the limit in \eqref{eq: 2}, which gives
    \[
        \E\left[ F(Z) \right] = \lim_{\varepsilon \to 0 } \E\left[ e^{\ell_T(\widetilde{Z}^{\varepsilon})} F(\widetilde{Z}^{\varepsilon}) \right] = \E\left[ e^{\ell_T(\widetilde{Z})} F(\widetilde{Z}) \right].
    \]
    Since $\P^Z \sim Z$ and $\P^{\widetilde{Z}} \sim \widetilde{Z}$, the assertion is proved, once we have shown \eqref{eq: 8}.

    \textit{Step 4.} We finally prove \eqref{eq: 8}. First, note that 
    \[
        \E\left[ e^{\ell_T(\widetilde{Z})} \right] \leq \liminf_{\varepsilon \to 0}\E\left[ e^{\ell_T( \widetilde{Z}^{\varepsilon})}\right] = 1
    \]
    by the Fatou's Lemma combined with $\E\left[ e^{\ell_T( \widetilde{Z}^{\varepsilon})}\right] = 1$. To prove the converse inequality, let us define the stopping time $\tau_n(x) = \inf\{t \in (0,T] \ : \ |x_t| \geq n \} \wedge T$ where $x \in C([0,T]; \R^d)$. Recall that $u$ is given by \eqref{eq: u}, whence $u_t(z) = u(x_0 + \Gamma_t(z))$ satisfies $u_t(z) = \sigma(x_0 + \Gamma_t(z))^{\top}\Sigma_t(z)B_t^-(z)$. By definition of $\ell_T$, a short computation gives
    \begin{align*}
        \ell_T(\widetilde{Z}) = \int_0^T u_t(\widetilde{Z})^{\top} dB_t - \frac{1}{2}\int_0^T u_t(\widetilde{Z})^{\top}u_t(\widetilde{Z}) dt.
    \end{align*}
    Define $\ell_T^{(n)}(\widetilde{Z})$ by
    \begin{align*}
        \ell^{(n)}_T(\widetilde{Z}) = \int_0^T \1_{\{t\leq \tau_n(\widetilde{X})\}}u_t(\widetilde{Z})^{\top} dB_t - \frac{1}{2}\int_0^T \1_{\{t \leq \tau_n(\widetilde{X})\}}u_t(\widetilde{Z})^{\top}u_t(\widetilde{Z}) dt
    \end{align*}
    where $\widetilde{X} = x_0 + \Gamma_t(\widetilde{Z})$ is a functional of $\widetilde{Z}$ due to Lemma \ref{lemma: fractional differentiation}. Note that $\ell_T(\widetilde{Z}) = \ell_T^{(n)}(\widetilde{Z})$ on the event $\{\tau_n(\widetilde{X}) = T\}$, and that 
    \begin{align*}
        |u_t(\widetilde{Z})|^2
        &= | \sigma(x_0 + \Gamma_t(\widetilde{Z}))^{\top}\Sigma_t(\widetilde{Z})B_t^-(\widetilde{Z})|^2
        \\ &= B_t^-(\widetilde{Z})^{\top} \Sigma_t(\widetilde{Z}) B_t^-(\widetilde{Z})
        \\ &= b(x_0 + \Gamma_t(\widetilde{Z}))^{\top}\Sigma_t(\widetilde{Z})b(x_0 + \Gamma_t(\widetilde{Z})) + \widetilde{b}(x_0 + \Gamma_t(\widetilde{Z}))^{\top}\Sigma_t(\widetilde{Z})\widetilde{b}(x_0 + \Gamma_t(\widetilde{Z})) 
        \\ &\qquad \qquad - 2 \widetilde{b}(x_0 + \Gamma_t(\widetilde{Z}))^{\top}\Sigma_t(\widetilde{Z})b(x_0 + \Gamma_t(\widetilde{Z}))
        \\ &\leq 2 b(\widetilde{X}_t)^{\top}(\sigma(\widetilde{X}_t) \sigma(\widetilde{X}_t)^{\top})^{-1} b(\widetilde{X}_t) + 2\widetilde{b}(\widetilde{X})^{\top}(\sigma(\widetilde{X}_t) \sigma(\widetilde{X}_t)^{\top})^{-1}\widetilde{b}( \widetilde{X}_t)
    \end{align*}
    where we have used $x_0 + \Gamma_t(\widetilde{Z}) = \widetilde{X}_t$ combined with
    \begin{align*}
    &\ 2 \widetilde{b}( \widetilde{X}_t)^{\top}(\sigma(\widetilde{X}_t) \sigma(\widetilde{X}_t)^{\top})^{-1} b( \widetilde{X}_t )
    \\ &\quad \leq 2| (\sigma(\widetilde{X}_t) \sigma(\widetilde{X}_t)^{\top})^{-1/2} \widetilde{b}(\widetilde{X}_t)| \cdot | (\sigma(\widetilde{X}_t) \sigma(\widetilde{X}_t)^{\top})^{-1/2} b(\widetilde{X}_t)|
    \\ &\quad \leq | (\sigma(\widetilde{X}_t) \sigma(\widetilde{X}_t)^{\top})^{-1/2} \widetilde{b}(\widetilde{X}_t)|^2 + | (\sigma(\widetilde{X}_t) \sigma(\widetilde{X}_t)^{\top})^{-1/2} b(\widetilde{X}_t)|^2
    \\ &\quad = \widetilde{b}(\widetilde{X}_t)^{\top}(\sigma(\widetilde{X}_t) \sigma(\widetilde{X}_t)^{\top})^{-1} \widetilde{b}(\widetilde{X}_t) + b(\widetilde{X}_t)^{\top}(\sigma(\widetilde{X}_t) \sigma(\widetilde{X}_t)^{\top})^{-1} b(\widetilde{X}_t) = H(\widetilde{X}_t).
    \end{align*}
    By assumption, $H$ is locally bounded. Thus we find a constant $C(n) > 0$ such that $|u_t(\widetilde{Z})| \leq C(n)$ on the event $\{t \leq \tau_n(\widetilde{X})\}$, and hence 
    \[
        \E\left[ e^{\frac{1}{2}\int_0^T \1_{\{t \leq \tau_n(\widetilde{X})\}} u_n(\widetilde{Z})^{\top}u_t(\widetilde{Z}) dt}\right] < \infty.
    \]
    In particular, it follows from the multi-dimensional analogue of \cite[Theorem 6.1]{MR1800857} that $\E[e^{\ell_T^{(n)}(\widetilde{Z})}] = 1$ for $n \geq 1$. Hence setting $d\P^{(n)} = e^{\ell_T^{(n)}(\widetilde{Z})}d\P$, we obtain
    \begin{align*}
        \E\left[ e^{\ell_T(\widetilde{Z})} \right] 
        \geq \E\left[ e^{\ell_T(\widetilde{Z})} \1_{\{\tau_n(\widetilde{X}) = T\}}\right]
        &= \E\left[ e^{\ell^{(n)}_T(\widetilde{Z})} \1_{\{\tau_n(\widetilde{X}) = T\}}\right]
        \\ &= \P^{(n)}\left[ \tau_n(\widetilde{X}) = T\right]
        = 1 - \P^{(n)}\left[ \tau_n(\widetilde{X}) < T \right].
    \end{align*}
    To bound the last probability, remark that $B^{(n)}_t = B_t - \int_0^t \1_{\{s \leq \tau_n(\widetilde{X})\}} u_s(\widetilde{Z})ds$ is a Brownian motion with respect to $\P^{(n)}$, and $\widetilde{X}$ satisfies with respect to $\P^{(n)}$ the equation
    \begin{align*}
        \widetilde{X}_t &= x_0 + \int_0^t K(t-s)\left( \widetilde{b}(\widetilde{X}_s) + \1_{\{s \leq \tau_n(\widetilde{X})\}} \sigma(\widetilde{X}_s) u(\widetilde{X}_s) \right)ds + \int_0^t K(t-s)\sigma(\widetilde{X}_s)dB_s^{(n)}
        \\ &= x_0 + \int_0^t K(t-s)\left( \1_{\{ \tau_n(\widetilde{X}) < s\}} \widetilde{b}(\widetilde{X}_s) + \1_{\{s \leq \tau_n(\widetilde{X})\}} b(\widetilde{X}_s) \right)ds 
        \\ &\qquad \qquad \qquad \qquad \qquad \qquad \qquad \qquad \qquad + \int_0^t K(t-s)\sigma(\widetilde{X}_s)dB_s^{(n)}
    \end{align*}
    where we have used $\sigma u = b - \widetilde{b}$. Since $b,  \widetilde{b}$ are globally Lipschitz continuous, it has a unique solution which satisfies by \cite[Lemma 3.1]{MR4019885}
    \[
        \sup_{n \geq 1}\sup_{t \in [0,T]}\E_{\P^{(n)}}[ |\widetilde{X}_t|^p ] < \infty, \qquad p \geq 2.
    \]
    Hence, an application of \cite[Lemma 2.4]{MR4019885} shows that its H\"older seminorm has uniformly bounded $p$-moment, whence $\sup_{n \geq 1}\E_{\P^{(n)}}[ \sup_{t \in [0,T]}|\widetilde{X}_t|^p ] < \infty$ for $p \geq 2$. By the Markov inequality, we obtain
    \[
        \P^{(n)}\left[ \tau_n(\widetilde{X}) < T \right] \leq \P^{(n)}\left[ \sup_{t \in [0,T]}|X_t| \geq n \right]
        \leq \frac{1}{n} \sup_{n \geq 1}\E_{\P^{(n)}}\left[ \sup_{t \in [0,T]}|X_t| \right].
    \]
    Thus, we arrive at 
    \[
        \E\left[ e^{\ell_T(\widetilde{Z})} \right] \geq 1 - \P^{(n)}\left[ \tau_n(\widetilde{X}) < T \right]
        \geq 1 - \frac{1}{n} \sup_{n \geq 1}\E_{\P^{(n)}}\left[ \sup_{t \in [0,T]}|X_t| \right].
    \]
    Taking the limit $n \to \infty$ proves the assertion.
\end{proof}

Let $\P^X = \P \circ X^{-1}$ be the law of $X$ on $C(\R_+; \R^d)$, and $\P^{\widetilde{X}} = \P \circ \widetilde{X}^{-1}$ the law of $\widetilde{X}$, respectively. As a corollary, below we derive our first main result \eqref{eq: 4}.

\begin{Corollary}\label{cor: equivalence laws lipschitz}
    Suppose that assumption \ref{A} is satisfied, that the drift and diffusion coefficients $b, \widetilde{b}, \sigma$ are Lipschitz continuous, $\sigma \sigma^{\top}$ is invertible, and $H$ given by \eqref{eq: uniform nondegeneracy} is locally bounded, and \eqref{eq: integrability} holds for some $T > 0$. Then $\P^X|_{\F_T} \ll \P^{\widetilde{X}}|_{\F_T}$ with density
    \[
        \log\ \frac{d\P^X}{d\P^{\widetilde{X}}} \bigg|_{\F_T} = \ell_T(\widetilde{X})
    \]
    and
    \begin{align*}
        \ell_T(\widetilde{X}) &= \int_0^T \left( b(\widetilde{X}_t) - \widetilde{b}(\widetilde{X}_t)\right)^{\top} \left( \sigma( \widetilde{X}_t) \sigma(\widetilde{X}_t)^{\top}\right)^{-1} dZ_t(\widetilde{X})
        \\ &\qquad - \frac{1}{2} \int_0^T \left( b(\widetilde{X}_t) - \widetilde{b}(\widetilde{X}_t)\right)^{\top} \left( \sigma( \widetilde{X}_t) \sigma(\widetilde{X}_t)^{\top}\right)^{-1} \left( b(\widetilde{X}_t) + \widetilde{b}(\widetilde{X}_t)\right)^{\top} dt.
    \end{align*} 
\end{Corollary}
\begin{proof}
    Let $F: C([0,T]; \R^d) \longrightarrow \R_+$ be bounded and measurable. Then using first Lemma \ref{lemma: fractional differentiation} and then Theorem \ref{thm: equivalence of laws}, we get
    \begin{align*}
        \E\left[ F(X) \right]
        &= \E\left[ F(x_0 + \Gamma(Z(X)) \right]
        \\ &= \E_{\P^Z}\left[ F(x_0 + \Gamma(z)) \right]
        \\ &= \E_{\P^{\widetilde{Z}}}\left[ e^{\ell_T(\widetilde{z})} F(x_0 + \Gamma(\widetilde{z}))\right]
        \\ &= \E\left[ e^{\ell_T(\widetilde{Z})} F(x_0 + \Gamma(\widetilde{Z})) \right]
        = \E\left[ e^{\ell_T(Z(\widetilde{X}))} F(\widetilde{X}) \right].
    \end{align*}
    This proves the assertion. 
\end{proof}

The required Lipschitz continuity is not essential, and may be replaced by other conditions found in \cite{MR1800857}. Furthermore, Lipschitz continuity may also be removed by additional approximation arguments, similarly to step 4 in the proof of Theorem \ref{thm: equivalence of laws}. Finally, under the conditions of Theorem \ref{thm: equivalence of laws} or Corollary \ref{cor: equivalence laws lipschitz}, and if additionally to \eqref{eq: integrability} also $\int_0^T H(X_t)dt < \infty$ holds a.s., then we even get $\P^Z|_{\F_T} \sim \P^{\widetilde{Z}}|_{\F_T}$ and $\P^X|_{\F_T} \sim \P^{\widetilde{X}}|_{\F_T}$.

\subsection{Application to the Volterra Ornstein-Uhlenbeck process}

In this section, we apply the previous equivalence of measures for the particular case of the Volterra Ornstein-Uhlenbeck process defined by \eqref{eq: VOU}. For given $(b,\beta)$, let $\P_{b,\beta}$ be the law of $X$ given by \eqref{eq: VOU} with drift parameters $(b,\beta)$. We obtain the following result on the MLE as defined in \eqref{eq: minimum MLE}.

\begin{Theorem}\label{lemma: VOU Girsanov}
    Suppose that the Volterra kernel $K$ satisfies assumption \ref{A}, and that $\sigma \sigma^{\top}$ is invertible. Define the $d \times d$-matrices $A_T(X), F_T(X)$ by
    \begin{align*}
    A_T(X)_{kj} &= \frac{1}{T} \int_0^T X_t^j ( (\sigma \sigma^{\top})^{-1} dZ_t)_k - \frac{1}{T}\int_0^T X_t^j dt \frac{((\sigma \sigma^{\top})^{-1} Z_T)_k}{T},
    \\ F_{T}(X)_{kj} &= \frac{1}{T}\int_0^T X_t^j X_t^k dt - \left(\frac{1}{T}\int_0^T X_t^j dt \right) \left( \frac{1}{T}\int_0^T X_t^k dt \right),
    \end{align*}    
    and assume that $\P_{b,\beta}\left[\mathrm{det}(F_T(X)) \neq 0 \right] = 1$. Then the corresponding MLE for $(b,\beta)$ are given by
    \begin{align*}
        \widehat{b}_T &=  (\sigma \sigma^{\top})^{-1}\frac{Z_T}{T} - A_T(X) F_T(X)^{-1} \frac{1}{T}\int_0^T X_t dt,
        \\  \widehat{\beta}_T &= (\sigma \sigma^{\top}) A_T(X) F_T(X)^{-1}.
    \end{align*}    
\end{Theorem}
\begin{proof}
    To simplify the notation, let us choose $b_0 = \beta_0 = 0$ for the reference measure. It is clear that all conditions of Corollary \ref{cor: equivalence laws lipschitz} are satisfied. Hence $\P_{b,\beta}|_{\F_T} \sim \P_{0,0}|_{\F_T}$ on $\F_T$ with Radon-Nikodym derivative given by
    \begin{align*}
        \log\ \frac{d\P_{b,\beta}}{d\P_{0,0}}\bigg|_{\F_T}(X) = \ell_T(b,\beta; X),
    \end{align*}
    where
    \begin{align*}
        \ell_T(b,\beta; X) &= \int_0^T \left( b + \beta X_t\right)^{\top} (\sigma \sigma^{\top})^{-1} dZ_t(X) 
        \\ &\qquad \qquad - \frac{1}{2} \int_0^T \left( b + \beta X_t\right)^{\top} (\sigma \sigma^{\top})^{-1} \left( b + \beta X_t\right) dt.
    \end{align*}
    The maximum likelihood estimator for $(b,\beta)$ is defined as the maximum of the likelihood ratio $\ell_T(b, \beta; X)$. Remark that the drift is a linear function in $(b,\beta)$, whence $\ell_T$ is a quadratic form in $(b,\beta)$. Hence, to maximise it, it suffices to solve the first-order conditions $\nabla \ell_T(b,\beta; X) = 0$. Using a linear expansion for the perturbation $\ell_T(b + b_0, \beta + \beta_0; X)$ allows us to determine its derivative. Namely, we obtain
\begin{align*}
    \ell_T(b + b_0, \beta + \beta_0; X) &= \ell_T(b,\beta; X) + \int_0^T \left( b_0 + \beta_0 X_t\right)^{\top} (\sigma \sigma^{\top})^{-1} dZ_t(X) 
    \\ &\qquad \qquad - \int_0^T \left( b_0 + \beta_0 X_t\right)^{\top} (\sigma \sigma^{\top})^{-1} \left( b + \beta X_t\right) dt
    \\ &\qquad \qquad - \frac{1}{2} \int_0^T \left( b_0 + \beta_0 X_t\right)^{\top} (\sigma \sigma^{\top})^{-1} \left( b_0 + \beta_0 X_t\right) dt.
\end{align*}
Hence, $\nabla \ell_T(b,\beta; X) = 0$ is equivalent to
\[
    \int_0^T \left( b_0 + \beta_0 X_t\right)^{\top} (\sigma \sigma^{\top})^{-1} dZ_t(X) 
    - \int_0^T \left( b_0 + \beta_0 X_t\right)^{\top} (\sigma \sigma^{\top})^{-1} \left( b + \beta X_t\right) dt = 0
\]
for all $b_0, \beta_0$. Now let us choose, for $j = 1,\dots, d$, $\beta_0$ such that $\beta_0 X_t = a_j X_t^j$, where $a_j \in \R^d$ is some vector. Then we obtain
\begin{align*}
    &\ b_0^{\top}\left(\int_0^T (\sigma \sigma^{\top})^{-1} dZ_t(X) - \int_0^T (\sigma \sigma^{\top})^{-1} \left( b + \beta X_t\right) dt \right)
    \\ &\qquad + a_j^{\top}\left( \int_0^T X_t^{j} (\sigma \sigma^{\top})^{-1} dZ_t(X) 
    - \int_0^T X_t^{j} (\sigma \sigma^{\top})^{-1} \left( b + \beta X_t\right) dt \right) = 0.
\end{align*}
Since $b_0, a_1,\dots, a_d$ are arbitrary, we arrive at the system of equations
\begin{align*}
    0 &= \int_0^T (\sigma \sigma^{\top})^{-1} dZ_t(X) - \int_0^T (\sigma \sigma^{\top})^{-1} \left( b + \beta X_t\right) dt
    \\ 0 &= \int_0^T X_t^{j} (\sigma \sigma^{\top})^{-1} dZ_t(X) 
    - \int_0^T X_t^{j} (\sigma \sigma^{\top})^{-1} \left( b + \beta X_t\right) dt
\end{align*}
whose solution determines the desired MLE $(\widehat{b}_T, \widehat{\beta}_T)$. Solving the first equation for $b$ gives the estimator
\begin{align}\label{eq: b estimator MLE}
    \widehat{b}_T = \frac{(\sigma \sigma^{\top})^{-1}Z_T}{T} - (\sigma \sigma^{\top})^{-1} \widehat{\beta}_T \frac{1}{T}\int_0^T X_t dt.
\end{align}
Solving now the second equation for $\beta$ and then inserting the formula for $\widehat{b}_T$, gives the system of equations $(\sigma \sigma^{\top})^{-1}\widehat{\beta}_T F_T(X)e_j = A_T(X)e_j$ for each $j = 1,\dots, d$. Since $F_T(X)$ is invertible, the vectors $F_T(X)e_j$ form a basis, and hence the above system of equations uniquely determines $\widehat{\beta}_T$, i.e. $(\sigma \sigma^{\top})^{-1}\widehat{\beta}_T F_T(X) = A_T(X)$. This yields the desired expression for $\widehat{\beta}_T$. Inserting the latter into \eqref{eq: b estimator MLE} gives the desired form of the estimator $\widehat{b}_T$ and proves the assertion.
\end{proof} 

From the particular form of $A_T, F_T$, it is clear that, to study the asymptotic behaviour of these estimators, limit theorems such as the Law of Large Numbers play a central role. Thus, based on the notion of asymptotic independence, in Section \ref{sec:3_birkhoff_theorem}, we establish a general result on the Law of Large Numbers for stochastic processes that do not need to satisfy the Markov property. Afterwards, in Section \ref{sec:4_ergodicity_vou}, we study the particular case of the Volterra Ornstein-Uhlenbeck process, while the convergence of the MLE will be treated in Section \ref{sec:5_MLE}. To simplify our arguments and keep this work at a reasonable length, we focus on the one-dimensional case $d = 1$. However, following the ideas provided in \cite{MR2471289}, we expect that also the multi-variate case could be studied.

\section{A simple uniform Birkhoff Theorem}\label{sec:3_birkhoff_theorem}

In this section, we provide an abstract set of conditions that allow us to establish the Law of Large Numbers for a stochastic process that is not necessarily a Markov process. The latter allows us to prove the consistency of the method of moments, but is also crucial for the study of consistency and asymptotic normality of the MLE. Let $\mathbb{T} = [0,\infty)$ with $\tau(ds) = ds$ or $\mathbb{T} = \N_0$ with $\tau$ the counting measure on $\mathbb{T}$. The choice $\tau(ds) = ds$ corresponds to continuous-time averages, while the other case captures the case of discrete-time averages. 

\begin{Theorem}\label{lemma: Birkhoff ergodic theorem}
 Let $\Theta \subset \R^d$ be a parameter set, $(\Omega, \F, \P_{\theta})_{\theta \in \Theta}$ a family of probability spaces, and let $(Y_t)_{t \in \mathbb{T}}$ be a $\R$-valued measurable process on $(\Omega, \F, \P_{\theta})$ for each $\theta \in \Theta$. Suppose that $\sup_{\theta \in \Theta}\sup_{t \in \mathbb{T}}\E_{\theta}[|Y_t|^2] < \infty$ and there exists $m: \Theta \longrightarrow \R$ such that
 \begin{align}\label{eq: Birkhoff first moment}
  \lim_{t \to \infty}\sup_{\theta \in \Theta}\left|\frac{1}{t}\int_{[0,t)} \E_{\theta}[Y_s]\tau(ds) - m(\theta)\right| = 0.
 \end{align}
 If $(Y_t)_{t \in \mathbb{T}}$ has asymptotically vanishing covariance in the ceasaro sense uniformly in $\theta \in \Theta$, i.e.,
 \begin{align}\label{eq: asymptotic covariance}
  \lim_{t \to \infty}\sup_{\theta \in \Theta}\frac{1}{t^2}\int_{[0,t)^2} \mathrm{cov}_{\theta}(Y_s, Y_u) \tau(du)\tau(ds) = 0,
 \end{align}
 then $Y$ satisfies the Birkhoff ergodic theorem in the sense that
  \begin{align*}
   \lim_{t \to \infty} \sup_{\theta \in \Theta}\E_{\theta}\left[\left|\frac{1}{t}\int_{[0,t)} Y_s \tau(ds) - m(\theta)\right|^2 \right] = 0.
  \end{align*}
  Moreover, if there exist $p \in [2,\infty)$ and $\e > 0$ such that 
  \begin{align}\label{eq: moment bound}
    \sup_{\theta \in \Theta}\sup_{t \in \mathbb{T}} \E_{\theta}[|Y_t|^{p+\e}] < \infty
  \end{align}
  then even
  \begin{align*}
   \lim_{t \to \infty} \sup_{\theta \in \Theta}\E_{\theta}\left[\left|\frac{1}{t}\int_{[0,t)} Y_s \tau(ds) - m(\theta)\right|^p \right] = 0.
  \end{align*}
\end{Theorem}
\begin{proof}
    We only prove the continuous time case. First note that 
    \[
     \frac{1}{t}\int_0^t Y_sds - m(\theta) = \frac{1}{t}\int_0^t (Y_s - \E_{\theta}[Y_s])ds + \frac{1}{t}\int_0^t \E_{\theta}[Y_s]ds - m(\theta),
    \]
    and hence
    \begin{align*}
        \left( \frac{1}{t}\int_0^t Y_sds - m(\theta) \right)^2
        &= \left( \frac{1}{t}\int_0^t (Y_s - \E_{\theta}[Y_s])ds \right)^2 
        \\ &\qquad + 2 \left( \frac{1}{t}\int_0^t (Y_s - \E_{\theta}[Y_s])ds \right) \left( \frac{1}{t}\int_0^t \E_{\theta}[Y_s]ds - m \right) 
        \\ &\qquad + \left( \frac{1}{t}\int_0^t \E_{\theta}[Y_s]ds - m \right)^2.
    \end{align*}
    Taking expectations and noting that the second term equals zero, we arrive at 
    \begin{align}\label{eq: Birkhoff formula}
    \E_{\theta}\left[ \left|\frac{1}{t}\int_{0}^t Y_sds - m(\theta) \right|^2  \right]
    = \frac{1}{t^2}\int_0^t \int_0^t \mathrm{cov}_{\theta}(Y_u, Y_s)du ds + \left( \frac{1}{t}\int_0^t \E_{\theta}[Y_u] du - m \right)^2.
    \end{align}
    The latter expression tends to zero, uniformly in $\theta \in \Theta$. For the second assertion, let $Z_t = \frac{1}{t}\int_{[0,t)} Y_s \tau(ds)$. Then 
    \begin{align*}
        \E_{\theta}\left[\left|Z_t - m(\theta)\right|^p \right] &\leq R^{p-2} \sup_{\theta \in \Theta}\E_{\theta}\left[ |Z_t - m(\theta)|^2 \right] 
        + \E_{\theta}\left[\1_{\{|Z_t - m(\theta)| > R\}} |Z_t - m(\theta)|^p \right].
    \end{align*}
    The first term converges, for fixed $R$, to zero as $t \to \infty$. To bound the second term, note that $\{|Z_t - m| > R\} \subset \{|Z_t| + |m(\theta)| > R\} \subset \{|Z_t| > R/2 \} \cup \{ |m(\theta)| > R/2\}$. It is easy to see that $\sup_{\theta \in \Theta}|m(\theta)| < \infty$. Hence we may take $R$ large enough (independent of $\theta$) such that $\{ |m(\theta)| > R/2\} = \emptyset$ holds for each $\theta \in \Theta$. This gives
    \begin{align*}
        \E_{\theta}\left[\1_{\{|Z_t - m(\theta)| > R\}} |Z_t - m(\theta)|^p \right]
        &\leq 2^{p-1}\E_{\theta}\left[\1_{\{|Z_t| > R/2\}} |Z_t|^p \right]
        \\ &\qquad + 2^{p-1}\left(\sup_{\theta \in \Theta}|m(\theta)|^p \right) \E_{\theta}\left[\1_{\{|Z_t| > R/2\}} \right]
        \\ &\leq 2^{p-1}\left( \frac{2^{\e}}{R^{\e}} + \frac{2^{p + \e}}{R^{p + \e}}\sup_{\theta \in \Theta}|m(\theta)|^p\right) \sup_{\theta \in \Theta}\sup_{t \in \mathbb{T}}\E_{\theta}[|Z_t|^{p+\e}].
    \end{align*}
    To summarize, letting $t \to \infty$, we obtain for each fixed $R$ sufficiently large
    \begin{align*}
        \limsup_{t \to \infty}\E_{\theta}\left[\left|Z_t - m(\theta)\right|^p \right] 
        \leq 2^{p-1}\left( \frac{2^{\e}}{R^{\e}} + \frac{2^{p + \e}}{R^{p + \e}}\sup_{\theta \in \Theta}|m(\theta)|^p\right) \sup_{\theta \in \Theta}\sup_{t \in \mathbb{T}}\E_{\theta}[|Z_t|^{p+\e}].
    \end{align*}
    Letting now $R \to \infty$, proves the assertion.
\end{proof}

Here and below we work with weak convergence uniformly in a given set of parameters $\Theta$. The precise definition and some related properties are given in the appendix. Condition \eqref{eq: asymptotic covariance} can be seen as a variant of mixing conditions often studied for stationary processes, see \cite{MR2178042} for an overview on this topic. 

\begin{Remark}\label{Remark sufficiency}
    Suppose that there exists a family of probability measures $\pi_{\theta}$ on $\R$ such that $\mathcal{L}_{\theta}(Y_t) \Longrightarrow \pi_{\theta}$ for $t \to \infty$, and $\mathcal{L}_{\theta}(Y_s,Y_t) \Longrightarrow \pi_{\theta} \otimes \pi_{\theta}$ as $s,t \to \infty$ with $|s-t| \to \infty$ uniformly on $\Theta$. Moreover, suppose that \eqref{eq: moment bound} holds. Then the assumptions of Lemma \ref{lemma: Birkhoff ergodic theorem} are satisfied with $m(\theta) = \int_{\R}x \pi_{\theta}(dx)$.
\end{Remark}
\begin{proof}
    Since $\mathcal{L}_{\theta}(Y_t) \Longrightarrow \pi_{\theta}$ and $\mathcal{L}_{\theta}(Y_s,Y_t) \Longrightarrow \pi_{\theta} \otimes \pi_{\theta}$, in view of the uniform moment bound \eqref{eq: moment bound}, Proposition \ref{prop: P characterization} gives $\E_{\theta}[Y_s] \longrightarrow m(\theta)$ as $s \to \infty$ uniformly in $\theta \in \Theta$ and $\mathrm{cov}_{\theta}(Y_s,Y_t) \longrightarrow 0$ uniformly in $\theta$. Thus, conditions \eqref{eq: Birkhoff first moment} and \eqref{eq: asymptotic covariance} are satisfied. 
\end{proof}

An application of Theorem \ref{lemma: Birkhoff ergodic theorem} and the above remark for the particular choice $Y_t = f(X_t)$, allows us to deduce the Law of Large Numbers for a given stochastic process $X$ and a suitable class of test functions $f$. To verify the assumptions \eqref{eq: Birkhoff first moment} and \eqref{eq: asymptotic covariance}, it is natural to show that $X$ is \textit{asymptotically independent} in the sense of \eqref{eq: asymptotic independence}. Let us remark that Theorem \ref{lemma: Birkhoff ergodic theorem} and Remark \ref{Remark sufficiency} allow for weak convergence uniformly in the parameter set $\Theta$, in which case we obtain a uniform Law of Large Numbers. For other related results in the direction of Birkhoff theorems for Markov processes, we refer to \cite{MR3684455}. 

Thus, to establish a Law of Large Numbers, we focus on the verification of the asymptotic independence \eqref{eq: asymptotic independence} probably uniformly over $\Theta$. For Gaussian processes, or more generally, affine processes, such a property may be verified through the convergence of the characteristic function (see Section \ref{sec:4_ergodicity_vou} for the VOU process). Furthermore, if $X$ is a Markov process, asymptotic independence follows from the convergence in the Wasserstein distance, which is a well-studied property for classical Markov processes. For non-affine Volterra processes, property \eqref{eq: asymptotic independence} may be reduced to the Markov case by considering Markovian lifts.

\begin{Lemma}\label{lemma: strong consistency}
    Let $(Y_t)_{t \geq 0}$ be a stochastic process on $\R$. Suppose that there exists $p \in [1,\infty)$ and $a \in \R$ such that $\sup_{t \geq 0}\E[|Y_t|^p] < \infty$ and
    $\frac{1}{t}\int_0^t Y_s ds \longrightarrow a$ in $L^p(\Omega)$. Then 
    \[
     \P\left[ \lim_{n \to \infty}\frac{1}{t_n}\int_0^{t_n}Y_s ds = a \right] = 1
    \] 
    holds for each sequence $(t_n)_{n \geq 1}$ that satisfies $t_n \nearrow \infty$ and
    \begin{align*}
     \sum_{n=1}^{\infty}\left| 1 - \frac{t_n}{t_{n+1}}\right|^p < \infty.
    \end{align*}
\end{Lemma}
\begin{proof}
    Define $Z_t = \frac{1}{t}\int_0^t Y_s ds - a$. Then it suffices to show that $Z_t \longrightarrow 0$ a.s. For $s,t > 0$ we obtain 
    \[
     Z_{t} - Z_s = \frac{1}{t}\int_{s}^{t}Y_r dr + \left( \frac{s}{t} - 1\right) \frac{1}{s}\int_0^{s}Y_r dr.
    \]
    Let $A = \sup_{r \geq 0}\E\left[|Y_r|^p\right] < \infty$.
    The first term can be bounded according to
    \begin{align*}
        \E\left[ \left|  \frac{1}{t}\int_{s}^{t}Y_r dr \right|^p \right]
        &\leq t^{-p} (t-s)^{p-1}\int_{s}^{t}\E\left[|Y_r|^p\right] dr 
        \\ &\leq t^{-p}(t-s)^p A.
    \end{align*}
    For the second term, we use the Jensen inequality to find that 
    \begin{align*}
     \left| \frac{s}{t} - 1\right|^p\E\left[ \left|\frac{1}{s}\int_0^{s}Y_r dr \right|^p \right]
     &\leq \left| \frac{s}{t} - 1\right|^p\frac{1}{s} \int_0^{s}\E\left[ |Y_r|^p \right] dr
     \\ &\leq t^{-p}(t-s)^pA.
    \end{align*}
    Hence we arrive at $\E\left[ |Z_t - Z_s|^p \right]         \leq 2^p t^{-p}|t-s|^pA$. Let $t_n \nearrow \infty$ be given as in the assumptions. Then
    \begin{align*}
     \sum_{n=1}^{\infty}\E\left[ |Z_{t_{n+1}} - Z_{t_n}|^p \right]
     \leq 2^p A \sum_{n=1}^{\infty} t_{n+1}^{-p}|t_{n+1} - t_n|^p
     = 2^p A \sum_{n=1}^{\infty}\left| 1 - \frac{t_n}{t_{n+1}}\right|^p < \infty.
    \end{align*}
    Hence $Z := \sum_{n=1}^{\infty}(Z_{t_{n+1}} - Z_{t_n})$ is $\P$-a.s. convergent. Finally, since $Z_{t_{n+1}} = Z_{t_1} + \sum_{j=1}^{n}(Z_{t_{j+1}} - Z_{t_j})$ and the right-hand side is $\P$-a.s. convergent, also $Z_{t_{n+1}}$ is $\P$-a.s. convergent. Since $Z_{t_{n+1}} \to 0$ in $L^p$, uniqueness of a.s. limits yield the assertion.      
\end{proof}

To conclude this section, let us remark that the established Law of Large Numbers allows us to prove the consistency for the \textit{generalised method of moments}, provided that the stochastic Volterra process of interest satisfies the conditions of Theorem \ref{lemma: Birkhoff ergodic theorem} and allows to express the moments of the process semi-explicitly. In the next section, we verify the conditions of Theorem \ref{lemma: Birkhoff ergodic theorem} for the Volterra Ornstein-Uhlenbeck process and demonstrate that the corresponding method-of-moments is consistent.

\section{Ergodicity for the Volterra Ornstein-Uhlenbeck process}\label{sec:4_ergodicity_vou}

\subsection{Ergodic regime}

In this section, we study the ergodic properties of the one-dimensional Volterra Ornstein-Uhlenbeck process given by \eqref{eq: VOU} under the following set of conditions. 
\begin{assumption}\label{VOU}
     Let $b, x_0 \in \R$, $\beta < 0$, and $\sigma > 0$. Let $0 \leq K \in L_{loc}^2(\R_+)$ be completely monotone, and assume that there exists $\gamma \in (0,1]$ with the property
    \begin{align}\label{eq: K global regularity}
        \sup_{0 \leq t-s \leq 1}(t-s)^{-2\gamma}\left(\int_s^t |K(r)|^2 dr + \int_0^{\infty}|K((t-s) + r) - K(r)|^2 dr\right) < \infty.
    \end{align}    
\end{assumption}
Let $X$ be the VOU process given by \eqref{eq: VOU}. Then, under assumption \eqref{eq: K global regularity}, an application of \cite{MR4019885} shows that $X$ has for each $\e > 0$ a modification with $\gamma - \e$ H\"older continuous sample paths. Similarly to the non-Volterra case, the VOU process admits an explicit solution formula as described below. Let $E_{\beta} \in L_{loc}^2(\R_+)$ be the unique solution of the linear Volterra equation
\begin{align}\label{eq: Ebeta definition}
 E_{\beta}(t) = K(t) + \int_0^t K(t-s)\beta E_{\beta}(s)ds.
\end{align}
whose existence is guaranteed by \cite[Theorem 3.5]{MR1050319}. It is easy to see that, by Young's inequality, this function satisfies $E_{\beta} \in C((0,\infty))$. Note that $R_{\beta}(t) := E_{\beta}(t)(-\beta)$ is the resolvent of the second kind of $K (-\beta)$, i.e. is the unique solution of $R_{\beta} = K(-\beta) + K(-\beta) \ast R_{\beta}$, see \cite[Chapter 2]{MR1050319}. An application of the variation of constants formula for Volterra equations (see e.g. \cite[Lemma 2.5]{MR4019885}), and invoking relation $R_{\beta} = E_{\beta}(-\beta)$ shows that $X$ given by \eqref{eq: VOU} also solves the equivalent equation  
\begin{align}\label{eq: VOU Ebeta formulation}
 X_t = \left( 1 + \int_0^t E_{\beta}(s)\beta\ ds \right)x_0 + \left(\int_0^t E_{\beta}(s)ds\right)b + \int_0^t E_{\beta}(t-s)\sigma dB_s.
\end{align}
In particular, $X$ is a Gaussian process with mean and covariance structure given by
\begin{align*}
    \E[X_t] &= \left( 1 + \int_0^t E_{\beta}(s)\beta\ ds \right)x_0 + \left(\int_0^t E_{\beta}(s)ds\right)b,
    \\ \mathrm{cov}(X_t,X_s) &= \int_0^{s \wedge t} E_{\beta}(|t-s|+r)\sigma \sigma^{\top}E_{\beta}(r)^{\top} \ dr.
\end{align*} 

Let us write $\|K\|_{L^1(\R_+)} = \int_0^{\infty}K(t)dt \in [0,\infty]$. Below, we summarise some useful results for the Gaussian process $X$. 

\begin{Lemma}\label{prop: Ebeta uniform bounds}
    Suppose that assumption \ref{VOU} is satisfied. Then the following assertions hold:
    \begin{enumerate}
        \item[(a)] $E_{\beta} \in L^1(\R_+) \cap L^2(\R_+)$ is completely monotone and
        \begin{align}\label{eq: Ebeta integral}
            \int_0^{\infty}E_{\beta}(t)dt = \frac{1}{\|K\|_{L^1(\R_+)}^{-1} + |\beta|}
        \end{align}
        with the convention that $1/\infty := 0$. Moreover, for each $N \geq 1$ it holds that
        \[
        \lim_{t \to \infty}\sup_{\beta \leq -1/N}\left( \int_t^{\infty}E_{\beta}(r)dr + \int_t^{\infty}E_{\beta}(r)^2dr\right) = 0.
        \]
        \item[(b)] Let $\Theta \subset \R \times (-\infty,0)$ be a compact. Then the limits 
        \begin{align}
            m_1(b,\beta) &:= \lim_{t \to \infty}\E[X_t] = \frac{x_0}{1+\|K\|_{L^1(\R_+)}|\beta|} + \frac{b}{|\beta|}, \label{eq: first moment VOU}
            \\ m_2(b, \beta) &:= \lim_{t \to \infty} \E[X_t^2] = (m_1(b,\beta))^2 + \sigma^2 \int_0^{\infty}E_{\beta}(r)^2 dr \notag
        \end{align}
        exist locally uniformly in $(b,\beta) \in \Theta$.    
    \end{enumerate}
\end{Lemma}
\begin{proof}
    Since $K \neq 0$ is completely monotone, it follows that second kind $R_{\beta} \in L^1(\R_+)$ is completely monotone, see \cite[Chapter 5, Theorem 3.1]{MR1050319}. Thus, since $(-\beta)^{-1} > 0$, also $E_{\beta} = (-\beta)^{-1}R_{\beta}$ is completely monotone and integrable. Since $E_{\beta}$ is nonincreasing, it follows that $E_{\beta} \in L^2(\R_+)$. To compute the integral of $E_{\beta}$, we take Laplace transforms of the equation $E_{\beta} = K + \beta K \ast E_{\beta}$, which gives $\widehat{E}_{\beta} = \widehat{K} + \beta \widehat{K} \widehat{E}_{\beta}$ and hence
    \[
        \widehat{E}_{\beta}(z) = \frac{\widehat{K}(z)}{1 - \beta \widehat{K}(z)} = \frac{1}{\widehat{K}(z)^{-1} - \beta}, \qquad z > 0.
    \]
    Since $\widehat{K}(z) > 0$ by complete monotonicity, we may take the limit $z \searrow 0$ to find 
    \[
        \int_0^{\infty} E_{\beta}(t) dt = \frac{1}{\left(\int_0^{\infty}K(t)dt\right)^{-1} + |\beta|}
    \]
    which proves \eqref{eq: Ebeta integral}. Since $E_{\beta} \geq 0$ for $\beta < 0$, an application of \cite[Lemma A.6]{BBF23} implies that $E_{\beta} \leq E_{\widetilde{\beta}}$ for $\widetilde{\beta} < \beta < 0$. Thus we obtain
    \begin{align*}
        \int_t^{\infty}E_{\beta}(r)dr + \int_t^{\infty}E_{\beta}(r)^2dr
        \leq \left(1 + E_{\beta_0}(t)\right) \int_t^{\infty}E_{\beta_0}(r)dr
    \end{align*}
    for $\beta_0 = -N^{-1} < 0$ where we have used that $E_{\beta_0}$ is nonincreasing. This proves assertion (a). Assertion (b) follows directly from the explicit form of $\E[X_t]$, formula $\E[X_t^2] = \mathrm{cov}(X_t,X_t) - (\E[X_t])^2$, and an application of part (a).
\end{proof}

Next, we introduce the corresponding stationary process under assumption \ref{VOU}. Let $(B'_t)_{t \geq 0}$ be another standard Brownian motion independent of $(B_t)_{t \geq 0}$ and define $W_t = \1_{[0,\infty)}(t)B_t + \1_{(-\infty, 0)}(t)B_t'$. Then 
\begin{align}\label{eq: VOU stationary process}
 X^{\mathrm{stat}}_t = m_1(b,\beta) + \sigma\int_{-\infty}^t E_{\beta}(t-s)dW_s, \qquad t \in \R
\end{align}
determines a strictly stationary Gaussian process with mean $m_1(b,\beta)$ and covariance structure
\[
 \mathrm{cov}(X^{\mathrm{stat}}_t, X_s^{\mathrm{stat}}) = \sigma^2 \int_0^{\infty}E_{\beta}(|t-s|+r)E_{\beta}(r)dr.
\]
Remark that, since the VOU process is a Gaussian process, its distributional properties are fully determined by its characteristic function and hence in terms of its mean and covariance structure. Thus, the stationary process is necessarily unique in the sense that its mean $m_1(b,\beta)$ and covariance structure $\mathrm{cov}(X_t^{\mathrm{stat}}, X_s^{\mathrm{stat}})$ is uniquely determined by the limits derived in Lemma 4.2.(b). However, remark that if $K \in L^1(\R_+)$, then the mean $m_1(b,\beta)$ and hence also the stationary process depends on the initial state $x_0$ which reflects the non-markovian nature of the dynamics. Below, we prove another auxiliary result on uniform bounds and convergence to equilibrium.

\begin{Lemma}\label{prop: VOU technical stuff}
 Suppose assumption \ref{VOU} holds. Then for each $p \in [2,\infty)$ there exists a locally bounded function $C_{p}(b,\beta) > 0$ in $(b,\beta) \in \R \times (-\infty, 0)$ such that
 \begin{align}\label{eq: bounded moments VOU}
  \sup_{t \geq 0}\E\left[ |X_t|^p + |X_t^{\mathrm{stat}}|^p \right] \leq C_p(b,\beta) < \infty
 \end{align}
 and for all $0 \leq t-s \leq 1$
 \begin{align}\label{eq: increment}
  \E[|X^{\mathrm{stat}}_t - X^{\mathrm{stat}}_s|^p] +  \E[|X_t - X_s|^p] \leq C_p(b,\beta)(t-s)^{p\gamma}.
 \end{align}
 In particular, for each $\eta \in (0, \gamma)$, the processes $X, X^{\mathrm{stat}}$ have a modification with $\eta$-H\"older continuous sample paths. Finally, it holds that 
 \begin{align}\label{eq: weak convergence stationary}
            (X_{t+h})_{t \geq 0} \Longrightarrow (X_t^{\mathrm{stat}})_{t \geq 0}, \qquad h \to \infty
 \end{align}
 weakly on the path space of continuous functions.    
\end{Lemma}

The proof of this lemma is postponed to the appendix. At this point let us remark that the complete monotonicity allows us to effectively verify that $E_{\beta} \in L^1(\R_+) \cap L^2(\R_+)$ whenever $\beta < 0$. The latter is crucial to carry out the asymptotic analysis, construct the corresponding stationary process, and establish the Law-of-Large numbers.

\subsection{Law of Large Numbers and mixing}

In this section, we prove the Law of Large Numbers and show that the stationary process is mixing. As a first step, we establish the asymptotic independence for $X$ and $X^{\mathrm{stat}}$.   

\begin{Lemma}\label{lemma: VOU asymptotic independence}
    Suppose that assumption \ref{VOU} is satisfied. Then the following assertions hold:
    \begin{enumerate}
        \item[(a)] $\P_{b,\beta}$ and $\P_{b,\beta}^{\mathrm{stat}}$ are asymptotically independent in the sense that, for each $n \in \N$, 
        \[
         (X_{t_1},\dots, X_{t_n}), \ \ (X_{t_1}^{\mathrm{stat}}, \dots, X_{t_n}^{\mathrm{stat}}) \Longrightarrow \pi_{b,\beta}^{\otimes n} 
        \]
        weakly as $t_1, \dots, t_n \to \infty$ with $\min_{k=1, \dots, n-1}|t_{k+1}-t_k| \to \infty$ locally uniformly in $(b,\beta) \in \R \times (-\infty, 0)$. Here $\pi_{b,\beta}$ denotes the Gaussian distribution with mean $m_1(b,\beta)$ and variance 
        \[
            m_{\mathrm{var}}(\beta) :=  \sigma^2 \int_0^{\infty}E_{\beta}(r)^2 dr = m_2(b,\beta) - (m_1(b,\beta))^2.
        \]

        \item[(b)] For all $n, N \in \N$, $0 \leq t_1 < \dots < t_n$ and $0 \leq s_1 < \dots < s_N$, one has
        \begin{align}\label{eq: VOU mixing weak convergence}
            (X^{\mathrm{stat}}_{t_1}, \dots, X^{\mathrm{stat}}_{t_n}, X^{\mathrm{stat}}_{s_1 +h}, \dots, X^{\mathrm{stat}}_{s_N + h}) \Longrightarrow (X^{\mathrm{stat}}_{t_1}, \dots, X^{\mathrm{stat}}_{t_n}, \overline{X}_{s_1}, \dots, \overline{X}_{s_N})
        \end{align}
        weakly as $h \to \infty$, where $\overline{X}$ is an independent copy of $X^{\mathrm{stat}}$.
    \end{enumerate}
\end{Lemma}
\begin{proof}
 The mean of $(X_{t_1},\dots, X_{t_n})$ satisfies 
 \[
  (\E[X_{t_1}], \dots, \E[X_{t_n}])^{\top} \longrightarrow (m_1(b,\beta), \dots, m_1(b,\beta)), \qquad t_1,\dots, t_n \to \infty
 \]
 locally uniformly in $(b,\beta)$ due to Lemma \ref{prop: Ebeta uniform bounds}. If $j \neq k$, then its covariance matrix satisfies by the Cauchy-Schwarz inequality and dominated convergence
 \begin{align*}
  |\mathrm{cov}(X_{t_j}, X_{t_k})| \leq \sigma^2 \int_0^{\infty}E_{\beta}(r)^2dr \int_{|t_j - t_k|}^{\infty}E_{\beta}(r)^2 dr \longrightarrow 0
 \end{align*}
 as $|t_j - t_k| \to \infty$. For $j = k$, we obtain
 \[
  \mathrm{cov}(X_{t_j}, X_{t_j}) = \sigma^2 \int_0^{t_j} E_{\beta}(r)^2dr \longrightarrow \sigma^2 \int_0^{\infty}E_{\beta}(r)^2 dr.
 \]
 Since the last two integrals converge locally uniformly in $\beta$ due to Lemma \ref{prop: Ebeta uniform bounds}, the desired weak convergence locally uniformly in $(b,\beta)$ follows from the uniform version of L\'evys continuity theorem, see Theorem \ref{thm: uniform Levy continuity theorem}. The case of $X^{\mathrm{stat}}$ can be shown in the same way.

 (b) As both sides are Gaussian random variables, it suffices to prove the convergence of the mean and covariance. Convergence of the mean is evident, for the covariance we first note that for $h$ sufficiently large one has
 \begin{align*}
  \mathrm{cov}(X_{t_j}^{\mathrm{stat}}, X_{s_k+h}^{\mathrm{stat}})
  = \sigma^2 \int_0^{\infty} E_{\beta}((s_k+h -t_j)+r)E_{\beta}(r)dr \longrightarrow 0
 \end{align*}
 as $h \to \infty$ by Cauchy-Schwarz and dominated convergence theorem. Hence the covariance matrix of $(X^{\mathrm{stat}}_{t_1}, \dots, X^{\mathrm{stat}}_{t_n}, X^{\mathrm{stat}}_{s_1 +h}, \dots, X^{\mathrm{stat}}_{s_N + h})$ satisfies
 \begin{align*}
  \lim_{h \to \infty}&\ \begin{pmatrix} \left(\mathrm{cov}(X_{t_j}^{\mathrm{stat}}, X_{t_k}^{\mathrm{stat}} \right)_{1 \leq j,k \leq n} & \left( \mathrm{cov}(X_{t_j}^{\mathrm{stat}}, X_{s_k+h}^{\mathrm{stat}} \right)_{1 \leq j \leq n, \ 1 \leq k \leq N} 
  \\ \left( \mathrm{cov}(X_{t_j}^{\mathrm{stat}}, X_{s_k+h}^{\mathrm{stat}} \right)_{1 \leq j \leq n, \ 1 \leq k \leq N}  & \left( \mathrm{cov}(X_{s_j+h}^{\mathrm{stat}}, X_{s_k+h}^{\mathrm{stat}} \right)_{1 \leq j,k \leq N} \end{pmatrix}
  \\ &= \begin{pmatrix} \left(\mathrm{cov}(X_{t_j}^{\mathrm{stat}}, X_{t_k}^{\mathrm{stat}} \right)_{1 \leq j,k \leq n} & 0 
  \\ 0  & \left( \mathrm{cov}(X_{s_j}^{\mathrm{stat}}, X_{s_k}^{\mathrm{stat}} \right)_{1 \leq j,k \leq N} \end{pmatrix}
 \end{align*}
 Since the right-hand side is the covariance matrix of $(X^{\mathrm{stat}}_{t_1}, \dots, X^{\mathrm{stat}}_{t_n}, \overline{X}_{s_1}, \dots, \overline{X}_{s_N})$, property \eqref{eq: VOU mixing weak convergence} is proved.
\end{proof}

The above result allows us to apply a version of Birkhoff's ergodic theorem as stated in Theorem \ref{lemma: Birkhoff ergodic theorem}, and to show that the stationary process is mixing. In particular, the stationary process is ergodic in the sense that its shift-invariant $\sigma$-algebra is trivial.

\begin{Theorem}\label{thm: law of large numbers VOU}
    Suppose that assumption \ref{VOU} holds. Then $\P_{b,\beta}^{\mathrm{stat}}$ is mixing. Let $\pi_{b,\beta}$ be the Gaussian distribution on $\R$ with mean $m_1(b,\beta)$ and variance $m_{\mathrm{var}}(\beta)$, and let $f: \R \longrightarrow \R$ be polynomially bounded and continuous a.e. except for a set of Lebesgue measure zero. Then, for each $p \in [2,\infty)$,
    \begin{align}\label{eq: Birkhoff VOU}
     &\ \lim_{t \to \infty}\E\left[ \left|\frac{1}{t}\int_0^t f(X_s)ds - \int_{\R} f(x) \pi_{b,\beta}(dx)\right|^p \right] = 0,
     \\ &\ \lim_{t \to \infty}\E\left[ \left| \frac{1}{t}\int_0^t f(X_s^{\mathrm{stat}})ds - \int_{\R}f(x)\pi_{b,\beta}(dx)\right|^p \right] = 0. \notag
    \end{align}
    holds locally uniformly in $(b,\beta) \in \R \times (-\infty, 0)$.
\end{Theorem}
\begin{proof}
 Let us first show that $\P_{b,\beta}^{\mathrm{stat}}$ is mixing, i.e., 
 \[
  \lim_{h \to \infty}\P_{b,\beta}^{\mathrm{stat}}[ A \cap T_h^{-1}(B)] = \P_{b,\beta}^{\mathrm{stat}}[A]\P_{b,\beta}^{\mathrm{stat}}[B]
 \]
 holds for all Borel sets $A,B$ on the path space $C(\R_+)$ where $T_h(x) = x(h + \cdot)$ denotes the shift operator on $C(\R_+)$. By standard Dynkin arguments, it suffices to prove this convergence for a $\cap$-stable system of sets that generate the Borel-$\sigma$-algebra on $C(\R_+)$. Let us consider, for $n \in \N$, $t_1,\dots, t_n \geq 0$, and $a_j < b_j$, $j = 1,\dots,n$, the set 
 \[
  A = \{ x \in C(\R_+) \ : \ (x_{t_1}, \dots, x_{t_n}) \in (a_1,b_1] \times \dots (a_n, b_n] \},
 \]
 and similarly, for $N \in \N$, $s_1, \dots, s_N \geq 0$ and $a_j' < b_j'$, the set
 \[
  B = \{ x \in C(\R_+) \ : \ (x_{s_1}, \dots, x_{s_N}) \in (a'_1,b'_1] \times \dots (a'_n, b'_n] \}.
 \]
 For this choice of sets, the assertion becomes 
 \begin{align*}
  &\ \lim_{h \to \infty}\P_{b,\beta}^{\mathrm{stat}}\left[(x_{t_1},\dots, x_{t_n}; x_{s_1+h}, \dots, x_{s_N+h}) \in \prod_{j=1}^n (a_j,b_j] \times \prod_{l=1}^N (a_l',b_l'] \right]
  \\ &= \P_{b,\beta}^{\mathrm{stat}}\left[(x_{t_1},\dots, x_{t_n}) \in \prod_{j=1}^n (a_j,b_j] \right]\P_{b,\beta}^{\mathrm{stat}}\left[(x_{s_1}, \dots, x_{s_N}) \in \prod_{l=1}^N (a_l',b_l'] \right].
 \end{align*}
 Because of \eqref{eq: VOU mixing weak convergence} combined with the Portmanteau theorem, the above convergence is satisfied provided that  
 \[
  \P_{b,\beta}^{\mathrm{stat}}\left[(x_{t_1},\dots, x_{t_n}) \in \partial \prod_{j=1}^n (a_j,b_j] \right] = \P_{b,\beta}^{\mathrm{stat}}\left[(x_{s_1}, \dots, x_{s_N}) \in \partial \prod_{l=1}^N (a_l',b_l'] \right] = 0.
 \]
 The latter is satisfied since the finite-dimensional distributions of $\P_{b,\beta}^{\mathrm{stat}}$ are absolutely continuous with respect to the Lebesgue measure, see \cite{F24}. This proves that $\P_{b,\beta}^{\mathrm{stat}}$ is mixing. 

 Next, we show that the conditions of Theorem \ref{lemma: Birkhoff ergodic theorem} are satisfied for $Y_t = f(X_t)$. By Lemma \ref{lemma: VOU asymptotic independence} we obtain $X_t \Longrightarrow \pi_{b,\beta}$ and $(X_s,X_t) \Longrightarrow \pi_{b,\beta} \otimes \pi_{b,\beta}$ locally uniformly in $(b,\beta)$. Since $m_{\mathrm{var}}(\beta) > 0$ due to $\sigma > 0$ and $K \neq 0$ so that $E_{\beta} \neq 0$, the Gaussian measures $\pi_{b,\beta}$ and $\pi_{b,\beta}\otimes \pi_{b,\beta}$ are absolutely continuous with respect to the Lebesgue measure. By abuse of notation, let $\pi_{b,\beta}(x)$ and $\pi_{b,\beta}(x,y)$ be the corresponding densities. Using the explicit form of these, we find 
 \[
    \int_{\R}\sup_{(b,\beta) \in \Theta}\pi_{b,\beta}(x)dx + \int_{\R^2}\sup_{(b,\beta) \in \Theta}\pi_{b,\beta}(x,y)dxdy < \infty
 \]
 for each compact $\Theta \subset \R \times (-\infty, 0)$. Hence we obtain $Y_t \Longrightarrow \pi_{b,\beta} \circ f^{-1} =: \pi_{b,\beta}^f$ and $(Y_t, Y_s) \Longrightarrow \pi_{b,\beta}^f \otimes \pi_{b,\beta}^f$ uniformly on $\Theta$ by the uniform continuous mapping theorem (see Theorem \ref{thm: uniform CMT}.(c)). Thus, Theorem \ref{lemma: Birkhoff ergodic theorem} is applicable for $Y_t = f(X_t)$ which proves the assertion for the first case. The case $Y_t = f(X^{\mathrm{stat}}_t)$ can be shown in the same way. 
\end{proof}

\begin{Remark}
    A similar statement to \eqref{eq: Birkhoff VOU} also holds in the discrete-time case under the same assumptions.
\end{Remark}

\subsection{Method of moments for fractional kernel}

Let us apply the previously shown law of large numbers to prove consistency for the estimators obtained from the method of moments. More precisely, define 
\begin{align}\label{eq: method of moments}
 m_1(T) = \frac{1}{T}\int_{[0,T)} X_t \tau(dt) \ \text{ and } \ m_2(T) = \frac{1}{T}\int_{[0,T)} X_t^2 \tau(dt)
\end{align}
where $\tau(dt) = dt$ in continuous time, and $\tau(dt)$ is a purely discrete measure for discrete-time observations. The method of moments estimators for $(b,\beta)$ are then obtained as the solutions of $m_1(T) = m_1(\widehat{b}_T, \widehat{\beta}_T)$ and $m_2(T) = m_2(\widehat{b}_T, \widehat{\beta}_T)$. 

Below we focus on the particular case of the fractional kernel for which we can explicitly compute the integral $\int_0^{\infty}E_{\beta}(t)^2dt$. 

\begin{Example}\label{example: fractional}
    Suppose that $K_{\alpha}$ is the fractional Riemann-Lioville kernel of order $\alpha \in (1/2, 1)$ and $\beta < 0$. Then 
    \[
     \int_0^{\infty} E_{\beta}(t)^2 dt = C_{\alpha}|\beta|^{\frac{1}{\alpha} - 2}
    \]
    where the constant $C_{\alpha}$ is given by
    \[
     C_{\alpha} = \frac{1}{\pi}\int_0^{\infty} \frac{du}{1 + 2u^{\alpha}\cos(\frac{\pi \alpha}{2}) + u^{2\alpha}}.
    \]
\end{Example}
\begin{proof}
 The Fourier transform of $E_{\beta}$ is given by
 \[
  \widehat{E}(\mathrm{i}z) = \frac{1}{\widehat{K}_{\alpha}(\mathrm{i}z)^{-1} + |\beta|} = ( e^{\mathrm{i}\pi \alpha/2}|z|^{\alpha} + |\beta|)^{-1}, \qquad z \in \R.
 \]
 The Plancherel identity combined with the substitution $u = |\beta|^{1/\alpha}z$ gives
 \begin{align*}
     \int_0^{\infty} E_{\beta}(t)^2 dt 
    &= \frac{1}{2\pi} \int_{\R} | \widehat{E}_{\beta}(\mathrm{i}z)|^2 dz
    \\ &= \frac{1}{2\pi} \int_{\R} \frac{dz}{||\beta| + |z|^{\alpha}e^{\mathrm{i}\mathrm{sgn}(z)\frac{\pi\alpha}{2}}|^2}
    \\ &= \frac{1}{\pi} \int_{0}^{\infty} \frac{dz}{|\beta|^2 + 2 |\beta|z^{\alpha}\cos(\frac{\pi \alpha}{2}) + z^{2\alpha}}
    \\ &= \frac{|\beta|^{\frac{1}{\alpha}-2}}{\pi} \int_0^{\infty} \frac{du}{1 + 2u^{\alpha} \cos(\frac{\pi \alpha}{2}) + u^{2\alpha}}.
 \end{align*} 
\end{proof}

Thus, for the particular case of the fractional Riemann-Liouville kernel \eqref{eq: fractional kernel} in \eqref{eq: VOU} combined with \eqref{eq: first moment VOU} and Example \ref{example: fractional}, we obtain
\begin{align}\label{eq: moments fractional}
    m_1(b,\beta) = \frac{b}{|\beta|} \ \ \text{ and } \ \ m_2(b,\beta) = \left(\frac{b}{|\beta|}\right)^2 + \sigma^2 C_{\alpha} |\beta|^{\frac{1}{\alpha} - 2}.
\end{align}
Solving these equations gives
\begin{align*}
    \widehat{b}_T &= m_1(T) \left( \frac{C_{\alpha}\sigma^2}{m_2(T) - m_1(T)^2} \right)^{\frac{\alpha}{2\alpha-1}}
    \\ \widehat{\beta}_T &= -\left( \frac{C_{\alpha}\sigma^2}{m_2(T) - m_1(T)^2} \right)^{\frac{\alpha}{2\alpha-1}}.
\end{align*}

The following result applies equally to continuous and discrete observations.

\begin{Corollary}\label{cor: VOU method of moments}
    Suppose that $K$ is given by the fractional kernel \eqref{eq: fractional kernel}, $\sigma > 0$ and $\beta < 0$. Then $(\widehat{b}_T, \widehat{\beta}_T)$ is consistent in probability locally uniform in $(b,\beta)$, i.e.
    \[
     \lim_{T \to \infty}\sup_{(b,\beta) \in \Theta}\P\left[ \left| (\widehat{b}_T, \widehat{\beta}_T) - (b,\beta)\right| > \e\right] = 0
    \]
    holds for each $\e > 0$ and any compact $\Theta \subset \R \times (-\infty, 0)$. Furthermore, if $(T_n)_{n \geq 1}$ is a sequence with $T_n \nearrow \infty$ and there exists $p \in [1,\infty)$ with 
    \begin{align}\label{eq: Tn summability}
     \sum_{n=1}^{\infty}\left| 1 - \frac{T_n}{T_{n+1}}\right|^p < \infty,
    \end{align}
    then $(\widehat{b}_{T_n}, \widehat{\beta}_{T_n})$ is strongly consistent.
\end{Corollary}
\begin{proof}
Consistency locally uniformly in $(b,\beta)$ follows from the uniform Law of Large Numbers given in Theorem \ref{thm: law of large numbers VOU} combined with the uniform continuous mapping theorem Proposition \ref{thm: uniform CMT}.(a) and the continuity of $m_1,m_2$ in \eqref{eq: moments fractional}. Furthermore, an application of Lemma \ref{lemma: strong consistency} implies the desired strong consistency of the estimators. 
\end{proof}

This corollary applies, e.g., for $T_n = n^{\beta}$ with $\beta > 0$ such that $p > 1/\beta$. However, $T_n = q^n$ with $q > 1$ does not satisfy \eqref{eq: Tn summability}.

\section{Maximum likelihood estimation}\label{sec:5_MLE}
\subsection{Continuous observations}

In this section, we study the MLE for the VOU process. To keep this work at a reasonable length, we focus on the univariate case $d = 1$. However, at the expense of more delicate arguments, see \cite{MR2471289} for the classical Ornstein-Uhlenbeck process, it should be possible to study also the multivariate case. Remark that, in dimension $d = 1$, the matrices $A_T(X), F_T(X)$ given in Theorem \ref{lemma: VOU Girsanov} take the form 
\begin{align*}
    A_T(X) &= \frac{\sigma^{-2}}{T} \int_0^T X_t dZ_t - \frac{\sigma^{-2}}{T}\int_0^T X_t dt \frac{Z_T}{T},
    \\ F_{T}(X) &= \frac{1}{T}\int_0^T X_t^2 dt - \left(\frac{1}{T}\int_0^T X_t dt \right)^2,
\end{align*}  
where $Z = Z(X)$ is given by \eqref{eq: Z process}, and hence $dZ_t(X) = (b + \beta X_t)dt + \sigma dB_t$. In particular, $F_T(X) \geq 0$ a.s. by the Cauchy-Schwarz inequality. Since $X$ is not constant due to $\sigma, K \neq 0$, we even have $F_T(X) > 0$ a.s. The corresponding MLE are in such a case given by
\begin{align*}
 \widehat{b}_T(X) &= \frac{ \frac{Z_T}{T} \frac{1}{T}\int_0^T X_s^2 ds - \left( \frac{1}{T}\int_0^T X_s ds \right) \left( \frac{1}{T}\int_0^T X_s dZ_s \right)}{\frac{1}{T}\int_0^T X_s^2ds - \left( \frac{1}{T}\int_0^T X_s ds \right)^2},
 \\ \widehat{\beta}_T(X) &= \frac{ \frac{1}{T} \int_0^T X_s dZ_s - \frac{Z_T}{T} \frac{1}{T}\int_0^T X_s ds }{\frac{1}{T}\int_0^T X_s^2ds - \left( \frac{1}{T}\int_0^T X_s ds \right)^2}.
\end{align*} 

The following theorem proves the consistency and asymptotic normality for these estimators.

\begin{Theorem}\label{thm: continuous MLE for VOU}
    Under assumption \ref{VOU} the maximum-likelihood estimator is strongly consistent. Moreover, for each compact $\Theta \subset \R \times (-\infty, 0)$, it holds that
    \begin{align}\label{eq: consistency}
     \lim_{T \to \infty}\sup_{(b,\beta) \in \Theta}\P_{b,\beta}\left[ \left| (\widehat{b}_T, \widehat{\beta}_T) - (b,\beta)\right| > \e\right] = 0, \qquad \forall \e > 0,
    \end{align}
    and the estimator is asymptotically normal in the sense that
    \begin{align}\label{eq: asymptotic normality}
     \mathcal{L}_{b,\beta}\left(\sqrt{T}(\widehat{b}_T - b, \widehat{\beta}_T - \beta) \right) \Longrightarrow \sigma I(b,\beta)^{-1/2}\mathcal{N}(0,\mathrm{id}_{2 \times 2}), \qquad T \longrightarrow \infty
    \end{align}
    holds uniformly in $(b,\beta) \in \Theta$ with Fisher information matrix 
    \[
     I(b,\beta) = \begin{pmatrix}
        1 & m_1(b,\beta)\\ m_1(b,\beta) & m_2(b,\beta)
    \end{pmatrix}.
    \]
\end{Theorem}
\begin{proof}
    Let $X$ be the VOU process with parameters $(b,\beta)$ obtained from \eqref{eq: VOU} and let $D_T = T\int_0^T X_s^2ds - \left(\int_0^TX_sds\right)^2$. Since $X$ is not constant due to $\sigma, K \neq 0$, one has $\P[D_T > 0] = 1$. Using $dZ_t(X) = (b+\beta X_t)dt + \sigma dB_t$ with respect to $\P$, we find after a short computation 
    \begin{align*}
     \begin{pmatrix}\widehat{b}_T(X) \\  \widehat{\beta}_T(X) \end{pmatrix} 
     &= \begin{pmatrix} b \\ \beta \end{pmatrix}
     + \frac{1}{D_T}\begin{pmatrix} \sigma B_T \int_0^T X_s^2ds - \sigma \int_0^T X_s ds \int_0^T X_s dB_s
     \\ \sigma T \int_0^T X_s dB_s - \sigma B_T \int_0^T X_s ds
     \end{pmatrix}
     \\ &= \begin{pmatrix} b \\ \beta \end{pmatrix} + \frac{\sigma}{D_T}\begin{pmatrix} \int_0^T X_s^2 ds & -\int_0^T X_s ds \\ -\int_0^T X_s ds & T \end{pmatrix} \begin{pmatrix} B_T \\ \int_0^T X_s dB_s \end{pmatrix}
     \\ &= \begin{pmatrix} b \\ \beta \end{pmatrix} + \sigma \langle M\rangle_T^{-1} M_T
    \end{align*}
    where $(M_T)_{T \geq 0}$ is a square-integrable martingale given by $M_T = (B_T, \int_0^T X_s dB_s)^{\top}$,
    \[
    \langle M \rangle_T = \begin{pmatrix} T & \int_0^T X_s ds \\ \int_0^T X_s ds & \int_0^T X_s^2 ds \end{pmatrix} \ \text{ and } \ \langle M \rangle_T^{-1} = \frac{1}{D_T}\begin{pmatrix} \int_0^T X_s^2 ds & -\int_0^T X_s ds \\ -\int_0^T X_s ds & T \end{pmatrix}.
    \]
    Since by Theorem \ref{thm: law of large numbers VOU}, $\frac{1}{T}\int_0^TX_sds$ and $\frac{1}{T}\int_0^T X_s^2 ds$ are both convergent in $L^2$ with limits $m_1(b,\beta)$ and $m_2(b,\beta)$ uniformly in $(b,\beta) \in \Theta$, we conclude that
    \begin{align}\label{eq: M bracket VOU convergence}
     \frac{\langle M \rangle_T}{T} \longrightarrow \begin{pmatrix} 1 & m_1(b,\beta) \\ m_1(b,\beta) & m_2(b,\beta)\end{pmatrix} = I(b,\beta)
    \end{align}
    converges in $L^2$ uniformly in $(b,\beta) \in \Theta$. Note that $\mathrm{det}(I(b,\beta)) = m_{\mathrm{var}}(\beta) > 0$. In particular, $\langle M \rangle_T \longrightarrow +\infty$ a.s. and hence, by the strong law of large numbers for martingales, we conclude that $\langle M \rangle_T^{-1}M_T \longrightarrow 0$ a.s., which proves strong consistency. 
    
    Let us show that the estimators are uniformly convergent. By direct computation we obtain
    \begin{align*}
        \left(\frac{\langle M_T\rangle_T}{T}\right)^{-1} &= \frac{1}{\frac{1}{T}\int_0^T X_s^2 ds - \left( \frac{1}{T}\int_0^T X_s ds \right)^2} \begin{pmatrix}
            \frac{1}{T}\int_0^T X_s^2 ds & - \frac{1}{T}\int_0^T X_s ds \\ -\frac{1}{T}\int_0^T X_s ds & 1 
        \end{pmatrix}
    \end{align*}
    which converges due to Theorem \ref{thm: law of large numbers VOU} locally uniformly towards
    \[
     \frac{1}{m_{\mathrm{var}}(\beta)} \begin{pmatrix}
            m_2(b,\beta) & - m_1(b,\beta) \\ -m_1(b,\beta) & 1 
        \end{pmatrix} = I(b,\beta)^{-1}
    \]
    Similarly, we see that $\frac{M_T}{T} = (\frac{B_T}{T},  \frac{1}{T}\int_0^T X_s dB_s)^{\top} \longrightarrow 0$ in $L^2(\Omega)$ locally uniformly in $(b,\beta)$. Theorem \ref{thm: uniform continuous mapping theorem}.(a) shows that 
    \begin{align*}
    \sup_{(b,\beta) \in \Theta}\P[ |(\widehat{b}_T(X), \widehat{\beta}_T(X)) - (b,\beta)| > \e]
        &= \sup_{(b,\beta) \in \Theta}\P\left[ \left| \left(\frac{\langle M_T\rangle_T}{T}\right)^{-1} \frac{M_T}{T}\right| > \sigma^{-1}\e\right] \longrightarrow 0.
    \end{align*}
    The asymptotic normality uniformly in $(b,\beta) \in \Theta$ follows from the central limit theorem \cite[Proposition 1.21]{MR2144185} applied to
    \[
     \left( \frac{\langle M \rangle_T}{T}\right)^{-1}\frac{M_T}{\sqrt{T}} \Longrightarrow I(b,\beta)^{-1} \mathcal{N}(0, I(b,\beta)) = I(b,\beta)^{-1/2}\mathcal{N}(0,\mathrm{id}_{2\times 2}).
    \]
\end{proof}

As a particular case, below we summarise the case where either $b$ or $\beta$ is estimated, while the other parameter is known.

\begin{Remark}
    The following assertions hold under the assumption \ref{VOU}:
    \begin{enumerate}
     \item[(a)] If $\beta < 0$ is known, then the MLE for $b$ given by
    \[
     \widehat{b}_T(X) = \frac{Z_T(X)}{T} - \frac{\beta}{T}\int_0^T X_t dt
    \]
    is strongly consistent. Moreover, it is consistent in probability and asymptotically normal locally with variance $\sigma^2$ uniformly in the parameters.
    
    \item[(b)] If $b$ is known, then the MLE for $\beta$ given by
    \[
        \widehat{\beta}_T(X) = \frac{\int_0^T X_t dZ_t(X) - b\int_0^T X_t dt}{\int_0^T X_t^2 dt}
    \]
    is strongly consistent.  Moreover, it is consistent in probability and asymptotically normal with variance $\sigma^2/m_2(b,\beta)$ locally uniformly in the parameters.
    \end{enumerate}
\end{Remark}

\subsection{Discrete high-frequency observations}

Next, we consider discrete high-frequency observations. For each $n \geq 1$ let $0 = t_0^{(n)} < \dots < t_n^{(n)}$ be a partition of length $t_n^{(n)}$. To simplify the notation, we let 
\begin{align}\label{eq: partition}
    \mathcal{P}_n = \{ [t_k^{(n)}, t_{k+1}^{(n)}] \ : \ k = 0,\dots, n-1 \}
\end{align}
be the collection of neighbouring intervals determined by this partition, and let
\[
    |\mathcal{P}_n| = \max_{[u,v] \in \mathcal{P}_n}(v-u)
\]
be the mesh size of the partition. To simplify the notation, we denote by $Z = Z(X)$ the process \eqref{eq: Z process} with $X$ given by \eqref{eq: VOU}. Let us define the discretization of $X$ with respect to $\mathcal{P}_n$ by
\begin{align}\label{eq: X discretization}
    X^{\mathcal{P}_n}_s = \sum_{[u,v] \in \mathcal{P}_n}\1_{[u,v)}(s)X_{u}.
\end{align}

Discretization of the continuous time maximum-likelihood estimator $(\widehat{b}_{t_n^{(n)}}, \widehat{\beta}_{t_n^{(n)}})$ gives
\begin{footnotesize}
\begin{align*}
    \widehat{b}^{(n,m)}(X) &= \frac{Z^{\mathcal{P}_n}_{t_n^{(n)}}\left(\sum_{[u,v] \in \mathcal{P}_n}X_{u}^2(v-u)\right) - \left( \sum_{[u,v] \in \mathcal{P}_n}X_{u}(v-u)\right)\left(\sum_{[u,v] \in \mathcal{P}_n}X_{u}(Z^{\mathcal{P}_m}_{v} - Z^{\mathcal{P}_m}_{u}) \right)}
    {t_n^{(n)}\left(\sum_{[u,v] \in \mathcal{P}_n}X_{u}^2(v-u)\right) - \left( \sum_{[u,v] \in \mathcal{P}_n}X_{u}(v-u)\right)^2}
    \\ \widehat{\beta}^{(n,m)}(X) &= \frac{t_n^{(n)}\left(\sum_{[u,v] \in \mathcal{P}_n}X_{u}(Z^{\mathcal{P}_m}_{v} - Z^{\mathcal{P}_m}_{u}) \right) - Z^{\mathcal{P}_n}_{t_n^{(n)}} \left(\sum_{[u,v] \in \mathcal{P}_n}X_{u} (v-u) \right) }
    {t_n^{(n)}\left(\sum_{[u,v] \in \mathcal{P}_n}X_{u}^2(v-u)\right) - \left( \sum_{[u,v] \in \mathcal{P}_n}X_{u}(v-u)\right)^2}.
\end{align*} \end{footnotesize}
For the auxiliary process $Z = Z(X)$ we use a finer discretization with respect to $\mathcal{P}_m$ with $m \geq n$ given by 
\begin{align}\label{eq: Z discretization}
     Z^{\mathcal{P}_m}_{u}(X) = \sum_{\bfrac{[u',v'] \in \mathcal{P}_m}{v' \leq u}}(X_{u - u'} - x_0)L((u', v']) + K(0_+)^{-1}(X_{u} - x_0)
\end{align}
where $[u,v] \in \mathcal{P}_n$, $Z^{\mathcal{P}_m}_0(X) = 0$, and we use the convention $1/+\infty = 0$. For applications, one would typically choose the equidistant partition and assume that $\mathcal{P}_m$ is a refinement of $\mathcal{P}_n$. Examples of partitions and Volterra kernels that satisfy the conditions imposed in this Section are given in Section \ref{sec:6_examples}.

\begin{Lemma}
 Suppose assumption \ref{VOU} holds. Then for each $p \in [1,\infty)$ there exists a constant $C_p > 0$ independent of $b,\beta, n$ such that
\begin{align}\label{eq: discretization bound 1}
        \sup_{n \geq 1}\sup_{t \in [0,t_n^{(n)})}\left(\|X_t\|_{L^{p}(\Omega)} + \|X^{\mathcal{P}_n}_t\|_{L^{p}(\Omega)}\right) \leq C_p
\end{align}
and 
\begin{align}\label{eq: discretization bound 2}
    \| X_t - X^{\mathcal{P}_n}_t\|_{L^p(\Omega)} \leq C_p |\mathcal{P}_n|^{\gamma}, \qquad t \in [0,t_n).
\end{align}
\end{Lemma}
\begin{proof}
 The first estimate is a direct consequence of \eqref{eq: bounded moments VOU} while the second one follows from H\"older continuity of sample paths (see Lemma \ref{prop: VOU technical stuff}).
\end{proof}

Below we let $\lesssim$ denote an inequality that is supposed to hold up to a constant independent of the discretisation and locally uniform in $(b,\beta) \in \R \times (-\infty, 0)$. The next lemma provides an error bound of the integrals for $X$ and $X^{\mathcal{P}_n}$. 

\begin{Lemma}\label{lemma: discretization VOU}
 Suppose that (VOU) holds. Let $f: \R \longrightarrow \R$ satisfy
 \[
  |f(x) - f(y)| \leq C_f(1 + |x|^q + |y|^q)|x-y|, \qquad \forall x,y \in \R,
 \]
 for some constants $C_f,q\geq0$. Then for each $p \in [2,\infty)$ there exists $C_p(b,\beta)$ independent of the discretization and locally bounded with respect to $(b,\beta) \in \R \times (-\infty, 0)$ such that
 \begin{align*}
      &\ \left\| \int_0^{t_n^{(n)}} f(X_s)ds -  \sum_{[u,v] \in \mathcal{P}_n}f(X_{u})(v-u) \right\|_{L^p(\Omega)} 
      \\ &\quad + \left\| \int_0^{t_n^{(n)}} f(X_s)dZ_s - \sum_{[u,v]\in \mathcal{P}_n}f(X_u)(Z_v - Z_u) \right\|_{L^p(\Omega)} 
     \leq C_p(b,\beta) \left( \left(t_n^{(n)}\right)^{\frac{1}{2}} + t_n^{(n)} \right)|\mathcal{P}_n|^{\gamma}
 \end{align*}
 and for each $[u,v] \in \mathcal{P}_n$ and $m \geq n$ one has
    \[
     \left\| Z_{u} - Z^{\mathcal{P}_m}_{u} \right\|_{L^p(\Omega)} \leq C_p(b,\beta) |\mathcal{P}_m|^{\gamma} L((0,u]).
    \]
\end{Lemma}
\begin{proof}
    Using $\sum_{[u,v] \in \mathcal{P}_n}f(X_{u})(v-u) = \int_0^{t_n^{(n)}}f(X^{\mathcal{P}_n}_s)ds$ combined with the H\"older inequality, we obtain for the first term
    \begin{align*}
        &\ \left\| \int_0^{t_n^{(n)}} f(X_s)ds - \sum_{[u,v] \in \mathcal{P}_n}f(X_{u})(v-u) \right\|_{L^p(\Omega)}
        \\ &\leq \int_0^{t_n^{(n)}} \| f(X_s) - f(X^{\mathcal{P}_n}_s)\|_{L^p(\Omega)}ds
        \\ &\lesssim \int_{0}^{t_{n}^{(n)}} \left\|(1 + |X_s|^q + |X^{\mathcal{P}_n}_{s}|^q)|X_s - X^{\mathcal{P}_n}_{s}| \right\|_{L^p(\Omega)} ds 
        \\ &\lesssim \int_{0}^{t_{n}^{(n)}} \left\| X_s - X^{\mathcal{P}_n}_{s} \right\|_{L^{2p}(\Omega)} ds 
       \lesssim |\mathcal{P}_n|^{\gamma} t_n^{(n)}
    \end{align*}
    where we have used \eqref{eq: discretization bound 1} and \eqref{eq: discretization bound 2}. 

    For the second bound, we first use $\sum_{[u,v] \in \mathcal{P}_n}f(X_{u})(Z_{v} - Z_{u})
     = \int_0^{t_n^{(n)}}f(X^{\mathcal{P}_n}_s)dZ_s$ and then the semimartingale representation $dZ_t = (b + \beta X_t)dt + \sigma dB_t$ to find that 
    \begin{align*}
    &\ \left\| \int_0^{t_n^{(n)}} f(X_s)dZ_s - \sum_{[u,v] \in \mathcal{P}_n}f(X_{u})(Z_{v} - Z_{u})\right\|_{L^p(\Omega)} 
    \\ &\leq \int_{0}^{t_n^{(n)}} \left\| (f(X_s) - f(X^{\mathcal{P}_n}_{s}))(b+\beta X_s) \right\|_{L^p(\Omega)} ds 
    \\ &\qquad + \sigma \left\|  \int_{0}^{t_{n}^{(n)}} (f(X_s) - f(X^{\mathcal{P}_n}_{s}))dB_s \right\|_{L^p(\Omega)} 
    \\ &\lesssim \int_{0}^{t_{n}^{(n)}} \left\| (1 + |X_s|^q + |X^{\mathcal{P}_n}_{s}|^q)(|b|+|\beta|| X_s|)|X_s - X^{\mathcal{P}_n}_{s}| \right\|_{L^{p}(\Omega)} ds 
    \\ &\qquad + \left( \E\left[ \left(\int_{0}^{t_{n}^{(n)}} (1 + |X_s|^q + |X^{\mathcal{P}_n}_{s}|^q)^2|X_s - X^{\mathcal{P}_n}_{s}|^2 ds\right)^{\frac{p}{2}} \right] \right)^{\frac{1}{p}}
    \\ &\lesssim \int_{0}^{t_{n}^{(n)}} \left\| X_s - X^{\mathcal{P}_n}_{s}\right\|_{L^{2p}(\Omega)} ds 
    + \left( t_n^{(n)}\right)^{\frac{1}{2} - \frac{1}{p}}\left( \int_{0}^{t_{n}^{(n)}} \E\left[|X_s - X^{\mathcal{P}_n}_{s}|^p \right]ds \right)^{\frac{1}{p}}
    \\ &\lesssim \left(t_n^{(n)} + \left( t_n^{(n)}\right)^{\frac{1}{2}} \right)|\mathcal{P}_n|^{\gamma}
    \end{align*}
    where we have used the H\"older inequality combined with \eqref{eq: discretization bound 1} and \eqref{eq: discretization bound 2}. 

    Finally, let us bound the discretisation error for $Z^{\mathcal{P}_m}$. Writing 
    $Z_{u} = K(0_+)^{-1}(X_u - x_0) + \int_{(0,u]}(X_{u - s} - x_0)L(ds)$ and using the particular form of $Z^{\mathcal{P}_m}$ and $X^{\mathcal{P}_m}$ we find
    \begin{align*}
     \left\| Z_{u} - Z^{\mathcal{P}_m}_u\right\|_{L^p(\Omega)}
     \leq \int_{[u,v]} \|X_{u-s} - X_{u-s}^{\mathcal{P}_m}\|_{L^p(\Omega)} L(ds)
     \leq |\mathcal{P}_m|^{\gamma} L((0,u])
    \end{align*}
    which proves the assertion.
\end{proof}

The following is our main result on the asymptotic behaviour of the discretised high-frequency estimators.

\begin{Theorem}\label{thm: discrete MLE VOU}
 Suppose that assumption \ref{VOU} holds. Assume that the discretization satisfies $t_n^{(n)} \longrightarrow \infty$ as $n \to \infty$, and 
 \begin{align}\label{eq: mesh size}
  \lim_{n \to \infty} \sqrt{t_n^{(n)}} |\mathcal{P}_n|^{\gamma} = 0.
 \end{align}
 Let $(m(n))_{n\geq 1}$ be such that $m(n) \geq n$ with property
 \begin{align}\label{eq: mesh size 1}
    \lim_{n \to \infty} n \cdot \frac{L((0,t_n^{(n)}])}{\sqrt{t_n^{(n)}}}|\mathcal{P}_{m(n)}|^{\gamma} = 0.
 \end{align}
 Then for each compact $\Theta \subset \R \times (- \infty, 0)$ and $\e > 0$
 \begin{align}\label{eq: discrete consistency}
  \lim_{n \to \infty}\sup_{(b,\beta) \in \Theta}\P_{b,\beta}\left[ \left|(\widehat{b}^{(n, m(n))}, \widehat{\beta}^{(n, m(n))}) - (b, \beta) \right| > \e \right] = 0,
 \end{align}
 and asymptotically normality holds uniformly in $(b,\beta) \in \Theta$, i.e.
 \begin{align}\label{eq: discrete asymptotic normality}
  \mathcal{L}_{b,\beta}\left(\sqrt{t_n^{(n)}}\left( (\widehat{b}^{(n,m(n))}, \widehat{\beta}^{(n,m(n))}) - (b,\beta)\right) \right) \Longrightarrow \sigma I(b,\beta)^{-1/2}\mathcal{N}(0, \mathrm{id}_{2\times 2}), \qquad n \to \infty.
 \end{align}
\end{Theorem}

\begin{proof}
 \textit{Step 1.} Firstly, let us define the auxiliary estimators
 \begin{footnotesize}
     \begin{align*}
    \widetilde{b}^{(n)} &= \frac{Z_{t_n^{(n)}}\left(\sum_{[u,v] \in \mathcal{P}_n}X_{u}^2 (v-u)\right) - \left( \sum_{[u,v] \in \mathcal{P}_n}X_{u} (v-u)\right)\left(\sum_{[u,v] \in \mathcal{P}_n}X_{u}(Z_{v} - Z_{u}) \right)}{t_n^{(n)}\left(\sum_{[u,v] \in \mathcal{P}_n}X_{u}^2(v-u)\right) - \left( \sum_{[u,v] \in \mathcal{P}_n}X_{u}(v-u)\right)^2}
    \\ \widetilde{\beta}^{(n)} &= \frac{t_n^{(n)}\left(\sum_{[u,v] \in \mathcal{P}_n}X_{u}(Z_{v} - Z_{u}) \right) - Z_{t_n^{(n)}}\left( \sum_{[u,v] \in \mathcal{P}_n}X_{u}(v-u)\right)}{t_n^{(n)}\left(\sum_{[u,v] \in \mathcal{P}_n}X_{u}^2(v-u)\right) - \left( \sum_{[u,v] \in \mathcal{P}_n}X_{u}(v-u)\right)^2}.
\end{align*}
 \end{footnotesize}
 Then we obtain 
  \begin{align}\label{eq: discretisation decomposition VOU}
    \sqrt{t_n^{(n)}}\left(\begin{pmatrix}\widehat{b}^{(n)} \\ \widehat{\beta}^{(n)} \end{pmatrix} - \begin{pmatrix}
        b \\ \beta \end{pmatrix} \right)
        &= \sqrt{t_n^{(n)}}\left(\begin{pmatrix}\widehat{b}^{(n)} \\ \widehat{\beta}^{(n)} \end{pmatrix} - \begin{pmatrix}\widetilde{b}^{(n)} \\ \widetilde{\beta}^{(n)} \end{pmatrix} \right)
        \\ \notag &\qquad + \sqrt{t_n^{(n)}}\left(\begin{pmatrix}\widetilde{b}^{(n)} \\ \widetilde{\beta}^{(n)} \end{pmatrix} - \begin{pmatrix}\widehat{b}_{t_n^{(n)}} \\ \widehat{\beta}_{t_n^{(n)}} \end{pmatrix}\right) 
        \\ \notag &\qquad + \sqrt{t_n^{(n)}}\left( \begin{pmatrix}\widehat{b}_{t_n^{(n)}} \\ \widehat{\beta}_{t_n^{(n)}} \end{pmatrix} - \begin{pmatrix}b \\ \beta \end{pmatrix}\right)
 \end{align}
 Since $t_n^{(n)} \longrightarrow \infty$, the last term converges by Theorem \ref{thm: continuous MLE for VOU} to the desired Gaussian law locally uniformly in $(b,\beta)$. Thus by Proposition \ref{prop: uniform weak convergence}.(a) it suffices to prove that the first two terms converge to zero in probability locally uniformly in $(b,\beta)$. 
 
 \textit{Step 2.} Let us define 
 \begin{align*}
     A^{(1)}_n &= \frac{1}{t_n^{(n)}}\int_0^{t_n^{(n)}} X_s ds, 
     \qquad A^{(1)}_{\mathcal{P}_n} = \frac{1}{t_n^{(n)}}\sum_{[u,v] \in \mathcal{P}_n}X_{u} (v-u)
     \\ A^{(2)}_{n} &= \frac{1}{t_n^{(n)}}\int_0^{t_n^{(n)}} X^2_s ds, \qquad A^{(2)}_{\mathcal{P}_n} = \frac{1}{t_n^{(n)}}\sum_{[u,v] \in \mathcal{P}_n}X^2_{u} (v-u)
     \\ B_n &= \frac{1}{t_n^{(n)}}\int_0^{t_n^{(n)}}X_s dZ_s, \qquad B_{\mathcal{P}_n} = \frac{1}{t_n^{(n)}}\sum_{[u,v] \in \mathcal{P}_n}X_{u}(Z_{v} - Z_{u}).
 \end{align*}
 Then $A_n^{(1)} \longrightarrow m_1(b,\beta)$, $A_n^{(2)} \longrightarrow m_2(b,\beta)$, and since $\frac{1}{t_n^{(n)}}\int_0^{t_n^{(n)}} X_s dB_s \longrightarrow 0$ in $L^2(\Omega)$, also
 \[
  B_n = \frac{1}{t_n^{(n)}}\int_0^{t_n^{(n)}} X_s (b + \beta X_s)ds + \frac{\sigma }{t_n^{(n)}}\int_0^{t_n^{(n)}} X_s dB_s \longrightarrow bm_1(b,\beta) + \beta m_2(b,\beta)
 \]
 in probability locally uniformly in $(b,\beta)$ due to Theorem \ref{thm: law of large numbers VOU}. By Lemma \ref{lemma: discretization VOU} and assumption \eqref{eq: mesh size} on the discretization, the discretizations $A^{(1)}_{\mathcal{P}_n}, A^{(2)}_{\mathcal{P}_n}, B_{\mathcal{P}_n}$ have the same locally uniform limit in probability as $A_n^{(1)}, A_n^{(2)}, B_n$. Finally, note that
 \[
  D_{\mathcal{P}_n}:= A^{(2)}_{\mathcal{P}_n} - ( A^{(1)}_{\mathcal{P}_n} )^2 \longrightarrow m_2(b,\beta) - \left(m_1(b,\beta)\right)^2 = m_\mathrm{var}(\beta)
 \]
 in probability locally uniformly in $(b,\beta)$ due to Theorem \ref{thm: uniform CMT}.(a). Similarly, we find 
 \[
  D_n = \left( A_n^{(2)} - (A_n^{(1)})^2\right) \longrightarrow m_{\mathrm{var}}(\beta).
 \] 
 
 \textit{Step 3.} By direct computation we see that
 \begin{align*}
     &\ \sqrt{t_n^{(n)}}\left( \widehat{b}^{(n)} - \widetilde{b}^{(n)}\right) 
     \\ &= D_{\mathcal{P}_n}^{-1}\left[ \frac{Z_{t_n^{(n)}} - Z^{\mathcal{P}_n}_{t_n^{(n)}}}{\sqrt{t_n^{(n)}}} A^{(2)}_{\mathcal{P}_n} - A^{(1)}_{\mathcal{P}_n}\frac{1}{\sqrt{t_n^{(n)}}}\sum_{[u,v] \in \mathcal{P}_n}X_{u}\left\{(Z_{v} - Z^{\mathcal{P}_m}_{v}) - (Z_{u} - Z^{\mathcal{P}_m}_{u})\right\}\right]
 \end{align*}
 By Lemma \ref{lemma: discretization VOU} and assumption \eqref{eq: mesh size} on the mesh size of the discretization we obtain
 \[
     \left\|\frac{Z_{t_n^{(n)}} - Z^{\mathcal{P}_n}_{t_n^{(n)}}}{\sqrt{t_n^{(n)}}}\right\|_{L^p(\Omega)} \lesssim (t_n^{(n)})^{-1/2} L((0,t_n^{(n)}]) |\mathcal{P}_n|^{\gamma} \lesssim (t_n^{(n)})^{1/2}|\mathcal{P}_n|^{\gamma} \longrightarrow 0
 \]
 where we have used Remark \ref{rem: L resolvent bound} from the next section. Similarly, using the boundedness of moments (see \eqref{eq: bounded moments VOU}) and then Lemma \ref{lemma: discretization VOU}, we obtain 
 \begin{align*}
     &\ \left\| \frac{1}{\sqrt{t_n^{(n)}}}\sum_{[u,v] \in \mathcal{P}_n}X_{u} \left\{(Z_{v} - Z^{\mathcal{P}_m}_{v}) - (Z_{u} - Z^{\mathcal{P}_m}_{u})\right\}\right\|_{L^p(\Omega)}
     \\ &\lesssim n (t_n^{(n)})^{-1/2} \max_{[u,v] \in \mathcal{P}_n}\left(\left\| Z_{v} - Z^{\mathcal{P}_m}_{v} \right\|_{L^{2p}(\Omega)}
     + \left\| Z_{u} - Z^{\mathcal{P}_m}_{u} \right\|_{L^{2p}(\Omega)} \right)
     \\ &\lesssim  n (t_n^{(n)})^{-1/2} L((0,t_n^{(n)}]) |\mathcal{P}_m|^{\gamma} \longrightarrow 0
 \end{align*}
 where we have used $m = m(n)$ and assumption \eqref{eq: mesh size 1}. In view of Step 2, the convergence $\sqrt{t_n^{(n)}}\left( \widehat{b}^{(n)} - \widetilde{b}^{(n)}\right) \longrightarrow 0$ follows from Proposition \ref{thm: uniform CMT}.(a). Similarly, we obtain 
 \begin{align*}
     &\ \sqrt{t_n^{(n)}}\left( \widehat{\beta}^{(n)} - \widetilde{\beta}^{(n)}\right) 
     \\ &= D_{\mathcal{P}_n}^{-1}A_{\mathcal{P}_n}^{(1)}\frac{1}{\sqrt{t_n^{(n)}}}\sum_{[u,v] \in \mathcal{P}_n}X_u \left\{ (Z_v - Z_v^{\mathcal{P}_m}) - (Z_u - Z_u^{\mathcal{P}_m})\right\}
     \\ &\qquad \qquad - D_{\mathcal{P}_n}^{-1}\frac{Z_{t_n^{(n)}} - Z_{t_n^{(n)}}^{\mathcal{P}_n}}{\sqrt{t_n^{(n)}}} A_{\mathcal{P}_n}^{(1)} \longrightarrow 0.
 \end{align*}
 
 \textit{Step 4.} In this step, we show that the second term in \eqref{eq: discretisation decomposition VOU} converges to zero in probability locally uniformly in $(b,\beta)$. The latter completes the proof of this theorem. Firstly, by direct computation, we see that
 \begin{align*}
     \sqrt{t_n^{(n)}}\left( \widetilde{b}^{(n)} - \widehat{b}_{t_n^{(n)}}\right)
  &= \sqrt{t_n^{(n)}}\frac{ A_{\mathcal{P}_n}^{(2)} A_n^{(1)} B_n - A_n^{(2)} A_{\mathcal{P}_n}^{(1)} B_{\mathcal{P}_n}}{D_n}
  \\ &\qquad + \sqrt{t_n^{(n)}}\frac{(A_n^{(1)})^2 A_{\mathcal{P}_n}^{(1)}B_{\mathcal{P}_n} - (A_{\mathcal{P}_n}^{(1)})^2 A_n^{(1)} B_n}{D_n}
  \\ &\qquad + \sqrt{t_n^{(n)}}\frac{Z_{t_n^{(n)}}}{t^{(n)}_n}\frac{(A_{\mathcal{P}_n}^{(1)})^2 A_n^{(2)} - (A_n^{(1)})^2 A_{\mathcal{P}_n}^{(2)}}{D_n}
  \\ &= I_1 + I_2 + I_3
 \end{align*}
 where after further manipulations we arrive at
 \begin{align*}
     I_1 &= \sqrt{t_n^{(n)}}\frac{(A_{\mathcal{P}_n}^{(2)} - A_n^{(2)})A_n^{(1)} B_n + A_n^{(2)}(A_n^{(1)} - A_{\mathcal{P}_n}^{(1)})B_n + A_n^{(2)}A_{\mathcal{P}_n}^{(1)}(B_n - B_{\mathcal{P}_n})}{D_n}
     \\ I_2 &= \sqrt{t_n^{(n)}}\frac{((A_n^{(1)})^2 - (A_{\mathcal{P}_n}^{(1)})^2) A_{\mathcal{P}_n}^{(1)} B_{\mathcal{P}_n} + (A_{\mathcal{P}_n}^{(1)})^2( A_{\mathcal{P}_n}^{(1)} - A_n^{(1)})B_{\mathcal{P}_n} + (A_{\mathcal{P}_n}^{(1)})^2 A_n^{(1)} (B_{\mathcal{P}_n} - B_n)}{D_n}
     \\ I_3 &= \sqrt{t_n^{(n)}}\frac{Z_{t_n^{(n)}}}{t_n^{(n)}}\frac{ ( (A_{\mathcal{P}_n}^{(1)})^2 - (A_n^{(1)})^2)A_n^{(2)} + (A_n^{(1)})^2 ( A_n^{(2)}  - A_{\mathcal{P}_n}^{(2)}) }{D_n}.
 \end{align*}
 The desired convergence now follows from step 2, Theorem \ref{thm: uniform CMT}.(a), and Lemma \ref{lemma: discretization VOU}. A similar computation and argument also gives
 \begin{align*}
     \sqrt{t_n^{(n)}}\left( \widetilde{\beta}^{(n)}_{t_n^{(n)}} - \widehat{\beta}_{t_n^{(n)}} \right) \longrightarrow 0.
 \end{align*} 
\end{proof}

Finally, let us briefly outline the case where one of the drift parameters is known.

\begin{Remark}\label{remark: VOU estimation b}
    The following assertions hold under assumption \ref{VOU}: 
    \begin{enumerate}
     \item[(a)] If $\beta$ is known, the partition $\mathcal{P}_n$ satisfies $t_n^{(n)} \longrightarrow \infty$ and \eqref{eq: mesh size}, then the discretized estimator 
    \[
     \widehat{b}^{(n)}(X) = \frac{Z^{\mathcal{P}_n}_{t_n^{(n)}}(X)}{t_n^{(n)}} - \frac{\beta}{t_n^{(n)}}\sum_{[u,v] \in \mathcal{P}_n}X_{u}(v-u)
    \]
    is consistent in probability and asymptotically normal with variance $\sigma^2$ locally uniformly in the parameters.
    
    \item[(b)] If $b$ is known, the discretization $\mathcal{P}_n$ satisfies $t_n^{(n)} \longrightarrow \infty$ and \eqref{eq: mesh size} and \eqref{eq: mesh size 1}, then the discretized estimator
    \[
        \widehat{\beta}^{(n, m)}(X) = \frac{\sum_{[u,v] \in \mathcal{P}_n} X_{u}\left( Z_v^{\mathcal{P}_m}(X) - Z_u^{\mathcal{P}_m}(X) \right) - b\sum_{[u,v] \in \mathcal{P}_n}X_{u}(v-u)}{\sum_{[u,v] \in \mathcal{P}_n}X_{u}^2(v-u)}
    \]
    is also consistent in probability and asymptotically normal with variance $\sigma^2/m_2(b,\beta)$ locally uniformly in the parameters.
    \end{enumerate}
 
\end{Remark}

\section{Examples}\label{sec:6_examples}

In this section, we briefly discuss examples of scalar-valued Volterra kernels for which the results of Section 4 are applicable. Because of Theorem \ref{thm: discrete MLE VOU}, bounds on the growth rate of $L((0,t])$ are essential to verify conditions \eqref{eq: mesh size} and \eqref{eq: mesh size 1}. The latter are collected in the following simple remark.

\begin{Remark}\label{rem: L resolvent bound}
    Suppose that $K$ is completely monotone and not identically zero. Then the resolvent of the first kind exists and has representation $L(ds) = K(0_+)^{-1}\delta_0(ds) + L_0(s)ds$ where $L_0 \in L_{loc}^1(\R_+)$ is completely monotone and $K(0_+) := \lim_{t \to 0}K(t) \in (0, +\infty]$, see \cite[Chapter 5, Theorem 5.4]{MR1050319}. Moreover, by monotonicity combined with $K \ast L = 1$ we obtain
    \[
    L((0,t]) = K(t)^{-1}\int_{(0,t]}K(t)L(ds) \leq K(t)^{-1}\int_{(0,t]}K(t-s)L(ds) \leq K(t)^{-1}.
    \]
    Similarly, using the monotonicity of $L_0$, we obtain
    \begin{align*}
        L((0,t]) &= \int_0^1 L_0(s)ds + \int_1^t L_0(s)ds
        \leq \left(\int_0^1 L_0(s)ds - L_1(1)\right) + L_0(1)t.
    \end{align*}
\end{Remark}

For applications with discrete observations, it is natural to choose an equidistant partition where $t_k = k |\mathcal{P}_n|$. Below we briefly discuss conditions \eqref{eq: mesh size}, \eqref{eq: mesh size 1} and (K) with particular focus on the choices $|\mathcal{P}_n| = n^{-\eta}$ with $\eta \in (0,1)$ whence $t_n^{(n)} = n^{1-\eta} \longrightarrow \infty$, and $|\mathcal{P}_n| = \frac{\log(n)}{n}$ whence $t_n^{(n)} = \log(n) \longrightarrow \infty$. These examples allow for a trade-off between long-horizon (big $t_n^{(n)}$) and discretisation depth.

\begin{Example}[fractional kernel]
    The fractional kernel $K(t) = t^{\alpha-1}/\Gamma(\alpha)$ with $\alpha \in (1/2,1)$ satisfies \eqref{eq: K asymptotics} and \eqref{eq: K global regularity} with $\gamma = \alpha - 1/2$. Moreover, we have $L(ds) = \frac{t^{-\alpha}}{\Gamma(1-\alpha)}ds$ and hence $L((0,t]) = \Gamma(\alpha)^{-1}t^{1-\alpha}$.
    \begin{enumerate}
    \item[(a)] Let $|\mathcal{P}_n| = n^{-\eta}$ with $\frac{1}{2\alpha} < \eta < 1$. Since $2\alpha \eta > 1$, condition \eqref{eq: mesh size} is satisfied while condition \eqref{eq: mesh size 1} holds for any sequence $(m(n))_{n \geq 1}$ with the property 
    \[
        \lim_{n \to \infty} \frac{n^{1 + (1-\eta)\left(\frac{3}{2}-\alpha\right)}}{m(n)^{\left(\alpha - \frac{1}{2}\right)\eta}} = 0.
    \]

    \item[(b)] Let $|\mathcal{P}_n| = \frac{\log(n)}{n}$. Then conditions \eqref{eq: mesh size} holds, and \eqref{eq: mesh size 1} reduces to 
    \[
        \lim_{n \to \infty}\left(\frac{\log(m(n))}{\log(n)}\right)^{\alpha - 1/2}\frac{n}{m(n)^{\alpha - 1/2}} = 0.
    \]
    \end{enumerate}
\end{Example}

The next example illustrates the behaviour of a kernel that is not regular but essentially behaves like the fractional kernel with $\alpha \nearrow 1$.

\begin{Example}[log-kernel]
    The kernel $K(t) = \log(1+1/t)$ satisfies \eqref{eq: K asymptotics} and \eqref{eq: K global regularity} for any choice of $\gamma \in (0,1/2)$ and any $\alpha \in (1/2, 1]$. Remark \ref{rem: L resolvent bound} gives $L((0,t]) \leq C \log(1+1/t)^{-1}$.
    \begin{enumerate}
        \item[(a)] Let $|\mathcal{P}_n| = n^{-\eta}$ for $\frac{1}{2} < \eta < 1$, then condition \eqref{eq: mesh size} is satisfied and \eqref{eq: mesh size 1} is satisfied if for some $\gamma \in (0,1/2)$
    \begin{align*}
        \lim_{n \to \infty}n^{\frac{3 - \eta}{2}} m(n)^{- \eta \gamma} = 0. 
    \end{align*}

        \item[(b)] Let $|\mathcal{P}_n| = \frac{\log(n)}{n}$. Then \eqref{eq: mesh size} holds and \eqref{eq: mesh size 1} is satisfied if for some $\gamma \in (0,1/2)$
        \[
        \lim_{n \to \infty} n \left( \frac{\log(m(n))}{m(n)}\right)^{\gamma}. = 0.
        \]
    \end{enumerate}
\end{Example}

Below, we illustrate our conditions for the regular kernel obtained as a linear combination of exponentials. The latter plays a central role in multi-factor Markovian approximations, see \cite{MR3934104, MR4521278}.

\begin{Example}[exponential kernel]
    Let $K(t) = \sum_{i=1}^{N}c_i e^{-\lambda_i t}$ for some $c_1,\dots, c_N > 0$ and $\lambda_1,\dots, \lambda_N \geq 0$. Then conditions \eqref{eq: K asymptotics} and \eqref{eq: K global regularity} hold for $\alpha = 1$ and $\gamma = 1/2$. Using Remark \ref{rem: L resolvent bound} we obtain:
    \begin{enumerate}
        \item[(a)] Let $|\mathcal{P}_n| = n^{-\eta}$ with $0 < \eta < 1$, then \eqref{eq: mesh size} holds, and \eqref{eq: mesh size 1} is implied by
        \[
            \lim_{n \to \infty} n^{\frac{1+\eta}{2}}m(n)^{-\frac{1}{2}} = 0.
        \]

        \item[(b)] Let $|\mathcal{P}_n| = \frac{\log(n)}{n}$. Then \eqref{eq: mesh size} holds and \eqref{eq: mesh size 1} is satisfied if 
        \[
        \lim_{n \to \infty} n \log(n)^{\frac{1}{2}} \left( \frac{\log(m(n))}{m(n)}\right)^{\frac{1}{2}} = 0.
        \]
    \end{enumerate}
\end{Example}

Similarly, we could also consider kernels, e.g., of the form $\log(1 +  t^{-\alpha})$ and $e^{-t^{\alpha}}$ with $\alpha \in (0,1)$, and also $\log(1+t)/t$. 

Finally, we close this section with a refined estimate with a slower growth rate based on the asymptotical behaviour of $\int_1^t K(s)ds$ and $\int_0^t K(s)ds$. 

\begin{Lemma}\label{lemma: resolvent of the first kind}
 Let $K \in L_{loc}^1(\R_+)$ be completely monotone and suppose that there exist $\alpha_0 \in [0,1])$, $\alpha_1 > 0$, and a constant $C > 0$ such that
 \[
    \int_0^tK(s)ds \geq Ct^{\alpha_0}, \ t \in (0,1], \qquad \int_1^t K(s)ds \geq C t^{\alpha_1}, \ t > 1.
 \]
 Then there exists another constant $C' > 0$ such that the resolvent of the first kind satisfies
 \begin{align*}
  L((0,t]) \leq C' \left( \1_{ \{ 0 < \alpha_1 < 1\}}t^{1-\alpha_1} + \1_{\{ \alpha_1 = 1\}}\log(t) + \1_{ \{\alpha_1 > 1\} }\right), \qquad t \geq 1.
 \end{align*}
\end{Lemma}
\begin{proof}
    Let $\widehat{L}, \widehat{K}$ be the Laplace transforms of $L, K$. Since $K$ is completely monotone, \cite[Lemma 6.2.2]{MR1050319} implies that it has an analytic extension onto $\mathrm{Re}(z) > 0$, a continuous extension onto $\{ \mathrm{Re}(z) \geq 0 \ : \ z \neq 0 \}$, and the following lower bound holds
    \begin{align}\label{eq: K FT lower bound}
        |\widehat{K}(\mathrm{i}z)| &\geq \frac{1}{\sqrt{2}}\int_0^{|z|^{-1}}K(t)dt \geq \frac{C}{\sqrt{2}}\left( \1_{\{|z| > 1\}}|z|^{-\alpha_0} + \1_{\{|z| \leq 1\}}C|z|^{-\alpha_1}\right)
    \end{align}
    Likewise, since $L(ds) = K(0_+)^{-1}\delta_0(ds) + L_0(s)ds$ with $L_0$ completely monotone, also $\widehat{L}$ has an analytic extension onto $\mathrm{Re}(z) > 0$ and continuous extension onto $\{ \mathrm{Re}(z) \geq 0 \ : \ z \neq 0 \}$. Since $L \ast K = 1$, one has $\widehat{L}(z)\widehat{K}(z) = z^{-1}$ for $z > 0$, and hence by continuation also for $z = \mathrm{i}u$ with $u \in \R \backslash \{0\}$. 

    Let $\1_{(0,t]}$ be the indicator function on $(0,t]$. Its Fourier transform is given by
    \[
     \widehat{\1_{(0,t]}}(\mathrm{i}z) = \frac{e^{\mathrm{i}tz}-1}{\mathrm{i}z}.
    \]
    Below we would like to apply the Plancherel identity. However, since $L$ is not a $ L^2$ function, we use the following additional approximation which is justified by dominated convergence. Let $p_{\lambda}(z) = (2\pi \lambda^2)^{-1/2}e^{- \frac{z^2}{2\lambda}}$ with $\lambda > 0$. Then $\widehat{p}_{\lambda}(\mathrm{i}z) = e^{-\lambda |z|^2/2}$. An application of Plancherel identity combined with $\widehat{p_{\lambda} \ast L}(\mathrm{i}z) = 2\pi \widehat{p}_{\lambda}(\mathrm{i}z) \widehat{L}(\mathrm{i}z)$ gives
    \begin{align*}
        L((0,t]) &= \lim_{\lambda \searrow 0}\int_{\R} \1_{(0,t]}(s) (p_{\lambda} \ast L)(ds)
        \\ &= \lim_{\lambda \searrow 0} \int_{\R} \frac{e^{\mathrm{i}tz}-1}{\mathrm{i}z} e^{-\lambda |z|^2/2}\widehat{L}(\mathrm{i}z)dz
        = \int_{\R} \frac{e^{\mathrm{i}tz}-1}{\mathrm{i}z} \frac{1}{\mathrm{i}z \widehat{K}(\mathrm{i}z)}dz
    \end{align*}
    where the last equality follows from $\widehat{L}(\mathrm{i}z)\widehat{K}(\mathrm{i}z) = (\mathrm{i}z)^{-1}$. Hence we obtain from \eqref{eq: K FT lower bound} combined with $|e^{\mathrm{i}tz} - 1| \leq 1 \wedge t|z| = \1_{\{|z| \leq 1/t\}}t|z| + \1_{\{|z| > 1/t\}}$ and by using $t \geq 1$ the bound
    \begin{align*}
        L((0,t]) &\leq \int_{\R} \frac{|e^{\mathrm{i}tz} - 1|}{|z|^2}\frac{1}{|\widehat{K}(\mathrm{i}z)|}dz
        \\ &\leq C'\int_{\R} |z|^{-2} \left( \1_{\{|z| \leq 1/t\}}t|z| + \1_{\{|z| > 1/t\}} \right) \left( \1_{\{|z| > 1\}}|z|^{\alpha_0} + \1_{\{|z| \leq 1\}}|z|^{\alpha_1} \right) dz
        \\ &= 2C' \int_{\R_+} z^{-2}\left( t\1_{\{z \leq 1/t\}} z^{1 + \alpha_1} + \1_{\{1/t < z \leq 1\}}z^{\alpha_1} + \1_{\{z > 1 \}}|z|^{\alpha_0} \right)dz
    \end{align*}
    where $C'$ is some constant. The first integral gives 
    \[
     t\int_{\R_+} z^{-2} \1_{\{z \leq 1/t\}}z^{1 + \alpha_1} dz
     = \frac{t^{1-\alpha_1}}{\alpha_1},
    \]
    the second is bounded by
    \[
        \int_{\R_+} z^{-2} \1_{\{1/t < z \leq 1\}}z^{\alpha_1} dz
        \leq \begin{cases} \frac{1}{1-\alpha_1}t^{1- \alpha_1}, & \alpha_1 \in (0,1) \\ \log(t), & \alpha_1 = 1 \\ \frac{1}{\alpha_1-1}, & \alpha_1 > 1,\end{cases}
    \]
    and the last integral is bounded by a constant. This proves the assertion.
\end{proof}

\section{Numerical experiments and conclusion}

\subsection{Numerical validation}

In this section, we outline some numerical experiments that demonstrate the accuracy of the discrete high-frequency estimation and the method of moments for the drift parameters. As a first step, we simulate the sample paths of the univariate Volterra Ornstein-Uhlenbeck process \eqref{eq: VOU}. Our numerical method is based on the Euler scheme with equidistant discretisations $t_k^{(n)} = \frac{k}{n}T$ and a fixed time horizon $T > 0$. For simplicity, we write $t_k$ instead of $t_k^{(n)}$. Let $\xi_1,\dots, \xi_n$ be a sequence of independent standard Gaussian random variables, then we obtain the numerical scheme
\begin{align*}
    \widehat{X}_{k+1} = x_0 + \frac{T}{n}\sum_{i=0}^{k} K(t_{k+1} - t_i) (b+\beta \widehat{X}_{i}) + \sigma \sqrt{\frac{T}{n}} \sum_{i=0}^{k} K(t_{k+1} - t_i)\xi_{i+1}.
\end{align*}
Using this scheme, we generate for each of the Volterra kernels from (i) -- (iii) with different choices of $T$ and $d t = T/n$, $N \in \N$ sample paths with parameters $x_0 = 1$, $b = 1.2$, $\beta = -1$, and $\sigma = 0.3$. For these sample paths, we have evaluated the discretised maximum-likelihood estimators $(\widehat{b}, \widehat{\beta})$ as described in Theorem \ref{thm: discrete MLE VOU}, and if $K$ is the fractional kernel, also the method of moments estimators $(\widehat{b}^{\text{mom}}, \widehat{\beta}^{\text{mom}})$ described in Corollary \ref{cor: VOU method of moments}. Appendix B presents our results for the particular case $N = 200$ and $\sigma = 0.3$. 

Firstly, we demonstrate the Law of Large Numbers for the fractional kernel $K(t) = t^{\alpha-1}/\Gamma(\alpha)$ with $\alpha = 0.75$ with discretisation depth $dt \in \{0.2, 0.5, 1\}$, see Figure \ref{fig:Moments_fractional_dt}. In all cases, we observe the convergence of the first and second moments
\begin{align*}
    m_1(T) &= \frac{1}{T}\int_0^T X_t dt \longrightarrow \frac{b}{|\beta|} =: m_1
    \\ m_2(T) &= \frac{1}{T}\int_0^T X_t^2 dt \longrightarrow \left( \frac{b}{|\beta|} \right)^2 + C_{\alpha} \sigma^2 |\beta|^{2-\frac{1}{\alpha}} =: m_2
\end{align*}
as $T \to \infty$. Convergence of the method of moments, despite being theoretically predicted by Corollary \ref{cor: VOU method of moments}, turns out to be very slow, see Figure \ref{fig:MoM_fractional_dt}. In particular, it still has a large irreducible error for very large $T$, which cannot be further reduced by varying the discretisation depth $dt$, increasing the sample size $N$, or extending the time horizon $T$. For instance, in the particular case $dt = 0.5$ with $T = 500$, we find that the relative error for $m_1$ is negligibly smaller than $0.1\%$ while the relative error for $m_2$ is of order $0.6\%$. Nevertheless, the relative error for $\widehat{b}_T, \widehat{\beta}_T$ at $T = 500$ turns out to be approximately 23\%. Such an effect persists across different sets of parameters, including all choices of $\alpha$. One can easily verify that such large errors for the Method of Moments stem from the highly nonlinear method of moments formula derived in Corollary \ref{cor: VOU method of moments}. Indeed, by Taylor expansion of first order with respect to $m_2$, we find for the relative error in the estimation
\[
    \left|\frac{\widehat{\beta}_T - \beta}{\beta}\right| \approx  \frac{\alpha}{2\alpha - 1}\frac{\Delta}{m_2 - m_1^2 },
\]
where $\Delta$ denotes the absolute error for the estimation of $m_2$. Inserting $\Delta = 0.6\% \cdot m_2$ and the explicit values for $m_1,m_2$ in the situation of Figure \ref{fig:MoM_fractional_dt}, gives the relative error $\left|\frac{\widehat{\beta}_T - \beta}{\beta}\right| \approx 23.48\%$. To summarise, the Law of Large Numbers and hence Method of Moments is very effective for estimating either $b$ or $\beta$ when the other parameter is known, but generally fails to perform well for joint estimation in any realistic scenario of parameters.  

Contrary, the MLE converges very fast regardless of the particular choice of Volterra kernel and is robust with respect to different choices of discretisation depth, see Figure \ref{fig:MLE fractional} for the fractional kernel $K(t) = t^{\alpha-1}/\Gamma(\alpha)$ with $\alpha = 0.75$, Figure \ref{fig:MLE log} for the $\log$-kernel $K(t) = \log(1+1/t)$, and Figure \ref{fig:MLE exp} for the kernel $K(t) =  e^{- t} + 2 e^{-2 t}$. Figure \ref{fig:MLE fractional alpha} illustrates the performance of the MLE for the average relative error of $b$ and $\beta$ across different values of $\alpha$ with fixed discretisation depth $dt = 0.2$. Also here, we observe that the error decreases rapidly across all choices of $\alpha \in \{0.55, 0.65, 0.75, 0.85, 0.95\}$. Regarding the convergence rate and asymptotic normality, Figures \ref{fig:MLE histogram b} and \ref{fig:MLE histogram beta} illustrate the empirical distribution of $\sqrt{T}(\widehat{b}_T - b)$ and $\sqrt{T}(\widehat{\beta}_T - \beta)$ for $T = 200$ with $dt = 0.2$. The latter illustrates our theoretically established asymptotic normality.

\subsection{Conclusion}

We have established sufficient conditions for the equivalence of the laws for stochastic Volterra processes under a change of drift. The latter allows us to develop a rigorous MLE framework for stochastic Volterra processes. Furthermore, we provide a sufficient condition for the Law of Large Numbers that applies to general (not necessarily Markov) processes that satisfy the asymptotic independence condition. As a particular example, we carry out explicit computations for the Volterra Ornstein-Uhlenbeck process. As a consequence of the Law of Large Numbers, we prove the consistency for the Method of Moments, and derive consistency and asymptotic normality for the MLE under continuous and discrete observations. Our numerical results highlight that the MLE performs very well accross different sets of Volterra kernels, while the Method of Moments only provides reasonable results in the estimation of one fixed parameter.

\appendix 

\section{Proof of Lemma \ref{prop: VOU technical stuff}}

    Since $X$ and $X^{\mathrm{stat}}$ are Gaussian processes, property \eqref{eq: bounded moments VOU} follows from the boundedness of the first and second moments. The latter is, however, evident by Lemma \ref{prop: Ebeta uniform bounds}.(b) combined with \eqref{eq: VOU stationary process}. Let us now prove the desired inequality on the increments of the processes. Using first that $E_{\beta}$ is nonincreasing, then \cite[Lemma A.2]{FJ22}, and finally condition \eqref{eq: K global regularity} combined with $\int_0^{\infty}R_{\beta}(r)dr = |\beta|\int_0^{\infty}E_{\beta}(r)dr \leq 1$, we obtain 
    \begin{align*}
        \int_s^t E_{\beta}(r)^2 dr &\leq \int_0^{t-s} E_{\beta}(r)^2 dr \\ &\leq 2\left( 1+ \left(\int_0^{t-s}R_{\beta}(r)dr\right)^2 \right) \int_0^{t-s}K(r)^2 dr
        \\ &\leq 4C (t-s)^{2\gamma},
    \end{align*}
    where $C > 0$ is some constant. Likewise, \cite[Lemma A.2]{FJ22} yields
    \begin{align*}
     \int_0^{\infty} (E_{\beta}(t-s + r) - E_{\beta}(r))^2dr 
     \leq 12C (t-s)^{2\gamma}.
    \end{align*}
    Hence, we find a constant $C > 0$ (independent of $\beta$) such that for all $0 \leq t-s \leq 1$ and any $\beta < 0$
    \begin{align}\label{eq: Ebeta increments}
        \int_s^t E_{\beta}(r)^2 dr + \int_0^{\infty}(E_{\beta}(t-s+r) - E_{\beta}(r))^2 dr \leq C(t-s)^{2\gamma}.
    \end{align}
    To prove \eqref{eq: increment}, we use representation \eqref{eq: VOU Ebeta formulation} to get $X_t = \E[X_t] + \sigma \int_0^t E_{\beta}(t-s)dB_s$. This gives 
 \begin{align*}
        \E[|X_t - X_s|^p] &\leq 3^{p-1}|\E[X_t] - \E[X_s]|^p + 3^{p-1}\sigma^p \E\left[ \left|\int_0^s (E_{\beta}(t-r) - E_{\beta}(s-r))dB_r\right|^p \right]
        \\ &\qquad + 3^{p-1}\sigma^p \E\left[ \left| \int_s^t E_{\beta}(t-r)dB_r \right|^p\right].
 \end{align*}
 For the first term, we obtain by Cauchy-Schwarz and \eqref{eq: Ebeta increments} 
 \begin{align*}
  |\E[X_t] - \E[X_s]| 
    \leq (|x_0||\beta|+|b|)\int_s^t |E_{\beta}(r)|dr 
    \leq \sqrt{C}(|x_0||\beta|+|b|)(t-s)^{\frac{1}{2} + \gamma}.
 \end{align*}
 The remaining integrals can be bounded by
 \begin{align*}
        &\ \E\left[ \left|\int_0^s (E_{\beta}(t-r) - E_{\beta}(s-r))dB_r\right|^p \right]
        + \E\left[ \left| \int_s^t E_{\beta}(t-r)dB_r \right|^p\right]
        \\ &\leq C_p^{(1)} \left(\int_0^s |E_{\beta}(t-r) - E_{\beta}(s-r)|^2dr\right)^{p/2} 
        + \left( \int_s^t |E_{\beta}(t-r)|^2 dr \right)^{p/2}
        \leq C_p^{(2)} (t-s)^{p\gamma}
 \end{align*}
 where $C_p^{(1)}, C_p^{(2)}$ are some positive constants. To prove a similar bound for the stationary process, let us write  
 \[
  X_t^{\mathrm{stat}} = m_1(b,\beta) + \sigma\int_{0}^{\infty} E_{\beta}(t+s)dB_s' + \sigma \int_0^t E_{\beta}(t-s)dB_s.
 \]
 Hence it suffices to bound the integral against $(B_t')_{t \geq 0}$ for which we obtain
 \begin{align*}
     &\ \E\left[ \left| \int_{0}^{\infty} E_{\beta}(t+\tau)dB_{\tau}' - \int_{0}^{\infty} E_{\beta}(s+\tau)dB_{\tau}' \right|^p\right]
    \\ &\leq C_p^{(3)} \left(\int_0^{\infty}(E_{\beta}(t+\tau) - E_{\beta}(s + \tau))^2 dr \right)^{p/2} 
    \leq C_p^{(4)}(t-s)^{2\gamma}
 \end{align*}
 for some positive constants $C_p^{(3)}, C_p^{(4)}$. This proves \eqref{eq: increment}, and the desired H\"older regularity for $X$ and $X^{\mathrm{stat}}$ follows from the Kolmogorov-Chentsov theorem.

 Let us prove the last assertion. By Kolmogorov's tightness criterion and \eqref{eq: increment}, it follows that $(X_t^h)_{t \geq 0} = (X_{t + h})_{t \geq 0}$ is tight. Since $X$ and $X^{\mathrm{stat}}$ are Gaussian processes, it suffices to prove convergence of the mean and covariance structure. Using $E_{\beta} \in L^1(\R_+) \cap L^2(\R_+)$ it is clear that $\lim_{h \to \infty}\E[X_{h+t}] = \E[X_0^{\mathrm{stat}}]$. For the covariance structure, we find 
    \begin{align*}
        \mathrm{cov}(X_{t+h}, X_{s+h}) = \sigma^2 \int_0^{(t \wedge s)+h} E_{\beta}(|t - s|+r)E_{\beta}(r)dr
        \longrightarrow \mathrm{cov}(X^{\mathrm{stat}}_t, X_s^{\mathrm{stat}}).
    \end{align*}
    This proves the desired convergence \eqref{eq: weak convergence stationary}, and completes the proof of Lemma \ref{prop: VOU technical stuff}.

\section{Uniform weak convergence}

Let $(E, d)$ be a complete separable metric space and $\mathcal{B}(E)$ the Borel-$\sigma$ algebra on $E$. Let $p \geq 0$ and let $\mathcal{P}_p(E)$ be the space of Borel probability measures on $E$ such that $\int_E d(x,x_0)^p \mu(dx) < \infty$ holds for some fixed $x_0 \in E$. Here and below, we let $\Theta$ be an abstract index set (of parameters). 

\begin{Definition}
    Let $p \geq 0$ and $(\mu^{\theta}_n)_{n \geq 1, \theta \in \Theta}, (\mu^{\theta})_{\theta \in \Theta} \subset \mathcal{P}_p(E)$. Then $(\mu^{\theta}_n)_{n \geq 1, \theta \in \Theta}$ converges weakly in $\mathcal{P}_p(E)$ to $(\mu^{\theta})_{\theta \in \Theta}$ uniformly on $\Theta$, if
    \begin{align}\label{eq: weak convergence}
        \lim_{n \to \infty} \sup_{\theta \in \Theta}\left|\int_E f(x)\mu_n^{\theta}(dx) - \int_{E}f(x)\mu^{\theta}(dx) \right| = 0
    \end{align}
    holds for each continuous function $f: E \longrightarrow \R$ such that there exists $C_f > 0$ with $|f(x)| \leq C_f(1 + d(x_0,x)^p)$ for $x \in E$.
\end{Definition}

We abbreviate the weak convergence on $\mathcal{P}_p(E)$ by $\mu_n^{\theta} \Longrightarrow_p \mu^{\theta}$. In the particular case $p=0$ we write $\mathcal{P}_0(E) = \mathcal{P}(E)$ and denote the convergence by $\mu_n^{\theta} \Longrightarrow \mu^{\theta}$ which corresponds to \eqref{eq: weak convergence} for bounded and continuous functions. In the next proposition, we provide a characterisation for convergence in $\mathcal{P}_p(E)$.

\begin{Proposition}\label{prop: P characterization}
    Let $(\mu^{\theta}_n)_{n \geq 1}, (\mu^{\theta})_{\theta \in \Theta} \subset \mathcal{P}_p(E)$ with $p \in (0,\infty)$, and suppose that
    \begin{align}\label{eq: tightness 1}
        \lim_{R \to \infty}\sup_{\theta \in \Theta}\int_E \1_{\{d(x,x_0) > R\}}d(x,x_0)^p \mu^{\theta}(dx) = 0
    \end{align}
    Then the following are equivalent:
    \begin{enumerate}
        \item[(a)] $\mu_n^{\theta} \Longrightarrow_p \mu^{\theta}$ uniformly on $\Theta$. 

        \item[(b)] $(\mu^{\theta}_n)_{n \geq 1} \Longrightarrow \mu^{\theta}$ uniformly on $\Theta$, and
        \[
         \lim_{n \to \infty}\sup_{\theta \in \Theta}\left| \int_E d(x,x_0)^p \mu_n^{\theta}(dx) - \int_E d(x,x_0)^p \mu^{\theta}(dx)\right| = 0.
        \]

        \item[(c)] $(\mu^{\theta}_n)_{n \geq 1} \Longrightarrow \mu^{\theta}$ uniformly on $\Theta$, and
        \[
         \lim_{R \to \infty}\limsup_{n \to \infty}\sup_{\theta \in \Theta}\int_E \1_{\{d(x,x_0) > R\}}d(x,x_0)^p \mu_n^{\theta}(dx) = 0.
        \]
    \end{enumerate}
\end{Proposition}

The proof of this statement follows standard approximation arguments and is therefore left to the reader. For applications to limit theorems, we need analogues of the continuous mapping theorem and Slutzky's theorem where convergence is uniform on $\Theta$. To simplify the notation, we formulate it with respect to random variables $X_n^{\theta}, X^{\theta}$ defined on some probability space $(\Omega, \F, \P)$. Thus, we say that $X_n^{\theta} \Lu X^{\theta}$ if $\mathcal{L}(X_n^{\theta}) \Longrightarrow \mathcal{L}(X^{\theta})$ uniformly on $\Theta$. Finally, we say that $X_n^{\theta} \longrightarrow X^{\theta}$ in probability uniformly on $\Theta$, if 
\[
 \lim_{n \to \infty}\sup_{\theta \in \Theta}\P[ d(X_n^{\theta}, X^{\theta}) > \e] = 0, \qquad \forall \e > 0.
\]

Next, we state a simple version of the continuous mapping theorem adapted towards convergence uniform on $\Theta$.

\begin{Theorem}[Uniform continuous mapping theorem]\label{thm: uniform continuous mapping theorem}\label{thm: uniform CMT}
 Let $(X_n^{\theta})_{n \geq 1, \theta \in \Theta}$ and $(X^{\theta})_{\theta \in \Theta}$ be $E$-valued random variables on some probability space $(\Omega, \F, \P)$. Let $F: E \longrightarrow E'$ be measurable where $(E',d')$ is another complete and separable metric space. The following assertions hold:
 \begin{enumerate}
     \item[(a)] If $X_n^{\theta} \longrightarrow c(\theta)$ in probability uniformly on $\Theta$, $c(\theta)$ is deterministic and $F$ is uniformly continuous in $(c(\theta))_{\theta \in \Theta}$, then $F(X_n^{\theta}) \longrightarrow F(c(\theta))$ in probability uniformly on $\Theta$.

    \item[(b)] If $X_n^{\theta} \longrightarrow X^{\theta}$ in probability uniformly on $\Theta$ and $F$ is uniformly continuous, then $F(X_n^{\theta}) \longrightarrow F(X^{\theta})$ in probability uniformly on $\Theta$. 

    \item[(c)] Suppose there exists $\mu \in \mathcal{P}(E)$ such that for each $\e > 0$ there exists $\delta > 0$ with
    \begin{align}\label{eq: locally absolute continuity}
     \mu(A) < \delta \quad \Longrightarrow \quad \sup_{\theta \in \Theta}\P[X^{\theta} \in A] < \e.
    \end{align}
    for any Borel set $A \subset E$. If $X_n^{\theta} \Longrightarrow X^{\theta}$ and the set of discontinuity points $D_F \subset E$ of $F$ satisfies $\P[X^{\theta} \in D_F] = 0$, then $F(X_n^{\theta}) \Longrightarrow F(X^{\theta})$ uniformly on $\Theta$.
 \end{enumerate}
\end{Theorem}

The proof of this theorem is essentially a variant of the results obtained in \cite[Theorem 2.1]{BH19}. Thus, we leave the details to the interested reader.

\begin{Proposition}\label{prop: uniform weak convergence}
 Let $(X_n^{\theta})_{n \geq 1, \theta \in \Theta}$, $(Y_n^{\theta})_{n \geq 1, \theta \in \Theta}$, $(X^{\theta})_{\theta \in \Theta}$, and $(Y^{\theta})_{n \geq 1, \theta \in \Theta}$ be $E$-valued random variables on some probability space $(\Omega, \F, \P)$. The following assertions hold:
 \begin{enumerate}
     \item[(a)] If $X_n^{\theta} \Longrightarrow X^{\theta}$ uniformly on $\Theta$, $(Y_n^{\theta})_{n \geq 1, \theta \in \Theta}$ satisfies
     \[
      \lim_{n \to \infty}\sup_{\theta \in \Theta}\P[ d(X_n^{\theta}, Y_n^{\theta}) > \e] = 0, \qquad \forall \e > 0,
     \]
     and for each $\e > 0$ there exists a compact $K \subset E$ such that
    \begin{align*}
     \sup_{\theta \in \Theta}\sup_{n \geq 1}\P\left[ X_n^{\theta} \not \in K\right] + \sup_{\theta \in \Theta}\P \left[X^{\theta} \not \in K \right] < \e,
    \end{align*}
    then $Y_n^{\theta} \Longrightarrow X^{\theta}$ uniformly on $\Theta$.

     \item[(b)] Suppose that $X_n^{\theta} \Longrightarrow X^{\theta}$ uniformly on $\Theta$, $Y_n^{\theta} \longrightarrow c(\theta)$ in probability uniformly on $\Theta$ with deterministic $(c(\theta))_{\theta \in \Theta}$, and for each $\e > 0$ there exists a compact $K \subset E$ such that $c(\theta) \not \in K^c$ for each $\theta \in \Theta$ and
     \begin{align}\label{eq: tightness 3}
      \sup_{n \geq 1}\sup_{\theta \in \Theta}\left(\P[X_n^{\theta} \in K^c] + \P[X^{\theta} \in K^c] + \P[Y_n^{\theta} \in K^c] \right) < \e.
     \end{align}
     Then $(X_n^{\theta}, Y_n^{\theta}) \Longrightarrow (X^{\theta}, c(\theta))$ uniformly on $\Theta$.
 \end{enumerate}
\end{Proposition}

Since the proof is essentially the same as for the classical case with $\Theta$ being a singleton, we leave the details for the reader. Finally, let us briefly consider the case where $E = \R^d$ is equipped with the Euclidean distance. If $\mu_n^{\theta} \Lu \mu^{\theta}$, then it is clear that the corresponding characteristic functions converge uniformly on $\Theta$, i.e., 
\[
   \sup_{\vartheta \in \Theta}\left| \int_{\R^d} e^{\mathrm{i} \langle z, x\rangle} \mu_n^{\theta}(dx) - \int_{\R^d}e^{\mathrm{i}\langle z,x \rangle}\mu^{\theta}(dx)\right| \longrightarrow 0, \qquad n \to \infty.
\]
Under an additional tightness (or equivalently uniform continuity) condition, also the converse statement is true as stated below. 

\begin{Theorem}\label{thm: uniform Levy continuity theorem}
    Let $(\mu_{n}^{\theta})_{n \geq 1, \theta \in \Theta}$ and $(\mu^{\theta})_{\theta \in \Theta}$ be Borel probability measures on $\R^d$. The following properties are equivalent:
    \begin{enumerate}
        \item[(a)] For each $z \in \R^d$ the characteristic functions converge uniformly on $\Theta$, and for each $\e > 0$ there exists $\delta > 0$ such that
        \begin{align}\label{eq: muntheta continuity}
         \sup_{\theta \in \Theta}\left| 1 - \int_{\R^d} e^{\mathrm{i}\langle z, x\rangle} \mu^{\theta}(dx) \right| < \e, \qquad |z| < \delta.
        \end{align}
        \item[(b)] For each $z \in \R^d$ the characteristic functions converge uniformly on $\Theta$, and for each $\e>0$ there exists $R > 0$ such that
        \[
         \sup_{\theta \in \Theta}\mu^{\theta}(\{|x| > R\}) < \e.
        \]
    \end{enumerate}
    Both conditions imply that $\mu_n^{\theta} \Longrightarrow \mu^{\theta}$ uniformly on $\Theta$.
\end{Theorem}

Such a result should be well-known from the literature and is therefore omitted.

\subsection*{Acknowledgement}

M. F. would like to thank the LMRS for the wonderful hospitality and Financial support in 2023 during which a large part of the research was carried out. Moreover, M.F. gratefully acknowledges the Financial support received from Campus France awarded by Ambassade de France en Irlande to carry out the project “Parameter estimation in affine rough models”. The authors would like to thank the referees for many critical remarks that have led to a significantly improved version of the manuscript.

\bibliographystyle{amsplain}
\phantomsection\addcontentsline{toc}{section}{\refname}\bibliography{Bibliography}

\section{Supplementary material: Tables and figures}

\begin{figure}[H]
  \centering
  \begin{minipage}[b]{0.85\textwidth}
    \includegraphics[width=\textwidth]{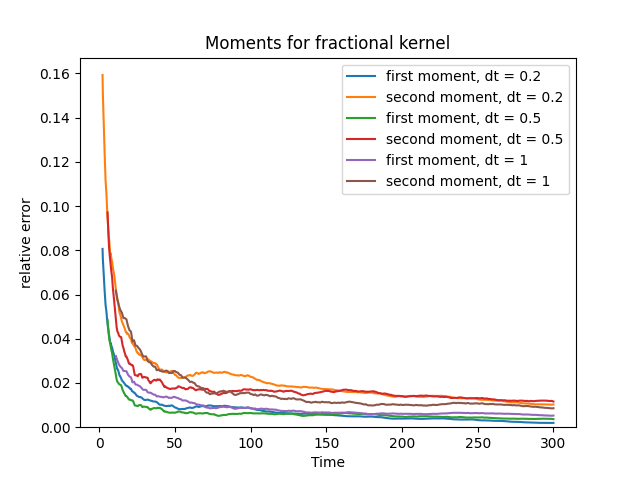}
    \captionsetup{width=1\linewidth}
    \caption{Moments with fractional kernel with $\alpha=0.75$, and $\Delta t_{k+1} \in \{0.2, 0.5, 1\}$, $N = 200$.}
    \label{fig:Moments_fractional_dt}
    \end{minipage}
\end{figure}

\begin{figure}[H]
  \centering
    \begin{minipage}[b]{0.85\textwidth}
    \includegraphics[width=\textwidth]{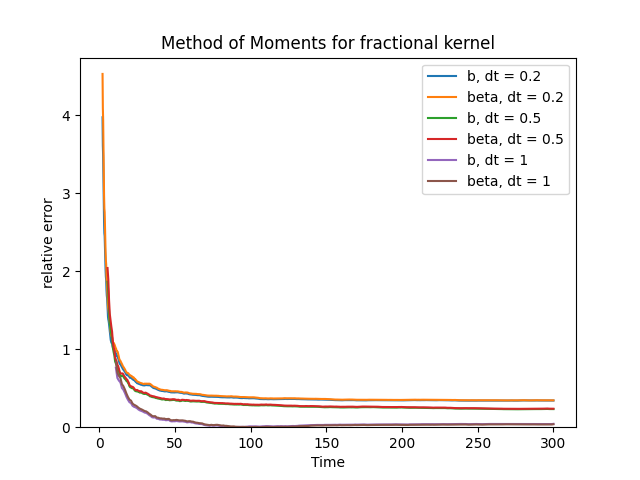}
    \captionsetup{width=1\linewidth}
    \caption{Method of Moments for the fractional kernel with $\alpha=0.75$, and $\Delta t_{k+1} \in \{0.2, 0.5, 1\}$, $N = 200$.}
    \label{fig:MoM_fractional_dt}
  \end{minipage}
\end{figure}

\begin{figure}[H]
  \centering
  \begin{minipage}[b]{0.85\textwidth}
    \includegraphics[width=\textwidth]{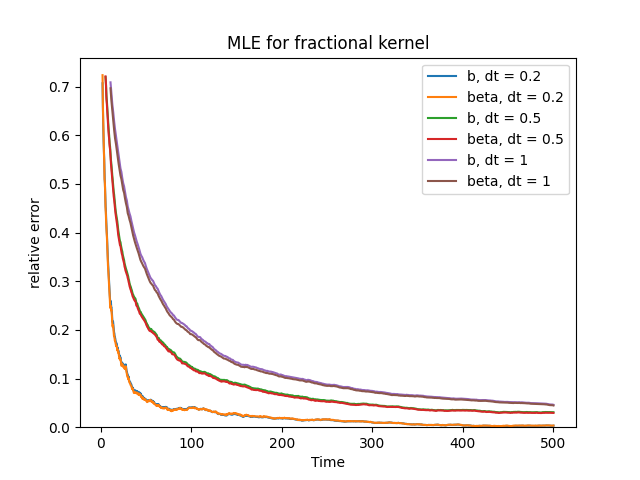}
    \captionsetup{width=1\linewidth}
    \caption{MLE with fractional kernel with $\alpha=0.75$, $\Delta t_{k+1} \in \{0.2, 0.5, 1\}$, $N = 200$}
    \label{fig:MLE fractional}
    \end{minipage}
\end{figure}

\begin{figure}[H]
  \centering
  \begin{minipage}[b]{0.85\textwidth}
    \includegraphics[width=\textwidth]{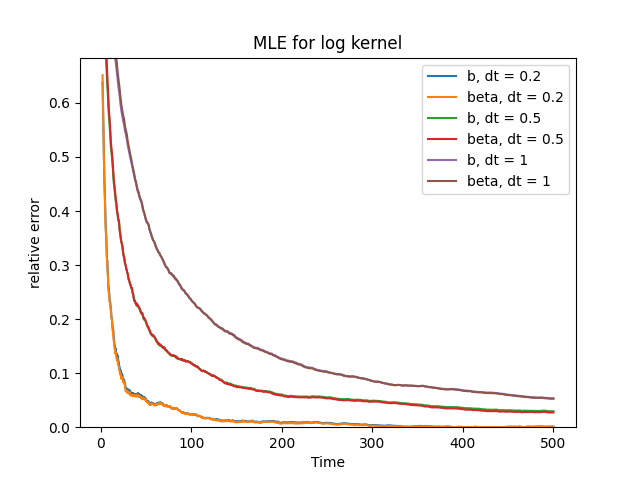}
    \captionsetup{width=1\linewidth}
    \caption{MLE with log kernel, $\Delta t_{k+1} \in \{0.2, 0.5, 1\}$, $N = 200$}
    \label{fig:MLE log}
    \end{minipage}
\end{figure}

\begin{figure}[H]
  \centering
  \begin{minipage}[b]{0.85\textwidth}
    \includegraphics[width=\textwidth]{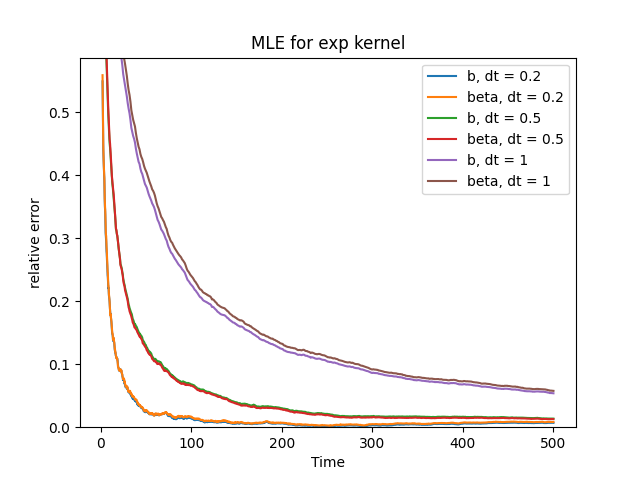}
    \captionsetup{width=1\linewidth}
    \caption{MLE with exp kernel, $\Delta t_{k+1} \in \{0.2, 0.5, 1\}$, $N = 200$}
    \label{fig:MLE exp}
  \end{minipage}
\end{figure}

\begin{figure}[H]
  \centering
  \begin{minipage}[b]{0.85\textwidth}
    \includegraphics[width=\textwidth]{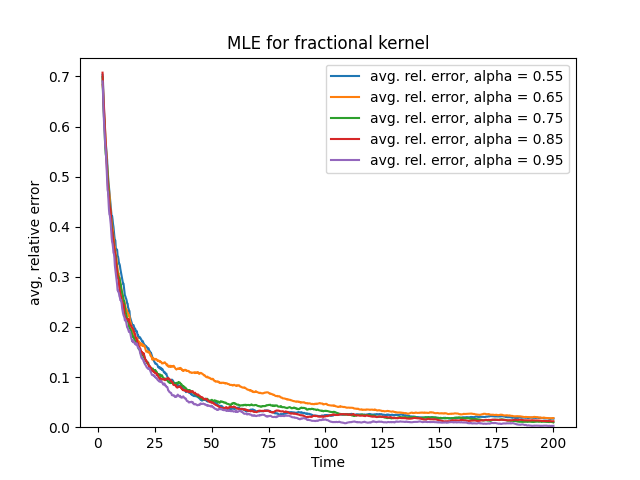}
    \captionsetup{width=1\linewidth}
    \caption{MLE with fractional kernel with $\alpha \in \{0.55, 0.65, 0.75, 0.85, 0.95\}$, $\Delta t_{k+1} =  0.2$. The case $\alpha = 0.65$ stands out due to the smaller sample size $N = 200$.}
    \label{fig:MLE fractional alpha}
    \end{minipage}
\end{figure}

\begin{figure}[H]
  \centering
  \begin{minipage}[b]{0.47\textwidth}
    \includegraphics[width=\textwidth]{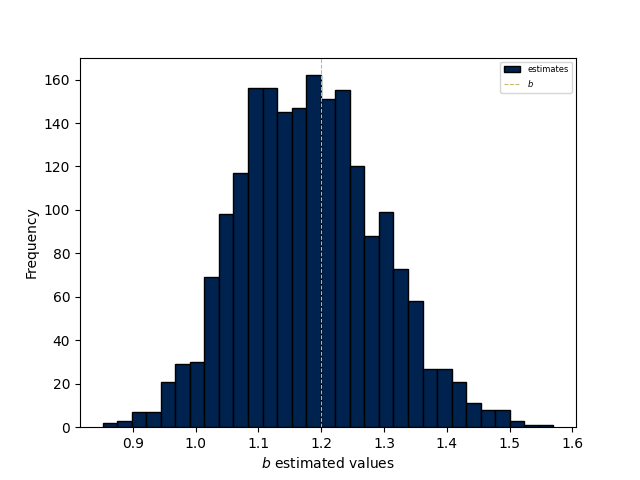}
    \captionsetup{width=1\linewidth}
    \caption{MLE with fractional kernel with $\alpha = 0.75$, $\Delta t_{k+1} = 0.2$, $N = 2000$, $T = 200$}
    \label{fig:MLE histogram b}
    \end{minipage}
  \hfill
  \begin{minipage}[b]{0.47\textwidth}
    \includegraphics[width=\textwidth]{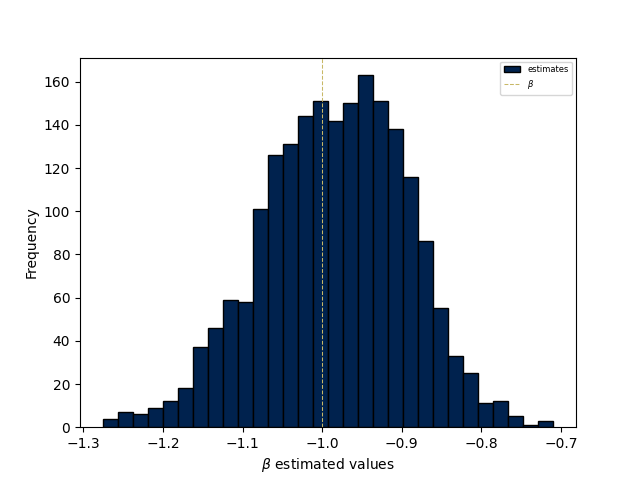}
    \captionsetup{width=1\linewidth}
    \caption{MLE with fractional kernel with $\alpha = 0.75$, $\Delta t_{k+1} = 0.2$, $N = 2000$, $T = 200$}
    \label{fig:MLE histogram beta}
  \end{minipage}
\end{figure}

\end{document}